\documentclass[11pt]{amsart} 
\usepackage{amssymb}
\usepackage{verbatim}
\usepackage{hyperref}
\usepackage{color}

\usepackage{tikz}

\begin{document}
\setcounter{tocdepth}{1}

\newtheorem{theorem}{Theorem}    
\newtheorem{proposition}[theorem]{Proposition}
\newtheorem{corollary}[theorem]{Corollary}
\newtheorem{lemma}[theorem]{Lemma}
\newtheorem{sublemma}[theorem]{Sublemma}
\newtheorem{conjecture}[theorem]{Conjecture}
\newtheorem{claim}[theorem]{Claim}
\newtheorem{fact}[theorem]{Fact}
\newtheorem{observation}[theorem]{Observation}

\newtheorem{definition}{Definition}
\newtheorem{notation}[definition]{Notation}
\newtheorem{remark}[definition]{Remark}
\newtheorem{question}[definition]{Question}
\newtheorem{questions}[definition]{Questions}
\newtheorem{hypothesis}[definition]{Hypothesis}

\newtheorem{example}[definition]{Example}
\newtheorem{problem}[definition]{Problem}
\newtheorem{exercise}[definition]{Exercise}

 \numberwithin{theorem}{section}
 \numberwithin{definition}{section}
 \numberwithin{equation}{section}

\def\repair{\medskip\hrule\hrule\medskip}

\def\bff{\mathbf f}
\def\bE{\mathbf E}
\def\bF{\mathbf F}
\def\bK{\mathbf K}
\def\bP{\mathbf P}
\def\bx{\mathbf x}
\def\bi{\mathbf i}
\def\bk{\mathbf k}
\def\bt{\mathbf t}
\def\bc{\mathbf c}
\def\ba{\mathbf a}
\def\bw{\mathbf w}
\def\bh{\mathbf h}
\def\bn{\mathbf n}
\def\bg{\mathbf g}
\def\bc{\mathbf c}
\def\bs{\mathbf s}
\def\bp{\mathbf p}
\def\by{\mathbf y}
\def\bv{\mathbf v}
\def\be{\mathbf e}
\def\bu{\mathbf u}
\def\bm{\mathbf m}
\def\bxi{{\mathbf \xi}}
\def\bR{\mathbf R}
\def\by{\mathbf y}
\def\bz{\mathbf z}
\def\bfb{\mathbf b}
\def\bPhi{{\mathbf\Phi}}

\newcommand{\norm}[1]{ \|  #1 \|}

\def\scriptl{{\mathcal L}}
\def\scriptc{{\mathcal C}}
\def\scriptd{{\mathcal D}}
\def\scrapd{{\mathcal D}}
\def\scripts{{\mathcal S}}
\def\scriptq{{\mathcal Q}}
\def\scriptt{{\mathcal T}}
\def\scriptf{{\mathcal F}}
\def\scriptm{{\mathcal M}}
\def\scripti{{\mathcal I}}
\def\scriptr{{\mathcal R}}
\def\scriptb{{\mathcal B}}
\def\scripte{{\mathcal E}}
\def\scripta{{\mathcal A}}
\def\scriptn{{\mathcal N}}
\def\scriptv{{\mathcal V}}
\def\scriptz{{\mathcal Z}}
\def\scriptj{{\mathcal J}}
\def\scriptk{{\mathcal K}}
\def\scriptg{{\mathcal G}}
\def\scripth{{\mathcal H}}

\def\bk{\mathbf k}
\def\kernel{\operatorname{kernel}}
\def\dist{\operatorname{distance}}
\def\eps{\varepsilon}

\def\reals{\mathbb R}
\def\naturals{\mathbb N}
\def\integers{\mathbb Z}
\def\rationals{\mathbb Q}
\def\one{\mathbf 1}
\def\complex{{\mathbb C}\/}

\def\lt{{L^2}}

\def\three{\mathbf 3}
\def\four{\mathbf 4}

\def\bart{\bar t}
\def\barz{\bar z}
\def\barx{{\bar x}}
\def\bary{\bar y}
\def\barz{{\bar z}}
\def\bars{\bar s}
\def\barc{\bar c}
\def\baru{\bar u}
\def\barr{\bar r}

\def\distance{\operatorname{distance}}
\def\md{{\mathcal D}}

\def\lsharp{\Lambda^\sharp}
\def\lnatural{\Lambda^\natural}
\def\sS{{\mathcal S}}
\def\barsS{\overline{\mathbb S}}

\title {A Three Term  Sublevel Set Inequality}

 \author{Michael Christ}

\address{
        Michael Christ\\
        Department of Mathematics\\
        University of California \\
        Berkeley, CA 94720-3840, USA}
\email{mchrist@berkeley.edu}

\date{February 14, 2022.}

\begin{abstract}
An upper bound is established for the Lebesgue measure
of the set on which a certain type of variable coefficient linear combination
of measurable functions is small.
\end{abstract}

\thanks{Research supported by NSF grant
DMS-1901413}

\maketitle



\tableofcontents

\part{Introduction}
\section{An inequality}\label{section:intro}

This paper establishes an inequality in differentiable combinatorics.
The inequality is intended as a tool for the analysis of integrals
of products of irregularly oscillating functions.
To formulate it, let $B\subset\reals^2$ be a closed ball of finite, positive radius.
Let $\tilde B\supset B$ be open and connected.
For each $j\in\{1,2,3\}$
let $a_j:\tilde B\to\reals$ be a $C^\omega$ scalar-valued function
that does not vanish identically.
Let $\varphi_j:\tilde B\to\reals^1$ be a nonconstant real analytic mapping. 
Write $\ba = (a_1,a_2,a_3)$, and ${\mathbf \Phi} = (\varphi_1,\varphi_2,\varphi_3)$.
To any ordered triple $\bff = (f_1,f_2,f_3)$ of Lebesgue measurable
functions $f_j:\phi_j(B)\to\reals$ and any $\eps>0$ associate the sublevel set
\begin{equation}
S(\bff,\eps) = \big\{x\in B: \big|\sum_{j=1}^3 a_j(x)(f_j\circ\varphi_j)(x)\big|<\eps\big\}.
\end{equation}
Our theme is that under natural and relatively minimal hypotheses,
if $\eps$ is small then the set $S(\bff,\eps)$ has small Lebesgue measure,
unless the functions $f_j$ are themselves small.
This theme is expressed by 
a sublevel set inequality, by which we mean an upper bound of the form
\begin{equation} \label{ineq:mainbound}
\big|S(\bff,\eps)\cap\{x: |\bff\circ\Phi(x)|\ge 1\}\big| \le C\eps^\tau
\end{equation}
for some finite, positive constants $\eps,\tau$ independent of $\bff$.
$|E|$ denotes the Lebesgue measure of a set $E$.

Upper bounds of this type have arisen in a study \cite{christquad}
of quadrilinear forms $\int_B \prod_{j=1}^4 (g_j\circ\varphi_j)$, 
for which an upper bound in terms of a product of negative order Sobolev norms
of the functions $g_j$ is sought.
Scalar-valued sums such as 
$\sum_{j=1}^3 a_j(f_j\circ\varphi_j)$, with three summands,
arise as components of the gradients of net phase functions,
for fourfold products.
The interpretation of our inequality is that if $g_j$ are rapidly oscillating, then
approximate stationarity can occur only on a set of small Lebesgue measure.
In this way, an inequality \eqref{ineq:mainbound}
provides one of the principal elements of the analysis in the companion paper \cite{christquad}.
Simpler sublevel set inequalities have been applied to
trilinear oscillatory forms in \cite{triosc}, and subsequently in \cite{CDR}
and \cite{christzhou}.

Our main hypothesis is a simple necessary condition for even a very weak form of 
the inequality \eqref{ineq:mainbound} to hold.
We write $\bff\in C^\omega(\bPhi(U))$ to indicate that $f_j\in C^\omega(\varphi_j(U))$ for each $j\in\{1,2,3\}$.

\medskip
\noindent {\bf Main hypothesis.}\
For any open set
$U\subset \tilde B$ and any  $\bff\in C^\omega(\bPhi(U))$ that satisfies
\begin{equation}  \label{exactequation}
\sum_{j=1}^3 a_j(x)\,(f_j\circ\varphi_j)(x)=0\ \ \forall\,x\in U,
\end{equation}
$f_j$ vanishes identically in $\varphi_j(U)$ for each $j\in\{1,2,3\}$.

This hypothesis is equivalently reformulated in a more 	quantitative
way, involving finite order Taylor expansions, in
Proposition~\ref{prop:equivalent_hypothesis}.

\medskip
The main result of this paper is the following theorem.
It involves an auxiliary hypothesis formulated below.
 
\begin{theorem} \label{thm:main}
Let $\ba,\bPhi$ be real analytic in a neighborhood $\tilde B$
of a closed ball $B\subset\reals^2$.
Assume that 
none of the mappings $\varphi_j$ are constant on $B$,
and that there exists no pair of distinct indices $i\ne j\in\{1,2,3\}$ for which
$\nabla\varphi_i$ and $\nabla\varphi_j$ are everywhere linearly dependent.
Assume that $(\ba,\bPhi)$ satisfies both the main and the auxiliary hypotheses. 
Then there exist $C<\infty$ and $\tau>0$, depending only on $B,\ba,\bPhi$, 
such that the sublevel set inequality \eqref{ineq:mainbound} holds
for every ordered triple $\bff$ of Lebesgue measurable functions.
\end{theorem}

Two special cases of Theorem~\ref{thm:main} were treated in \cite{triosc}:
firstly, when the coefficient functions $a_j$ are all constant, and secondly, 
when the web in $\tilde B\subset\reals^2$
defined by $\bPhi$ has identically vanishing curvature.\footnote{The 
auxiliary hypothesis was not assumed in those two cases.}
The author has not been able to apply those special analyses 
to the general case treated here.

Theorem~\ref{thm:main} treats only a special case of a more general problem.
It would be desirable to allow $B\subset\reals^D$ for arbitrary $D\ge 2$,
to allow $\varphi_j:\reals^D\to\reals^{d_j}$ for $d_j\ge 1$,
and to sum over an arbitrarily large finite set of indices $j$.
In the special case in which all mappings $\varphi_j$ are linear,
an extension to arbitrarily many summands is developed in 
Theorem~\ref{thm:linearmappings}.
Our purpose, besides establishing the specific Theorem~\ref{thm:main}, 
is to develop an argument that can potentially be extended to
these generalizations in future work.
Comments about such an extension are ventured in \S\ref{section:epilogue}.

\medskip
The auxiliary hypothesis of Theorem~\ref{thm:main} is a variant of the main hypothesis.
For each index $j\in\{1,2,3\}$,
let $W_j$ be a real analytic vector field in a neighborhood of $B$
that does not vanish identically, and satisfies $W_j(\varphi_j)\equiv 0$.

\medskip
\noindent {\bf Auxiliary hypothesis.}\ 
For each index $k\in\{1,2,3\}$,
for any nonempty open set $U\subset B$, for each $\tau\in\reals$, the only
$C^\omega$ solution $\bg  = (g_i: i\ne k)$ of
\begin{equation} \label{auxexactequation}
\sum_{i\ne k} a_i\cdot |W_k\varphi_i|^\tau \cdot (g_i\circ\varphi_i)\equiv 0 \ \text{ in $U$}
\end{equation}
is the trivial solution $g_i\equiv 0$ for each $i\ne k$.

This hypothesis is independent of the choice of vector fields $W_j$.
The special case $\tau=0$ of the auxiliary hypothesis is equivalent to what
we will call the weak auxiliary hypothesis: For any $i\ne j\in\{1,2,3\}$,
the ratio $a_i/a_j$ cannot be expressed as $(g_i\circ\varphi_i)/(g_j\circ\varphi_j)$
in any nonempty open set, with $g_i,g_j\in C^\omega$ nonvanishing.
The weak auxiliary hypothesis is also the special case of the main hypothesis
in which one of the three functions $f_j$ is required to vanish identically.

In principle, the proof produces an effective lower bound for
the exponent $\tau$ in the conclusion \eqref{ineq:mainbound} in terms of $(\ba,\bPhi)$,
and more specifically, in terms of properties of an auxiliary function
defined in terms of $(\ba,\bPhi)$ in \eqref{Fintroduced}. 
But we believe that any lower bound obtained in this way would be far from optimal,
and we have neither made an effort to make the proof more efficient in that respect,
nor developed a concrete bound for $\tau$.

\section{On the hypotheses}
The main hypothesis implies that no coefficient $a_j$ vanishes identically on $B$.
It does not imply that all $\varphi_j$ are nonconstant,
nor that $\nabla\varphi_i(x)$ cannot take the form $b(x)\nabla\varphi_j(x)$,
with $b$ scalar-valued, for some indices $i\ne j$.

We say that $(\ba,\bPhi)$ is nondegenerate in $B$ if all of the following conditions hold.
Firstly, each $a_j$ vanishes nowhere. 
Secondly, for any $i\ne j\in\{1,2,3\}$ and any $x\in B$,
$\nabla\varphi_i(x)$ and $\nabla\varphi_j(x)$ are linearly independent.
Thirdly, the curvature of the web defined by $(\ba,\bPhi)$ does not vanish identically
in $B$.
We will prove Theorem~\ref{thm:main} in the nondegenerate case,
then will explain in \S\ref{section:degenerate} how the general case follows,
essentially as a consequence of the nondegenerate case.

In the real analytic context, under the second of these nondegeneracy
conditions, the identical vanishing of the curvature of this web 
is locally equivalent to linearizability,
that is, to the existence of choices of coordinates in which all
mappings $\varphi_j$ are simultaneously linear.

The hypothesis of real analyticity of the datum $(\ba,\bPhi)$
represents a compromise between simplicity and generality;
if $\ba,\bPhi$ are merely $C^\infty$ then a more quantitative variant
of the main hypothesis is required to exclude situations in which
there exist smooth $\bff$ for which
$\sum_j a_j\cdot(f_j\circ\varphi_j)$ vanishes to infinite order at a point. 
In \S\ref{section:smoothcase} an alternative hypothesis appropriate for $C^\infty$ data
is formulated, and shown to be equivalent to the main hypothesis
for real analytic data $(\ba,\bPhi)$. 

The auxiliary hypothesis can be formulated more directly, without reference to unknowns $g_j$.
Let $(i,j,k)$ be an arbitrary permutation of $(1,2,3)$.
Any measurable solution, almost everywhere,
of \eqref{auxexactequation}
in any open set agrees almost everywhere there with a real analytic function.
Indeed, since the level curves of $\varphi_i,\varphi_j$ are transverse,
analyticity follows directly by restricting the equation
\[(g_j\circ\varphi_j) = - \Big(a_i |W_k\varphi_i|^\tau \,/\, a_j |W_k\varphi_j|^\tau\Big) (g_i\circ\varphi_i)\]
to almost any level curve of $\varphi_i$.
Moreover, after replacing $g_i,g_j$ by smooth functions
agreeing with them almost everywhere, $g_j$ is determined, up to constant multiples,
by this equation. The same holds for $g_i$. If there exists a set of positive
measure on which $g_i$ does not vanish, then it follows that $g_j$ vanishes
only on a set of measure zero; and conversely.
Thus \eqref{auxexactequation}
has a nonzero solution if and only if
$a_i |W_k\varphi_i|^\tau \,/\, a_j |W_k\varphi_j|^\tau$ can be expressed as the 
ratio of a nonvanishing $C^\omega$ function of $\varphi_i$
to a nonvanishing $C^\omega$ function of $\varphi_j$.

The author believes the auxiliary hypothesis to be superfluous. 

\part{Proof of the main theorem}
\section{Notation and preliminary remarks} \label{section:notation}

$\pi:\reals^2\times\reals^1\to\reals^2$ denotes the projection mapping
\begin{equation} \pi(x,t) = x.  \end{equation}
For any smooth vector field $V$ in a Euclidean space,
$e^{tV}x = h(t)$  is the solution of the ordinary differential equation
$h'(t) = V(h(t))$ with initial condition $h(0)=x$.

$c,C$ denote constants in $(0,\infty)$ whose values are permitted to change
from one occurrence to the next. These depend only on $\ba,\bPhi$,
on choices made in the constructions, and on constants that occur prior to their appearance
in the discussion; they do not depend on $\bff,\eps$.
For quantities $u,v\in[0,\infty)$, typically depending on $\bff,\eps$, 
$u\lesssim v$ means $u\le Cv$, while $u\gtrsim v$ means $u\ge cv$.
The three notations $O(\eps)$, $\le C\eps$, and $\lesssim\eps$ all have the same meaning.

In proving Theorem~\ref{thm:main},
it suffices to prove that there exists $\eps_0>0$ such that
the conclusion holds for every $\eps\in(0,\eps_0]$.
We will implicitly assume throughout the discussion that $\eps$ is sufficiently small.

Since all of our hypotheses are invariant under permutation of the indices $j\in\{1,2,3\}$,
an equivalent formulation is that if $|f_3(y)|\ge 1$ for every $y\in\varphi_3(B)$
then $|S(\bff,\eps)|\le C\eps^\tau$.

\medskip 
It suffices to prove Theorem~\ref{thm:main} under the additional hypothesis
that $\norm{f_j}_{L^\infty}\le 2$ for each index $j$.
Indeed, to recover the general case from this base case,
assume that $|f_3(y)|\ge 1$ for every $y\in \varphi_3(B)$
and that $|f_j|<\infty$ almost everywhere for all three indices.
For $n\ge 0$ and $j\in\{1,2,3\}$ define
\[ S_{n,j} = \{x\in S(\bff,\eps): |f_i\circ\varphi_i(x)|< 2^{n+1}
\text{ for every $i\in\{1,2,3\}$ and } |f_j\circ\varphi_j(x)|\ge 2^n\}.\]
Then $S(\bff,\eps) = \bigcup_{n=0}^\infty\bigcup_{j=1}^3 S_{n,j}$.
Applying the base case 
to $S_{n,j} \subset S(2^{-n}\bff,2^{-n}\eps)$  gives
$|S_{n,j}|\le C (2^{-n}\eps)^\tau$. Summing over $n,j$ completes the proof.

\medskip
We assume throughout the proof of Theorem~\ref{thm:main}
that the curvature of the web in $\tilde B$ defined by $\bPhi$ does not vanish identically.
This curvature vanishes identically 
in a connected open set $U\subset\tilde B$,
if and only if there exist real analytic coordinates for $U$
and associated real analytic coordinates for $\varphi_j(U)$
in which the three mappings $\varphi_j$ are simultaneously linear.

\medskip
The analysis is structured in three successive steps,
at scales $\eps^1$, $\eps^\gamma$ for an arbitrary choice of
$\gamma\in(\tfrac12,1)$, and $\eps^0$, respectively,
with information established at each of the smaller scales exploited at the next larger scale.
These steps are developed
in \S\ref{section:microscale}, in \S\ref{section:mesoscale}, 
and in \S\ref{section:macroscale}, respectively.
They are combined in \S\ref{section:conclusion} to complete the proof in the nondegenerate case.
The general case of Theorem~\ref{thm:main} is discussed in \S\ref{section:degenerate}.
Preliminary results used in the analysis are developed in earlier sections of the paper.

\section{Some properties of smooth functions} \label{section:lemmas}

\begin{lemma} \label{lemma:analyticsublevel}
Let $K\subset\reals^d$ be a connected compact set.
Let $\Omega\subset\reals^D$ be compact.
Let $H:K\times\Omega\to\reals$ be real analytic in a neighborhood of $K\times\Omega$.
Write $H_\omega(x) = H(x,\omega)$.
Suppose that there exists no $\omega\in\Omega$ for which $H_\omega$ vanishes
identically on $K$.
Then there exist $\tau>0$ and $C<\infty$ such that for every $\eps>0$,
\begin{equation}
\big|\big\{x\in K: |H_\omega(x)|\le\eps\big\}\big| \le C\eps^\tau.
\end{equation}
\end{lemma}

\begin{lemma} \label{lemma:quantsublevel}
Let $D,N\in\naturals$. Let $B$ be a closed ball in $\reals^D$ of positive, finite radius.
There exist $c,C\in(0,\infty)$ with the following property.

Let $F$ be a real-valued $C^{N+1}$ function defined in a neighborhood of $B$.
Suppose that $\norm{F}_{C^{N+1}}\le 1$.
Let $\delta>0$. Suppose that for each $x\in B$,
\begin{equation} \sum_{0\le |\alpha|\le N} |\partial^\alpha F(x)| \ge\delta.  \end{equation}
Then for every $\eta>0$,
	\begin{equation} \big|\big\{ x\in B: |F(x)|\le\eta\big\}\big|
	\le C \eta^c\delta^{-C}. \end{equation}
\end{lemma}

Proofs of Lemma~\ref{lemma:analyticsublevel} 
and of Lemma~\ref{lemma:quantsublevel} are sketched in \cite{triosc}. 

\begin{lemma} \label{lemma:steinstreetish}
Let $(x,y)\mapsto f(x,y)\in\reals$ be a $C^\omega$ function
in a connected open set $\Omega$ in some Euclidean space. Let $K\times K'\subset\Omega$ be compact.
There exist $N\in\naturals$ and $C\in(0,\infty)$
such that for every $y\in K'$,
\begin{equation} \label{ineq:bourgainsteinstreet}
\sup_{x\in K} \sum_{0\le |\alpha|\le N} |\partial^\alpha_x f(x,y)|
\le C\inf_{x\in K} \sum_{0\le |\alpha|\le N} |\partial^\alpha_x f(x,y)|.
\end{equation}
\end{lemma}

\begin{proof}
Let $\scriptk\subset\Omega$ be a connected compact neighborhood of $K\times K'$.
Consider the function $g(x,y,t) = f(x+t,y)$ in $\scriptk\times B$,
where $B$ is a closed ball of sufficiently small positive radius that $(x+t,y)\in\Omega$
whenever $(x,y)\in\scriptk$ and $t\in B$.

Expand $g$ in power series with respect to $t$:
\begin{equation}
g(x,y,t) = \sum_\alpha t^\alpha \partial_x^\alpha f(x,y)
\end{equation}
with $\partial_x^\alpha = \partial^\alpha/\partial x^\alpha$.
Applying Theorem~9.1 of Stein and Street \cite{steinstreet} to $g$,
we conclude that there exist $N$ and $C^\omega$ functions $h_\alpha$ satisfying
\begin{equation} \label{fxty1}
f(x+t,y) \equiv \sum_{|\alpha|\le N} h_\alpha(x,y,t)\,t^\alpha\, \partial_x^\alpha f(x,y) 
\end{equation}
in $\scriptk\times B$.
By applying $\partial^\beta/\partial t^\beta$ to both sides of \eqref{fxty1}
for each multi-index $\beta$ satisfying $|\beta| = N+1$,
it follows that
\begin{equation} \label{fxty2}
 \sum_{|\beta| = N+1} |\partial^\beta_x f(x,y)|
\le C \sum_{|\alpha|\le N} |\partial^\alpha_x f(x,y)| \ \forall (x,y)\in\scriptk.
\end{equation}

Consider the function
\begin{equation} u(x,y)  =  \sum_{|\alpha|\le N} (\partial^\alpha_x f(x,y))^2.  \end{equation}
It follows from Cauchy-Schwarz and \eqref{fxty2} that there exists $C'<\infty$ such that
\begin{equation}
	|\nabla_x u(x,y)| \le C' u(x,y)\ \forall\,(x,y)\in\scriptk.
\end{equation}
The conclusion \eqref{ineq:bourgainsteinstreet}
follows from this differential inequality, using the path connectivity of 
a neighborhood of $\scriptk$.
\end{proof}

\begin{lemma} \label{lemma:sublevel_notSigma}
Let $U,U'$ be connected open subsets of Euclidean spaces,
and let $K,K'$ be compact subsets of $U,U'$, respectively.
Let $(x,y)\mapsto F(x,y)\in\reals$ be a $C^\omega$ function in $U\times U'$.
Define
\begin{equation}
	\Sigma = \{y\in U': F(x,y)=0\ \forall\,x\in U\}.
\end{equation}
Then $\Sigma$ is a real analytic subvariety of positive codimension.
Moreover,
there exist $C<\infty$ and $\tau>0$ such that for every $\eps>0$ and every $y\in K'\setminus\Sigma$,
\begin{equation} \label{ineq:quant_Comega_sublevel}
\big|\big\{
	x\in K: |F(x,y)|<\eps
\big\}\big|
\le C\dist(y,\Sigma)^{-C}\eps^\tau.
\end{equation}
\end{lemma}

Lemma~\ref{lemma:sublevel_notSigma}
is proved by combining the preceding lemmas with \L{}ojasiewicz's theorem. \qed

\section{Preliminaries related to transversality and curvature} \label{section:morelemmas}

For each $k\in\{1,2,3\}$
let $W_k$ be a $C^\omega$ nowhere vanishing vector
field in $\tilde B$ that satisfies $W_k(\varphi_k)\equiv 0$.
For each $j\in\{1,2\}$
define $\psi_j:B\times\reals\to \reals^2$ by
\begin{equation}
\psi_j(x,t) = (\varphi_j(x),tW_3\varphi_j(x)).
\end{equation}
The linear independence of $\nabla\varphi_j,\nabla\varphi_3$ at each $x\in B$
guarantees that $W_k\varphi_i$ vanishes nowhere whenever $k\ne i$.
Therefore $\psi_j$ are submersions. 
By replacing $\varphi_j$ by $-\varphi_j$ for $j=1$ and/or for $j=2$
as necessary, we may assume without loss of generality that $W_3(\varphi_j)>0$.
We make this assumption throughout the discussion.

For $j\in\{1,2\}$, let $V_j$ be nowhere vanishing $C^\omega$ vector fields 
in $\tilde B\times(\reals\setminus\{0\})$  that satisfy $V_j(\psi_j)\equiv 0$.

\begin{lemma} \label{lemma:bracket}
Let $U\subset\reals^2$ be a connected open set. Then 
$V_1,V_2$ satisfy the bracket condition in $U\times (\reals\setminus\{0\})$
if and only if the curvature of the web defined by $\bPhi$
does not vanish identically in $U$.
\end{lemma}

\begin{proof}
Choose coordinates  in which
\begin{equation}
\varphi_j(x_1,x_2)\equiv x_j \ \text{ for $j=1,2$.}
\end{equation}
Then writing $\{1,2\} = \{i,j\}$, the vector fields
\begin{equation}
V_j =  \partial_{x_i} - (\partial_{x_i}\varphi_3)^{-1} \partial^2_{x_i x_i}\varphi_3(x) t\partial_t
\end{equation}
satisfy $V_j(\psi_j)\equiv 0$. Their Lie bracket is
\begin{equation} [V_1,V_2] = \theta(x)\, t\partial_t \end{equation}
with
\begin{equation} \theta(x) = 
\partial_{x_1}
\Big( \frac{\partial^2_{x_2 x_2}\varphi_3} {\partial_{x_2}\varphi_3} \Big)
-\partial_{x_2}
\Big( \frac{\partial^2_{x_1 x_1}\varphi_3} {\partial_{x_1}\varphi_3} \Big).
\end{equation}
The vector fields $(\partial_{x_i}\varphi_3)^{-1} \partial^2_{x_i x_i}\varphi_3(x) t\partial_t$
commute with $\theta(x)\,t\partial_t$.
Therefore up to nonzero constant factors, arbitrary higher-order Lie brackets of $V_1,V_2$ take the form
$\partial^\alpha\theta(x_1,x_2)\,\partial_t$ for arbitrary $\alpha$.
Thus for any $\barx\in\reals^2$ and any $t\ne 0$, the pair $V_1,V_2$
satisfies the bracket condition at $(\barx,t)$ if and only if $\theta$ does
not vanish to infinite order at $\barx$. Since $\theta$ is real analytic, it follows that the
bracket condition holds at $(\barx,t)$ if and only if it holds at 
every point of $U\times (\reals\setminus\{0\})$,
and that this is equivalent to $\theta$ not vanishing identically in $U$.

Both the bracket condition, and identical vanishing of the curvature of the web,
are invariant under diffeomorphism. 
The bracket condition is also invariant under multiplication of $V_j$
by nowhere vanishing smooth functions.
If the curvature of the web vanishes identically
in a neighborhood of some point $\barx$, then it is possible 
to change coordinates so that $\varphi_j(x_1,x_2)\equiv x_j$
for $j=1,2$ and so that $\varphi_3$ is linear. In that situation, $\nabla^2\varphi_3\equiv 0$,
so $\theta$ vanishes identically.

Suppose on the other hand that $\theta\equiv 0$
	in an open set. Working in coordinates in which $\varphi_j(x_1,x_2)\equiv x_j$
	for $j=1,2$ and with the notations $\kappa= \varphi_3$, $\kappa_i= \partial_{x_i}\kappa$,
and $\kappa_{i,j} = \partial^2_{x_i x_j}\kappa$, a calculation gives
\begin{equation}
\theta(x) = (-\kappa_2\partial_1 + \kappa_1\partial_2)(\kappa_2^{-1}\kappa_1^{-1}\kappa_{1,2});
\end{equation}
recall that $\kappa_i = W_j(\varphi_3)$ vanishes nowhere by hypothesis when $\{1,2\} = \{i,j\}$.
Any function annihilated by the vector field
$-\kappa_2\partial_1 + \kappa_1\partial_2$ can be expressed in the form $h\circ\kappa$
for some function $h$.  That is,
\begin{equation} \frac{-\kappa_{1,2}}{\kappa_1\kappa_2} = h\circ\kappa.  \end{equation}
Defining $g$ to satisfy $g''/g' = h$, this relation becomes
\begin{equation} \partial^2_{1,2}(g\circ\kappa)\equiv 0.  \end{equation}
Therefore $(g\circ\kappa)(x_1,x_2)$ can locally be expressed in the form $u(x_1)+v(x_2)$.
Thus with $x = (x_1,x_2)$,
\begin{equation}
\varphi_3(x) = \kappa(x) = g^{-1}(u(x_1)+v(x_2)).
\end{equation}
Thus upon changing variables in $\reals^2$ via the mapping
$x = (x_1,x_2) \leftrightarrow (u(x_1),v(x_2))$
and changing variables in the codomain $\varphi_3(B)$ so that $\varphi_3$ is replaced by $g\circ\varphi_3$,
all three mappings $\varphi_i$ become linear.  Thus the web has vanishing curvature.
\end{proof}

\begin{lemma} \label{lemma:det(B)}
Let $a_j,\varphi_j\in C^\omega$ for $j=2,3$.
Suppose that neither $a_2,a_3$ vanishes identically,
and that $\alpha = a_2/a_3$ cannot be expressed 
in any nonempty open subset of $B$ in the form
$(h_2\circ\varphi_2)/(h_3\circ\varphi_3)$,
with $h_j\in C^\omega$.  
Suppose that $\nabla\varphi_2,\nabla\varphi_3$ are pointwise linearly independent 
at each point of $B$.
Let \begin{equation} \Lambda = \{(x,x')\in B\times B: \varphi_2(x)=\varphi_2(x')\}.\end{equation}
Define $B(x,x')$ to be the matrix
\begin{equation}
B(x,x') = \begin{pmatrix}
\alpha(x)\,W_3\varphi_2(x) & W_3\alpha(x)
\\
\alpha(x')\,W_3\varphi_2(x') & W_3\alpha(x')
\end{pmatrix}.
\end{equation}
Then the determinant $\det B(x,x')$
does not vanish identically in any nonempty relatively open subset of $\Lambda$.
\end{lemma}

The hypothesis is the special case $\tau=1$ of the auxiliary hypothesis,
for a particular permutation of the indices $j\in\{1,2,3\}$.

\begin{proof}
The hypothesis that $\nabla\varphi_2,\nabla\varphi_3$
are pointwise linearly independent is equivalent to $W_3\varphi_2$ vanishing nowhere. 
If $\det B$ vanishes identically in some connected relatively open subset of 
$\Lambda = \{(x,x'): \varphi_2(x)=\varphi_2(x')\}$
then the ratio $\frac{W_3\alpha(x)}{\alpha(x)\,W_3\varphi_2(x)}$ is locally  
a function of the form $h\circ \varphi_2(x)$ with $h\in C^\omega$.
Let $H$ be an antiderivative of $-h$.
Assuming without loss of generality that $\alpha>0$, this relation can be written as
\[ W_3\log(\alpha) = W_3(\varphi_2)\cdot (h\circ\varphi_2) = -W_3(H\circ\varphi_2).\]
Thus $W_3(\log(\alpha)+(H\circ\varphi_2))\equiv 0$ locally,
whence $\log(\alpha) + H\circ\varphi_2$ is a function of the form $G\circ\varphi_3$,
contradicting the hypothesis.
\end{proof}

Recall that $\pi:\reals^2\times\reals^1\to\reals^2$
denotes the mapping $(x,t)\mapsto x$.
The auxiliary equation
\begin{equation}  \label{gexactsolution1}
\sum_{j=1}^2 (a_j\circ\pi)\cdot (g_j\circ\psi_j) = 0, \end{equation} 
and the equations obtained by permuting the indices in $\{1,2,3\}$, 
will arise in the analysis below.  The next lemma describes the structure 
of such nontrivial solutions, if any exist.
Proposition~\ref{prop:resurrected}, formulated below,
describes the structure of approximate solutions when no exact nontrivial solutions exist.

\begin{lemma} \label{lemma:exactsolns}
Let $\varphi_j:\tilde B\to\reals$ be $C^\omega$ submersions for
$j\in\{1,2\}$.
Let $a_j,\beta_j:\tilde B\to\reals$ be $C^\omega$.
Let $\psi_j(x,t) = (\varphi_j(x),t\beta_j(x))$.
Suppose that $\psi_j:\tilde B\times(0,\infty)\to\reals^2$ are submersions.
Let $V_j$ be nowhere vanishing $C^\omega$ vector fields in $B\times(0,\infty)$ that
satisfy $V_j(\psi_j)\equiv 0$.
Suppose that $V_1,V_2$ satisfy the bracket condition at each point of $B\times(0,\infty)$.

Let $\bg = (g_1,g_2)$ 
be a measurable $\reals^2$--valued function that satisfies
\eqref{gexactsolution1} 
almost everywhere
in some nonempty connected open subset of $B\times(0,\infty)$.
Then $\bg\in C^\omega$,
$\bg$ is unique up to constant scalar multiples, and there exist $\tau\in\reals$
and $h_j\in C^\omega$
such that $\bg$ takes the form
\begin{equation} \label{powerlaw1} g_j(y,t) = h_j(y)\,t^\tau\ \  \forall\,j\in\{1,2\}. \end{equation}
\end{lemma}

\begin{proof}
By induction on nonnegative integers $n$,
define mappings $\Theta_{n,z}$,
whose domains are small neighborhoods of $0$ in $\reals^n$
and whose codomains are $\reals^3$, as follows.
$\Theta_{0,z}=z$, and
\begin{equation} \Theta_{n,z}(t_1,\dots,t_n)
= 
\begin{cases}
e^{t_n V_1}\Theta_{n-1,z}(t_1,\dots,t_{n-1})
\ &\text{ if $n$ is odd}
\\
e^{t_n V_2}\Theta_{n-1,z}(t_1,\dots,t_{n-1})
\ &\text{ if $n$ is even.}
\end{cases}
\end{equation}

It follows from induction on $n$ and the relation \eqref{gexactsolution1} that if
$n$ is odd, then with the notation $\bt=(t_1,\dots,t_n)$, 
\begin{equation}
(g_2\circ\psi_2)(\Theta_{n,z}(\bt)) = \alpha b_n(\bt,z) 
\end{equation}
where $\alpha = g_1(z)$ and
\begin{multline}
b_n(\bt,z) = - 
\frac {a_1(\Theta_{n,z}(t_1,\dots,t_n))} {a_2(\Theta_{n,z}(t_1,\dots,t_n))}
\cdot
\frac {a_2(\Theta_{n-1,z}(t_1,\dots,t_{n-1}))} {a_1(\Theta_{n-1,z}(t_1,\dots,t_{n-1}))}
\cdot
\\
\cdot
\frac {a_1(\Theta_{n-2,z}(t_1,\dots,t_{n-2}))} {a_2(\Theta_{n-2,z}(t_1,\dots,t_{n-2}))}
\cdots
\frac {a_1(\Theta_{1,z}(t_1))} {a_2(\Theta_{1,z}(t_1))}.
\end{multline}
If $n$ is even, then
\begin{equation}
(g_1\circ\psi_1)(\Theta_{n,z}(\bt)) = \alpha b_n(\bt,z) 
\end{equation}
where  $\alpha$ is again equal to $g_1(z)$ and
\begin{equation} \label{bdefinition}
b_n(\bt,z) = 
\frac {a_2(\Theta_{n,z}(t_1,\dots,t_n))} {a_1(\Theta_{n,z}(t_1,\dots,t_n))}
\cdot
\frac {a_1(\Theta_{n-1,z}(t_1,\dots,t_{n-1}))} {a_2(\Theta_{n-1,z}(t_1,\dots,t_{n-1}))}
\cdots
\cdot
\frac {a_1(\Theta_{1,z}(t_1))} {a_2(\Theta_{1,z}(t_1))}.
\end{equation}
Each $b_n$ is a real analytic function of $(\bt,z)$. 

Suppose that $g_j$ are measurable and that $\bg$ satisfies 
\eqref{gexactsolution1} almost everywhere in some connected open set $U\subset\reals^3$, 
but that it is not the case that $\bg$ vanishes almost everywhere in $U$.

Since $V_1,V_2$ are real analytic and satisfy the bracket condition, for any $z\in U$
the set of all $\bt = (t_1,t_2,t_3)$ for which the Jacobian determinant
	$J_z(\bt)$ of the function $\bt\mapsto \Theta_{3,z}(\bt)$ 
satisfies $J_z(\bt)=0$ is a proper $C^\omega$ variety.
Therefore there exist arbitrarily small $\bt_0=(t_{1,0},t_{2,0},t_{3,0})$ such that
$J_{z}(\bt_0)\ne 0$. 

Now fix any $z_0\in U$, and let $\bt_0$ be as in the preceding paragraph.
The Jacobian determinant, with respect to 
$\bt$ at $\bt = \bt_0$, of 
\[\bt\mapsto e^{-t_{0,1}V_1}e^{-t_{0,2}V_2}e^{-t_{0,3}V_1}e^{t_3V_1}e^{t_2V_2}e^{t_1V_1}z_0,\]
is likewise nonzero. This mapping sends $\bt_0$ to $z_0$. 
Nonvanishing of the Jacobian determinant at $\bt_0$ is stable under small perturbation
of $z_0$. 

By the same reasoning as in the discussion above of
$g_2\circ\psi_2\circ \Theta_{3,z}$, there exists $r>0$ such that for almost every $z$
satisfying $|z-z_0|<r$, $g_2$ agrees almost everywhere in $\psi_2(B(z,r))$ 
with a $C^\omega$ function.
The corresponding assertion holds for $g_1$.
Thus $(g_1,g_2)$ agrees, almost everywhere in $U$, with a $C^\omega$  $\reals^2$-valued function.
Moreover, this reasoning proves that 
$(g_1,g_2)$ is uniquely determined up to multiplication by the constant scalar $\alpha = g_1(z_0)$.

For any $r\in\reals\setminus\{0\}$,
$\tilde g_j(y,t) = g_j(y,rt)$ also defines a solution of \eqref{gexactsolution1}.
In combination with the uniqueness up to scalar multiplication, this implies that
any smooth solution takes the form \eqref{powerlaw1}
for $y$ in the domain of $f_j$ and $t\in\reals^+$, for some exponent $\tau\in\reals$
and some $h_j\in C^\omega$,
with a corresponding representation for $t\in\reals^-$.
Moreover, if a smooth nontrivial solution does exist, then the exponent $\tau$ is unique.
\end{proof}

\section{Differences} \label{section:framework}

In this section,
we introduce a framework that will be used in the microscale and mesoscale analyses to follow.
In the nondegenerate case of Theorem~\ref{thm:main},
we may assume with no loss of generality that $a_3\equiv 1$,
by replacing each coefficient $a_j$ in the definition of $S(\bff,\eps)$ by $a_j/a_3$.
For each $(x,s)$ satisfying 
$(x,e^{sW_3}x)\in S(\bff,\eps)\times S(\bff,\eps)$,
we can subtract one inequality from another to cancel terms involving $f_3$, leaving
\begin{equation} \label{differenced1}
 \sum_{j=1}^2 a_j(x)\Big[(f_j\circ\varphi_j)(e^{sW_3}x)
-(f_j\circ\varphi_j)(x)\Big] = O(\eps)
\end{equation}
since
\begin{equation}
a_j(e^{sW_3}x)-a_j(x) = O(|s|) = O(\eps)
\end{equation}
and it is assumed that
$f_j = O(1)$.
Write  \eqref{differenced1} as
\begin{equation} \label{differenced2}
\sum_{j=1}^2 a_j(x)(\tilde g_j\circ\tilde\psi_j)(x,s) = O(\eps)
\end{equation}
with 
\begin{equation}
\left\{
\begin{gathered}
\tilde\psi_j(x,s) = (\varphi_j(e^{sW_3}x),\,\varphi_j(x))\in\reals^2
\\
\tilde g_j(y,z) = f_j(y)-f_j(z).
\end{gathered} \right. \end{equation}
We will temporarily ignore the structural relation $\tilde g_j(y,z) = f_j(y)-f_j(z)$,
will analyze solutions of \eqref{differenced2}, for general functions $\tilde g_j$ of two variables,
and then will combine the results of that analysis with the structural information
to recover information about $f_j$.
Permuting the indices will produce corresponding information concerning $f_3$.

It will be more useful to work in alternative coordinates 
adapted to the exploitation of asymptotic relations as $\eps\to 0^+$.
Change variables in domain and range, defining
$\psi_j^\eps:B\times[-C,C]\to\reals^2$ 
and functions $g_j$ by
\begin{equation} \label{psiandgjdefns}
\left\{
\begin{aligned}
\psi_j^\eps(x,t) &= (\varphi_j(x),\eps^{-1} \big[\varphi_j(e^{\eps t W_3}x)-\varphi_j(x)\big] )
\\
g_j(y,t) &= f_j(y+\eps t)-f_j(y).
\end{aligned} \right. \end{equation}
With these definitions,
\begin{equation}
(g_j\circ\psi_j^\eps)(x,t) = (f_j\circ\varphi_j)(e^{\eps t W_3}x)-(f_j\circ\varphi_j)(x)
\end{equation}
and therefore
\begin{equation} \label{gsublevel}
\sum_{j=1}^2 a_j(x)\,(g_j\circ\psi^\eps_j)(x,t) = O(\eps)
\end{equation}
for all $(x,t)$ satisfying $(x,e^{\eps tW_3}x)\in S(\bff,\eps)\times S(\bff,\eps)$.

By the Cauchy-Schwarz inequality, the set of $(x,t)$ 
satisfying $(x,e^{\eps tW_3}x)\in S(\bff,\eps)^2$
has Lebesgue measure $\gtrsim |S(\bff,\eps)|^2$.
More generally, for any measurable subset $S\subset S(\bff,\eps)$,
the set $S^\sharp$ defined by
\begin{equation} S^\sharp = \{(x,t): (x,\, e^{\eps tW_3}x) \in S\times S\}  \end{equation}
satisfies $|S^\sharp|\gtrsim |S|^2$.

As $\eps\to 0^+$, $\psi_j^\eps$ converges to the limiting mapping
\begin{equation}
\psi_j(x,t)  = \psi^0_j(x,t) = (\varphi_j(x),tW_3\varphi_j(x))
\end{equation}
studied in \S\ref{section:morelemmas}, with
\begin{equation} \psi_j^\eps(x,t) = \psi_j(x,t) +O(\eps) 
\ \text{in the $C^N$ topology, for every $N$.}
\end{equation}
The function $(x,t,\eps)\mapsto \psi_j^\eps(x,t)$ is well-defined
for small $\eps<0$, also, and $(x,t,\eps)\mapsto \psi_j^\eps(x,t)$ 
is real analytic in a neighborhood of $\eps=0$.

Let $V^\eps_j$ be nowhere vanishing real analytic vector fields 
in a neighborhood of $B\times(0,1]$ that depend real analytically
on $\eps$ in a neighborhood of $\eps=0$,
that agree with those already defined for $\eps=0$,
and that satisfy $V_j^\eps(\psi_j^\eps)\equiv 0$.
Then $V^\eps_1,V^\eps_2$ also satisfy the bracket condition,
for every sufficiently small $\eps$,
since that condition is stable under small perturbations in the $C^\infty$ topology.

By induction on nonnegative integers $n$,
define mappings $\Theta_{n,z}^\eps$ for each $z\in B$,
whose domains are small neighborhoods of $0$ in $\reals^n$
and whose codomains are $\reals^3$, 
by replacing $\psi_j,V_j$ by $\psi_j^\eps,V_j^\eps$ respectively
in the construction of $\Theta_{n,z}$ in \S\ref{section:morelemmas}. Thus
$\Theta^\eps_{0,z}=z$, and
\begin{equation} \Theta^{\eps}_{0,z}(t_1,\dots,t_n)
= 
\begin{cases}
e^{t_n V_1^\eps}\Theta_{n-1,z}^\eps(t_1,\dots,t_{n-1})
\ &\text{ if $n$ is odd}
\\
e^{t_n V_2^\eps}\Theta_{n-1,z}^\eps(t_1,\dots,t_{n-1})
\ &\text{ if $n$ is even.}
\end{cases}
\end{equation}
These satisfy
\begin{equation} \Theta_{n,z}^\eps=\Theta^0_{n,z}+O(\eps)
\ \text{ in the $C^N$ topology for every $N$.}
\end{equation}

It follows from induction on $n$ 
and the relation \eqref{gsublevel} that if
$n$ is odd, and if
$\Theta^\eps_{k,z}(t_1,\dots,t_k)\in S^\sharp$
for every $0\le k\le n$, then for $\bt=(t_1,\dots,t_n)$, 
\begin{equation}
(g_2\circ\psi^\eps_2)(\Theta^\eps_{n,z}(\bt)) = \alpha b_n^\eps(\bt,z) +O(\eps)
\end{equation}
where $\alpha = g_1(z)$ and
\begin{multline} \label{bndefn}
b_n^\eps(\bt,z) = - 
\frac {a_1(\Theta^\eps_{n,z}(t_1,\dots,t_n))} {a_2(\Theta^\eps_{n,z}(t_1,\dots,t_n))}
\cdot
\frac {a_2(\Theta_{n-1,z}^\eps (t_1,\dots,t_{n-1}))} {a_1(\Theta_{n-1,z}^\eps (t_1,\dots,t_{n-1}))}
\cdot
\\
\cdot
\frac {a_1(\Theta_{n-2,z}^\eps(t_1,\dots,t_{n-2}))} {a_2(\Theta_{n-2,z}^\eps (t_1,\dots,t_{n-2}))}
\cdots
\frac {a_1(\Theta_{1,z}^\eps(t_1))} {a_2(\Theta_{1,z}^\eps (t_1))}.
\end{multline}
Likewise, if $n$ is even  and 
$\Theta_{k,z}^\eps (t_1,\dots,t_k)\in S^\sharp$
for every $0\le k\le n$, then 
\begin{equation}
(g_1\circ\psi^\eps_1)(\Theta^\eps_{n,z}(\bt)) = \alpha b_n^\eps(\bt,z) +O(\eps)
\end{equation}
where $\alpha = g_1(z)$ and
\begin{equation} \label{bdefinition2}
b_n^\eps(\bt,z) = 
\frac {a_2(\Theta^\eps_{n,z}(t_1,\dots,t_n))} {a_1(\Theta^\eps_{n,z}(t_1,\dots,t_n))}
\cdot
\frac {a_1(\Theta^{\eps}_{n-1,z}(t_1,\dots,t_{n-1}))} {a_2(\Theta^{\eps}_{n-1,z}(t_1,\dots,t_{n-1}))}
\cdots
\cdot
\frac {a_1(\Theta^\eps_{1,z}(t_1))} {a_2(\Theta^\eps_{1,z}(t_1))}.
\end{equation}
$b_n^\eps(\bt,z)$ is defined  by \eqref{bdefinition2} and/or \eqref{bndefn}
for all $\eps$, positive or nonpositive, in some neighborhood of $0$.
Each $b_n$ is a real analytic function of $(\bt,z,\eps)$ in a neighborhood of $\eps=0$, and
\begin{equation} b_n^\eps(\bt,z) = b_n^0(\bt,z) + O(\eps) \end{equation} in $C^N$ norm for every $N$.


\begin{lemma} \label{lemma:prelimmicrostep}
Let $\bff,\eps,\ba,\bPhi$ satisfy the hypotheses of Proposition~\ref{prop:microscale}.
Assume that there exists a $C^\omega$ solution $\bg^* = (g_1^*,g_2^*)$
of $\sum_{j=1}^2 (a_j\circ\pi)\,(g_j^*\circ\psi_j^0)\equiv 0$
in some nonempty open subset of $B\times(0,\infty)$ 
that does not vanish identically.
Let $S\subset B\times(0,1]$ be measurable and satisfy
$(x,t)\in S\Rightarrow (x,e^{\eps tW_3}x)\in S(\bff,\eps)^2$. 
Let $\psi_j^\eps,g_j$ be defined as in \eqref{psiandgjdefns}.
There exist a scalar $\alpha\in\reals$
and a measurable set $E\subset S$ of Lebesgue measure $|E|\gtrsim|S|^C$ such that
\begin{equation}(g_j\circ\psi_j^\eps)(x,t) = \alpha\cdot (g_j^*\circ\psi_j^0)(x,t) + O(\eps)
\ \forall\,(x,t)\in E.\end{equation}
\end{lemma}

\begin{proof}[Proof of Lemma~\ref{lemma:prelimmicrostep}]
This is a combination of facts established above.  Firstly, it was shown in the proof 
of Lemma~\ref{lemma:exactsolns} that up to a nonzero constant factor,
\[b_n^0(\bt,z) = (g_j^*\circ\psi_j^\eps)(\Theta^\eps_{n,z}(\bt)))\]
with $\bt = (t_1,\dots,t_n)$ for odd $n$ when $j=2$, and for even $n$ when $j=1$.
Secondly,
$b_n^\eps(\bt,z) = b_n^0(\bt,z)+O(\eps)$. 
Thirdly, 
with $n=3$ for $j=2$ and $n=4$ for $j=1$,
\begin{equation} \label{Thirdly}
	(g_j\circ\psi_j^\eps)(\Theta^\eps_{n,z}(\bt)) = g_1(z) b_n^\eps(\bt,z) + O(\eps)
\end{equation}
for any $(\bt,z)$ that satisfies
$\Theta^\eps_{k,z}(t_1,\dots,t_k)\in S^\sharp$ for every $0\le k\le n$.
Fourthly, by repeated application of the Cauchy-Schwarz inequality,
there exists $z$ such that the set of all $\bt = (t_1,\dots,t_4)$
that satisfy
$\Theta^\eps_{k,z}(t_1,\dots,t_k)\in S^\sharp$ for every $0\le k\le 4$
has Lebesgue measure $\gtrsim |S^\sharp|^C\gtrsim  |S|^{2C}$.
Fifthly, as a consequence of the bracket condition,
the image under $\Theta^\eps_{4,z}$ of the set of all such $\bt$
has Lebesgue measure $\gtrsim |S|^C$,
and likewise 
the image under $\Theta^\eps_{3,z}$ of the set of all $\bt' =
(t_1,t_2,t_3)$ such there exists $t_4$ such that $\bt = (t_1,\dots,t_4)$
is in this same set, have Lebesgue measure $\gtrsim |S|^C$.
\end{proof}

\section{Microscale Analysis} \label{section:microscale}

The following hypotheses are in force throughout \S\ref{section:microscale},
\S\ref{section:mesoscale}, and \S\ref{section:macroscale}:
$a_j\in C^\omega(\tilde B)$ vanish nowhere. $\varphi_j:\tilde B \to\reals$ are
likewise $C^\omega$ and have pairwise linearly independent gradients at each point
of $B$. The curvature of the web defined by $\Phi$ does not vanish identically. 
The functions $f_j$ are Lebesgue measurable and are $O(1)$ almost everywhere.
These hypotheses are omitted from formulations of lemmas and intermediate propositions.
Either the auxiliary hypothesis, or a restricted variant of that hypothesis
involving a particular permutation of the indices $j\in\{1,2,3\}$,
will be assumed, and this will be indicated.
The main hypothesis and the pointwise lower bound $|f_3|\gtrsim 1$
will not be assumed until \S\ref{section:macroscale}. 

\medskip
The microscale analysis works with certain pairs of points $(x,x')\in S(\bff,\eps)\times S(\bff,\eps)$
that satisfy $|x-x'| = O(\eps)$. It is 
a hybrid microscale/macroscale analysis, in the sense that while $x-x'$ is small
(if $\eps$ is small), $x$ varies over the macroscopic region $B$.
It is at the microscale that the coefficients $a_j$ can be treated as constants,
since $\sum_j a_j(x)(f_j\circ\varphi_j)(x) = \sum_j a_j(x')(f_j\circ\varphi_j)(x) + O(\eps)$
when $x-x' = O(\eps)$.

For each $j\ne k$, $W_k\varphi_j$ vanishes nowhere; this is a restatement of the 
linear independence of the pair $\nabla\varphi_j,\nabla\varphi_k$.
Therefore $W_k\varphi_j$ is either everywhere strictly positive, or everywhere 
strictly negative. One could replace $\varphi_j$ by $-\varphi_j$
to ensure positivity. Thus the absolute value signs
on $W_k\varphi_j$ in the auxiliary hypothesis \eqref{auxhyp3}, below, are of no real significance.

\begin{lemma} \label{lemma:microscaleprelim}
Assume that for any $\tau\in\reals$,
any $C^\omega$ solution $(h_1,h_2)$ of the equation
\begin{equation} \label{auxhyp3}
\sum_{j=1}^2 a_j\,|W_3\varphi_j|^\tau \,(h_j\circ\varphi_j)=0
\end{equation}
in any nonempty open subset of $B$ vanishes identically.
Then for any $\sigma>0$ there exist $c,C\in(0,\infty)$ with the following property.

Let $\eps>0$. For each $j\in\{1,2,3\}$
let $f_j:\varphi_j(B)\to \reals$ be Lebesgue measurable
and satisfy $\norm{f_j}_{L^\infty} = O(1)$.
Let $\sS\subset S(\bff,\eps)$ be measurable.
Either
\begin{equation} |\sS| \le C\eps^c, \end{equation}
or there exist a function $f_1^*$ 
that is constant on intervals of length $\eps$, 
and a measurable set $S'\subset \sS$,
satisfying
\begin{equation} \left\{ \begin{aligned}
	&|S'|\ge c|\sS|^C  
	\\ &|f_1^*(\varphi_1(x))-f_1(\varphi_1(x))|\le C\eps^{1-\sigma} \ \forall\,x\in S'.
\end{aligned} \right. \end{equation} \end{lemma}

In this formulation, the phrase ``constant on intervals of length $\eps$''
means that there exists a partition of the domain of $f_1^*$ 
into intervals, each of length $\eps$, such that the restriction
of $f_1^*$ to each such interval is constant.
We use this language throughout the remainder of the analysis.

Recall the mappings $\psi_j$, from a subset of $\reals^3$ to $\reals^2$,
defined by $\psi_j(x,t) = (\varphi_j(x),tW_3\varphi_j(x))$.
According to Lemma~\ref{lemma:exactsolns}, any real analytic solution $\bg = (g_1,g_2)$ of 
the associated equation $\sum_{j=1}^2 a_j(x)\,(g_j\circ\psi_j)(x,t)=0$
in any nonempty open subset of $\reals^2\times\reals$ 
takes the form $g_j(y,t) = h_j(y) t^\tau$ for some $\tau\in\reals$ determined by $\bPhi$.
Therefore the hypothesis \eqref{auxhyp3} of Lemma~\ref{lemma:microscaleprelim}
implies that  any real analytic solution $\bg$ of this associated equation,
in any nonempty open set, vanishes identically.
It is this form of the hypothesis that will be used in the proof.

The hypotheses of Theorem~\ref{thm:main} are invariant under permutations
of the indices $j\in\{1,2,3\}$. 
Lemma~\ref{lemma:microscaleprelim}
is formulated in terms of a subset $\sS$ of $S(\bff,\eps)$
so that  it can be applied three times in succession, with the 
indices permuted in each application, and with $\sS$ initially taken to be $S(\bff,\eps)$
and then replaced by a suitable subset upon each application. 
This successive application will establish the following proposition,
which is the main result of this section. 

\begin{proposition}[Microscale] \label{prop:microscale}
Let $(\ba,\bPhi)$ satisfy the auxiliary hypothesis.
Then for any $\sigma>0$ there exist $c,C\in(0,\infty)$ with the following property.

Let $0<\eps\le 1$. For each $j\in\{1,2,3\}$
let $f_j:\varphi_j(B)\to \reals$ be Lebesgue measurable
and satisfy $\norm{f_j}_{L^\infty} = O(1)$.
Either
\begin{equation} |S(\bff,\eps)| \le C\eps^c, \end{equation}
or there exist functions $f_j^*$ that are constant on intervals of length $\eps$, 
and a measurable set $S\subset S(\bff,\eps)$, satisfying
\begin{equation} \label{microscaleconclusion2A}
|S|\ge c|S(\bff,\eps)|^C  
\end{equation}
and
\begin{equation} \label{microscaleconclusion2B}
|f_j^*(\varphi_j(x))-f_j(\varphi_j(x))|\le C\eps^{1-\sigma} \ \forall\,x\in S
\ \forall\,j\in\{1,2,3\}.
\end{equation}
\end{proposition}


To begin the proof of Lemma~\ref{lemma:microscaleprelim}, let $\bff,\eps,\sS$ be given.  
Let $\varrho_0$ be a small positive constant.
We say that $y\in\varphi_j(B)\subset\reals^1$ is heavy if
\begin{equation} |\{x\in \sS: \varphi_1(x)=y\}| \ge \varrho_0|\sS|.\end{equation}
In this definition, $|\cdot|$ denotes the one-dimensional measure of a subset of $\reals^2$.
Define
\begin{equation}
\sS' = \{x\in \sS: \varphi_1(x) \text{ is heavy} \}.
\end{equation}
If $\varrho_0$ is chosen to be a sufficiently small constant, 
depending only on $B$ and $\varphi_1$,
it follows that
\begin{equation}
|\sS'|\ge\tfrac12|\sS|.
\end{equation}
Indeed, the set of all $x\in \sS$ such that $\varphi_1(x)$ is not heavy
has two-dimensional Lebesgue measure $\le C\varrho_0|\sS|$ by Fubini's theorem,
with $C$ a finite constant that depends only on $B,\varphi_1$.
Choose and fix such a constant $\varrho_0$.

Form
\begin{equation}
S_0 = \{(x,s)\in B\times[-C\eps, C\eps]: (x,e^{sW_3}x)\in \sS'\times \sS'\}.
\end{equation}
The measure of $S_0$ satisfies the trivial bound
$|S_0| = O(\eps |\sS|)$. On the other hand,
by the Cauchy-Schwarz inequality, $|S_0|\gtrsim |\sS'|^2\eps\ge \tfrac14 |\sS|^2\eps$. 
Let $S_1$ be the set of all $(x,s)\in S_0$ that satisfy
\begin{equation} \label{smallvariationcase}
\sum_{j=1}^2 |(f_j\circ\varphi_j)(e^{sW_3}x)-(f_j\circ\varphi_j)(x)| \le \eps^{1-\sigma}.
\end{equation}

\begin{lemma} \label{lemma:microstep2C}
Under the hypotheses of Lemma~\ref{lemma:microscaleprelim},
if $|S_1|\ge \tfrac12 |S_0|$
then there exist a function $f_1^*$ that is constant on intervals of length $\eps$,
and a measurable set $S'_1\subset \sS$ 
satisfying $|S'_1|\gtrsim |\sS|^C$,
such that $|f^*_1\circ\varphi_1-f_1\circ\varphi_1|\le C_\sigma \eps^{1-\sigma}$
at each point of $S'_1$.
\end{lemma}

\begin{proof}
Let $\varrho>0$ be a small constant.
We say that $x\in B$ is rich if 
the set of all $s\in[-\eps,\eps]$
such that $(x,e^{sW_3}x)\in S_1$ has measure $\ge \varrho|S_1|$.
If $\varrho$ is chosen sufficiently small depending
only on $B,\varphi_3,W_3$ then it follows that the
set of all $(x,s)$ with $s\in[-\eps,\eps]$ and with $x$ not rich
has Lebesgue measure $\le\tfrac12 |S_1|$.
Therefore the set of all $(x,s)$ with $|s|\le\eps$ and $x$ rich
has measure $\ge\tfrac12 |S_1|$.
By Fubini's theorem,
the set of all rich $x\in B$ must have Lebesgue measure $\gtrsim \eps^{-1}|S_1| \gtrsim |\sS|^C$.

Choose and fix any partition of the domain of $f_1$ into intervals $I$ of sidelength $\eps$.
For a set of those intervals $I$ having cardinality $\gtrsim \eps^{-2}|S_1|$, 
there exists a rich point $x_I\in B$ satisfying $\varphi_1(x_I)\in I$.
We call those intervals $I$ good.
For each good interval $I$ choose such a point $x_I$,
and define $y_I = \varphi_1(x_I)\in I$.

The centers of the intervals $I$ take the form $n\eps$ with 
$n\in\integers$. For a large $N$, organize these into cosets according
to the residue class of $n$ modulo $N$.
Choose a coset that at least a $1/N$ fraction of all of the good intervals have centers belong to that residue class.
Retain all good intervals $I$ with centers in that coset, and discard all others.

For each retained interval $I$, let $I^*$ be the concentric interval
with $|I^*| = 2|I|$. These intervals $I^*$ are pairwise disjoint
except for their endpoints.
For each retained interval $I$,
define $f_1^*(y) = f_1(y_I)$ for all $y\in I^*$. 
Define $f_1^*\equiv 0$ on the complement of the union of these intervals $I^*$.

Define $S'_1$ to be the set of all $x\in\sS$ such that $\varphi_1(x) = e^{sW_3}x_I$
for some retained interval $I$ and for some $s$ such that $(x_I,s)\in S_1$. 
For any such $x\in \sS$, $|f_1^*(\varphi_1(x))-f_1(\varphi_1(x))| = O(\eps)$, as desired.
If $t\in\reals$ and $(x_I,e^{tW_3}x_I)\in S_1$,
$e^{tW_3}x_I$ is heavy since $S_1\subset \sS'\times\sS'$.
Therefore \[|\{x\in \sS: \varphi_1(x) = e^{tW_3}x_I\}|\ge\varrho_0|\sS|.\]

If $I\ne I'$ then are retained intervals then their contributions to $S'_1$
are disjoint. That is, if $\varphi_1(x) = e^{sW_3}x_I$ 
with $(x_I,s)\in S_1$ then $\varphi_1(x)$ is not of the form
$e^{tW_3}x_{I'}$ for any $(x_{I'},t)\in S_1$,
if $N$ is chosen sufficiently large in the above passage to cosets then
$\{\varphi_1(e^{sW_3}x_I)\}$ is guaranteed to be disjoint from	$\{\varphi_1(e^{tW_3}x_{I'})\}$.

Therefore by Fubini's theorem, $|S'_1|$ is greater than or equal to
a constant times $\varrho_0 |\sS|$ multiplied by the one-dimensional measure
of the set of all such $s$ times the number of intervals $I$ retained. 
For each $I$, the measure of the set of
all such $s$ is $\ge \varrho|S_1|$ since $x_I$ is rich. Therefore the contribution
made by each retained interval $I$ to the total measure 
is $\gtrsim \varrho_0|\sS|\cdot \varrho |S_1| \gtrsim \eps |\sS|^C$.
The number of retained intervals is $\gtrsim \eps^{-2} |S_1|\gtrsim \eps^{-1}|\sS|^C$.
Therefore $|S'_1|\gtrsim |\sS|^C$, as desired.
\end{proof}

For the remainder of the proof of Lemma~\ref{lemma:microscaleprelim},
assume that $|S_1|<\tfrac12|S_0|$.
Let $S= S_0\setminus S_1$, which satisfies $|S|\ge \tfrac12|S_0| \gtrsim |\sS|^2\eps$ and
\begin{equation} \label{makesalphalarge}
\sum_{j=1}^2 |(f_j\circ\varphi_j)(e^{sW_3}x)-(f_j\circ\varphi_j)(x)| > \eps^{1-\sigma}
\ \forall\,(x,s)\in S.
\end{equation}

\begin{lemma} \label{lemma:microstep2A}
Let the hypotheses of Lemma~\ref{lemma:microscaleprelim} be satisfied.
For any $\sigma>0$ there exist $c,C\in(0,\infty)$ with the following property.  
Let $E$ be the set of all
$(x,s)\in B\times(0,\eps]$ such that
$x\in S(\bff,\eps)$, $e^{sW_3}x \in S(\bff,\eps)$, and
\begin{equation} \label{bigdifferences}
\sum_{j=1}^2 |(f_j\circ\varphi_j)(e^{sW_3}x)-(f_j\circ\varphi_j)(x)| > \eps^{1-\sigma}. 
\end{equation}
Then $|E|\le C\eps^{c}$.
\end{lemma}

Before entering into the proof of Lemma~\ref{lemma:microstep2A}, we introduce a device that will be employed.
We aim to apply Lemma~\ref{lemma:analyticsublevel},
which states that real analytic functions that do not vanish
identically satisfy sublevel set bounds of the desired type.
Let $g_j$ and $\psi_j^\eps$ be defined as in \eqref{psiandgjdefns}. 
In order to apply  Lemma~\ref{lemma:analyticsublevel},
a $C^\omega$ approximant of the function $(x,t)\mapsto
\sum_{j=1}^2 a_j(x)\,(g_j\circ\psi_j^\eps)(x,t)$ is required. 
The functions $b_n^\eps(\bt,z)$ constructed in the framework introduced above and
defined in \eqref{bndefn} and \eqref{bdefinition2} are real analytic,
and provide good approximations to $(g_j\circ\psi_j^\eps)(\Theta^\eps_{n,z}(\bt))$,
with $n$ even for $j=1$ and $n$ odd for $j=2$.
But here a complication arises. 
While $b_n^\eps(\bt,z)$ is an analytic function of $(\bt,z)$,
we are not able to assert that
$b_n^\eps(\bt,z)$ defines a $C^\omega$ function
of $\Theta^\eps_{n,z}(\bt)$, as we would like to do,
since even in the case $n=3$ when the domain and codomains
of $\bt\mapsto\Theta^\eps_{n,z}(\bt)$ have equal dimensions, 
the mapping $\bt\mapsto\Theta^\eps_{n,z}(\bt)$ need not be invertible;
indeed, its Jacobian determinant vanishes at $\bt=(0,0,0)$.
Worse yet, for $n=4$ the dimensions of domain and codomain are unequal.
To sidestep this difficulty we will lift
$\sum_{j=1}^2 a_j(x)\,(g_j\circ\psi_j^\eps)(x,t)$
to a function of $\bt = (t_1,\dots,t_4)$.

To motivate the details of this lifting, suppose temporarily that
$\sum_{j=1}^2 (a_j\circ\pi)\cdot (g_j\circ\psi_j)$
were to vanish identically on some open set, rather than
merely being $O(\eps)$ on some measurable set. 
We could then conclude three exact functional relations. 
For $\bt = (t_1,t_2,t_3,t_4)\in\reals^4$,
let $\bt' = \bt'(\bt)= (t_1,t_2,t_3)\in\reals^3$. 
Firstly, just as for the limiting case $\eps=0$
in the proof of Lemma~\ref{lemma:exactsolns}, 
\begin{equation} \label{Thetalevelequation}
(a_1\circ\pi)(\Theta^\eps_{4,z}(\bt))
\cdot b_4^\eps(\bt,z)
+ (a_2\circ\pi)(\Theta^\eps_{4,z}(\bt))
\cdot b_3^\eps(\bt',z)
\equiv 0
\end{equation}
for all sufficiently small $|t_4|$
whenever $\Theta^\eps_{3,z}(\bt')\in U$.
Secondly, $b_3^\eps(\bt',z)$
would depend only on $\psi_2(\Theta^\eps_{3,z}(\bt'))$.
Thirdly, $b_4^\eps(\bt,z)$
would likewise depend only on $\psi_1(\Theta^\eps_{4,z}(\bt))$.
Conversely, if these three relations were to hold
for all $\bt = (\bt',t_4)$ with $|t_4|$ sufficiently small,
and with $\bt'$ in a connected open set in which
$D\Theta^0_{3,z}$ was locally invertible,
and if $\bt\mapsto \big(b_3^\eps(\bt',z),\,b_4^\eps(\bt,z)\big)$
did not vanish identically in that open set,
then there would exist a $C^\omega$ solution $(g_1,g_2)$ of 
$\sum_{j=1}^2 (a_j\circ\pi)\cdot (g_j\circ\psi_j^\eps)\equiv 0$
in a nonempty open set  that did not vanish identically there.

The conjunction of these three conditions on $(b_3^\eps,b_4^\eps)$
can be equivalently expressed as the vanishing on some nonempty open set in the $\reals^4_{\bt}$
space of a real analytic function that is defined solely in terms of the datum $(\ba,\bPhi)$,
and is defined in an open set that depends only on this datum.
Indeed, \eqref{Thetalevelequation} already expresses the vanishing of such a function. 
To express the other two conditions in the same way, construct a family of
$C^\omega$ vector fields $U_3^{z,\eta}$ 
in an appropriate open subset of $\reals^3$
that depend analytically on the parameters $z,\eps$ in a neighborhood of $\eps=0$,
do not vanishing identically in any nonempty open subset of $\reals^3$
for any $z,\eps$, and have the property that $b_3^\eta(\bt',z)$
depends only on $\psi_2^\eta(\Theta^\eta_{3,z}(\bt'))$
if and only if $U_3^{z,\eta}(b_3^\eta)\equiv 0$.
Such a family exists, and can be constructed using elementary row operations on matrices.
Likewise, there exists an $\reals^2$-valued $C^\omega$ vector field $U_4^{z,\eta}$,
defined in an appropriate open subset of $\reals^4$
and depending analytically on $(z,\eps)$, such that $b_4^\eta(\bt,z)$
depends only on $\psi_1^\eta(\Theta^\eta_{4,z}(\bt))$
if and only if $U_4^{z,\eta}(b_4^\eta)\equiv 0$ in the appropriate domain.
By squaring the left-hand side of \eqref{Thetalevelequation},
$U_3^{z,\eta}(b_3^\eta)\bt'$,
and $U_4^{z,\eta}(b_4^\eta)(\bt)$, then adding these three squares,
we obtain a real analytic function with the desired property.

A further reformulation is needed to discuss approximate relations for functions 
$g_j\circ\psi_j^\eta\circ \Theta^\eta_{n,z}$
that are approximately
equal to $(b_3^\eta,b_4^\eta)$, but are not necessarily differentiable.
$U_4^{z,\eta}(b_4^\eta)(\bt)$ vanishes identically
as a function of $\bt$ if and only if 
$b_4^\eta(e^{s'\cdot U^{\eta,z}_4}(\bt))-b_4^\eta(\bt)$ vanishes identically
as a function of $(\bt,s')$, for $s'$ in any small ball centered at $0\in\reals^2$.
There is a corresponding statement for $U_3^{z,\eta}(b_3^\eta)$.
Thus the conjunction of the three conditions in question on $(b_3^\eta,b_4^\eta)$
is equivalent to the identical vanishing, in a neighborhood of $0$, 
of the function of $(\bt,s,s')$ defined by
\begin{multline}
\big[(a_1\circ\pi)(\Theta^\eta_{4,z}(\bt))
\cdot b_4^\eta(\bt,z)
+ (a_2\circ\pi)(\Theta^\eta_{4,z}(\bt))
\cdot b_3^\eta(\bt',z)
\big]^2 
\\
+ [b_3^\eta(e^{sU_3^{z,\eta}}(\bt',z))-b_3^\eta(\bt',z)]^2
+ [b_4^\eta(e^{s'U_4^{z,\eta}}(\bt,z))-b_4^\eta(\bt,z)]^2.
\end{multline}
This concludes the explanation for the somewhat artificial construction that follows.

\begin{proof}[Proof of Lemma~\ref{lemma:microstep2A}]
Let $\bff,\eps$ be given, and let $S(\bff,\eps)\subset B$ be the associated sublevel set.
Form the functions $g_j$ and mappings $\psi_j^\eps$, as defined in \eqref{psiandgjdefns}. 
These satisfy
\[(g_j\circ\psi_j^\eps)(x,t) = (f_j\circ\varphi_j)(e^{\eps t W_3}x)-(f_j\circ\varphi_j)(x).\]
Define
\begin{equation} 
S^\sharp = \{(x,t)\in B\times(0,1]: x\in S(\bff,\eps) \text{ and } e^{\eps tW_3}x\in S(\bff,\eps)\},
\end{equation}
which satisfies $|S^\sharp|\gtrsim |S(\bff,\eps)|^2$.
Let $I$ be a compact subinterval of $(0,1]$.
Introduce auxiliary variables $s\in\reals^1$ and $s'\in\reals^2$,
varying over small neighborhoods of the origin in $\reals^1$ and in $\reals^2$, respectively.
Write $\bt = (t_1,t_2,t_3,t_4)\in\reals^4$ and $\bt' = \bt'(\bt) = (t_1,t_2,t_3)$ as above.
Let $S^\dagger$ be the set of all $(z,\bt,s,s')$,
such that $\bt\in I^4$, $|s|,|s'|\le c$, and 
\begin{equation} \left\{ \begin{aligned}
&z, \Theta^\eps_{n,z}(t_1,\dots,t_n)
\in S^\sharp \text{ for } 1\le n\le 4
\\ & \Theta^\eps_{3,z}(e^{sU_3^{z,\eps}}(\bt')),\  \Theta^\eps_{4,z}(e^{s'\cdot U_4^{z,\eps}}(\bt)) \in S^\sharp.
\end{aligned} \right. \end{equation} 

We claim that
$|S^\dagger|\gtrsim |S^\sharp|^C\gtrsim |S(\bff,\eps)|^{2C}$ for a certain exponent $C<\infty$.
To justify this, consider the simpler claim that the set of all $(z,t_1,t_2)$
such that $z\in S^\sharp$, 
$\Theta^\eps_{1,z}(t_1)\in S^\sharp$,
and $\Theta^\eps_{2,z}(t_2)\in S^\sharp$
has Lebesgue measure $\gtrsim |S^\sharp|^C$.
Equivalently, $z\in S^\sharp$,
$w_1 = e^{t_1V_1^\eps}(z)\in S^\sharp$,
and $w_2=e^{t_2V_2^\eps}(w_1)\in S^\sharp$.
We have already observed that by the Cauchy-Schwarz inequality,
the set of all $(w_1,t_2)$
such that $w_1\in S^\sharp$ and $e^{t_2V_2^\eps}(w_1)\in S^\sharp$
has Lebesgue measure $\gtrsim |S^\sharp|^2$.
Therefore the set $\tilde S$ of all $w_1\in S^\sharp$ with the property that
$e^{t_2V_2^\eps}(w_1)\in S^\sharp$ for a set of values of $t_2$
having one-dimensional Lebesgue measure $\gtrsim |S^\sharp|^2$
satisfies $|\tilde S| \gtrsim |S^\sharp|^2$.
By Cauchy-Schwarz, the set of all $(z,t_1)$ such that $z\in \tilde S$
and $e^{t_1V_1^\eps}w_1\in \tilde S$
has Lebesgue measure $\gtrsim |\tilde S|^2\gtrsim |S^\sharp|^4$.
By Fubini's theorem, then, the set of all $(z,t_1,t_2)$
satisfying the desired three inclusions 
has Lebesgue measure $\gtrsim |S^\sharp|^4\cdot|\tilde S| \gtrsim |S^\sharp|^6$.
This establishes the simpler claim.
The original claim is proved by repeating this same argument multiple times. 

For any $z$, define
\begin{equation} S^\dagger_z = \{(\bt,s,s'): (z,\bt,s,s')\in S^\dagger\big\}.  \end{equation}
By Cauchy-Schwarz and Fubini's theorem, there exists $\barz$ satisfying
\begin{equation} |S^\dagger_\barz| \gtrsim |S(\bff,\eps)|^{2C}.  \end{equation}
Choose and fix such a parameter $\barz$.

There exists a scalar $\alpha\in\reals$ 
such that for each $(\bt,s,s')\in S^\dagger_\barz$,
\begin{equation}
\left\{ \begin{aligned}
(g_1\circ\psi_1^\eps)(\Theta^\eps_{3,z}(\bt')) &= \alpha b_3^\eps(\bt') + O(\eps)
\\ (g_2\circ\psi_2^\eps)(\Theta^\eps_{4,z}(\bt)) &= \alpha b_4^\eps(\bt) + O(\eps).
\end{aligned} \right. \end{equation}
Since each $f_j=O(1)$, likewise $g_j = O(1)$.
The functions $b_n$ are products of the coefficients $a_k$ and of their inverses,
evaluated at certain points; these vanish nowhere. 
Consequently $|\alpha| = O(1)$.

On the other hand, a lower bound 
$\sum_{j=1}^2 \big|(f_j\circ\varphi_j)(e^{sW_3}x) - (f_j\circ\varphi_j)(x)\big|>\eps^{1-\sigma}$
is a hypothesis of Lemma~\ref{lemma:microstep2A}. This bound ensures that
\begin{equation} |(g_1\circ\psi_1^\eps)(\Theta^\eps_{3,z}(\bt'))|
+ |(g_2\circ\psi_2^\eps)(\Theta^\eps_{4,z,\eps}(\bt))| \gtrsim \eps^{1-\sigma} 
\ \ \forall\,(\bt,s,s')\in S^\dagger_{\barz}.
\end{equation}
Therefore
\begin{equation}\label{alphaislarge} |\alpha|\gtrsim \eps^{1-\sigma}. \end{equation}
The lower bound \eqref{alphaislarge}, and in particular its improvement
by a factor $\eps^{-\sigma}$ over $\eps$ itself, will be essential.

Since 
\begin{equation} \psi_1^\eps(\Theta^\eps_{3,z}(e^{sU_3^{z,\eps}}(\bt')))
\equiv \psi_1^\eps(\Theta^\eps_{3,z}(\bt')), \end{equation}
and since the corresponding statement holds for
$\psi_2^\eps(\Theta^\eps_{3,z}(e^{s'U_4^{z,\eps}}(\bt))$,
it follows that
\begin{equation}
b_3^\eps(e^{sU_3^{z,\eps}}(\bt')) = b_3^\eps(\bt') + O(\eps)
\ \text{ and } \  b_4^\eps(e^{s'U_4^{z,\eps}}(\bt)) = b_4^\eps(\bt) + O(\eps)
\end{equation}
for all $(z,\bt,s,s')\in S^\dagger$.
Since $b_n^\eps = b_n^0 + O(\eps)$, and since $|\alpha| = O(1)$,
we conclude that the function
$\scriptf_\barz$ defined by
\begin{multline} \label{scriptfdefn}
\scriptf_\barz(\bt,s,s') = 
\big[(a_1\circ\pi)(\Theta^0_{4,\barz}(\bt))
\cdot (b_4^\eta(\Theta^0_{4,\barz}(\bt)))
\\
+ (a_2\circ\pi)(\Theta^0_{4,\barz}(\bt))
\cdot (b_3^\eta(\Theta^0_{3,\barz}(\bt'))) \big]^2 
\\
+ [b_3^0(e^{sU_3^{z,0}}(\bt'))-b_3^0(\bt')]^2
+ [b_4^0(e^{s'\cdot U_4^{z,0}}(\bt))-b_4^0(\bt)]^2
\end{multline}
satisfies 
\begin{equation} 
\alpha \cdot \scriptf_\barz(\bt,s,s') = O(\eps)\ \forall\, (\bt,s,s')\in S^\dagger_\barz.
\end{equation}
The quantity $\scriptf_z(\bt,s,s')$ 
is a $C^\omega$ function of $(z,\bt,s,s')$.
The parameter $z$ can be taken to vary over a compact set.
Recall that the hypothesis \eqref{auxhyp3} of Lemma~\ref{lemma:microstep2A} 
ensures that any real analytic solution $\bg^* = (g_1^*,g_2^*)$ of 
$\sum_{j=1}^2 (a_j\circ\pi)\cdot(g_j^*\circ\psi_j)\equiv 0$
in any nonempty open set must vanish identically.
Therefore there exists no $z$ for which the function $\scriptf_z$ vanishes identically in 
a nonempty open set.
By Lemma~\ref{lemma:analyticsublevel}, 
these properties guarantee that there exist $\tau,C\in(0,\infty)$
such that for every $z$ and every $\delta>0$,
\begin{equation}
\big|\big\{ (\bt,s,s'): |\scriptf_z(\bt,s,s')|<\delta \big\}\big| \le C\delta^\tau.
\end{equation}
Applying this inequality with $\delta = O(\eps|\alpha|^{-1}) = O(\eps^\sigma)$, we conclude that
$|S^\dagger_\barz| = O(\eps^{\tau \sigma})$
and therefore that $|S(\bff,\eps)| = O(\eps^{c})$
for an exponent $c>0$ that depends on $\sigma$.
This concludes the proof of Lemma~\ref{lemma:microstep2A}.
\end{proof}

\section{A two term sublevel set inequality}

We formulate an intermediate result,
whose proof is implicit in that of Lemma~\ref{lemma:microstep2A} and
which will be a key element of the mesoscale analysis, below.
It concerns two term sublevel set inequalities, for functions of two variables,
in a three-dimensional ambient space.  
We use the notation
\begin{equation} \label{notation:resurrected} 
	S(\bg,\eps) = \{(x,t)\in B\times [-1,1]:
\Big|\sum_{j=1}^2 b_j(x)\,(g_j\circ\psi_j)(x,t)\Big|<\eps\}. \end{equation}

\begin{proposition} \label{prop:resurrected}
Let $B\subset\reals^2$ be a closed ball of positive, finite radius,
and let $\tilde B$ be an open neighborhood of $B$.
Let $b_j:\tilde B\to\reals$ be nowhere vanishing $C^\omega$ functions.
Let $\varphi_1,\varphi_2:\tilde B\to\reals^1$ be $C^\omega$
submersions whose gradients are everywhere linearly independent.
Let $\beta_j:B\to\reals$ be nowhere vanishing $C^\omega$ functions.
Let $\psi_j$ be the associated mappings
$\psi_j(x,t) = (\varphi_j(x),t\beta_j(x))$.
Assume that there exist nowhere vanishing  $C^\omega$ vector fields $V_j$
that satisfy $V_j(\psi_j)\equiv 0$, and 
satisfy the bracket condition at each point of 
$\tilde B\times(\reals\setminus\{0\})$.

Suppose that there exists no nonzero $C^\omega$ solution
$\bg^* = (b_1^*,b_2^*)$ of the equation 
	\[ \sum_{j=1}^2 b_j(x)(g_j^*\circ\psi_j)(x,t)=0\]
in any nonempty open subset of 
$\tilde B\times(\reals\setminus\{0\})$.
Then for each $\sigma>0$ there exist $C<\infty$ and $\eps_0>0$
such that for any $\eps\in(0,\eps_0]$
	and any measurable functions $g_j:B\times[-1,1]\to\reals$
there exists $S\subset S(\bg,\eps)$
satisfying $|S|\ge |S(\bg,\eps)|^C$
such that for each $j\in\{1,2\}$,
$|g_j(\psi(x))|\le C\eps^{1-\sigma}$ for every $x\in S$.
\end{proposition}

More generally,
let $N\ge 0$, let $B'\subset \reals^N$ be an open ball
centered at $0$, and let $b_j^*:\tilde B\times B'\to\reals$ 
and $\varphi_j^*:\tilde B\times B'\to\reals$ 
be $C^\omega$ functions.
Set $b_j^s(x) = b_j(x,s)$ and $\varphi_j^s(x)= \varphi_j(x,s)$.
Suppose that $((b_j^0,\varphi_j^0): j\in\{1,2,3\})$ satisfies the hypotheses
of Proposition~\ref{prop:resurrected}.
Then there exists a neighborhood $B''\subset B'$ of $0$
such that for each $s\in B''$, 
$((b_j^s,\varphi_j^s): j\in\{1,2,3\})$ satisfies the conclusions of
Proposition~\ref{prop:resurrected}, with constants independent of $s$.


\section{Mesoscale Analysis} \label{section:mesoscale}

In the author's view, the mesoscale step is the decisive one
in the analysis of sums $\sum_j a_j\cdot(f_j\circ\varphi_j)$ with variable coefficients $a_j$.
It works with certain pairs of points $(x,x')\in S(\bff,\eps)\times S(\bff,\eps)$
that satisfy $|x-x'| = O(\delta)$, for a certain $\delta$ that is large relative to $\eps$
but small relative to $1$. It is a hybrid analysis, in the same sense as the microscale analysis.
At the mesoscale, the variation of the coefficients $a_j$ creates significant effects.


Choose, and fix for the remainder of the proof of Theorem~\ref{thm:main},
an auxiliary exponent $\gamma\in(\tfrac12,1)$.
Given any $\eps>0$, set $\delta = \eps^\gamma$. 
Partition $\reals^1$ into intervals $J$, each having length $\delta$.

\begin{proposition}[Mesoscale] \label{prop:mesoscale}
Let $(\ba,\bPhi)$ satisfy the auxiliary hypothesis.
For any $\rho>0$ there exist
$c,C\in(0,\infty)$ with the following property.

Let $\eps>0$. For each $j\in\{1,2,3\}$
let $f_j:\varphi_j(B)\to \reals$ be Lebesgue measurable
and satisfy $\norm{f_j}_{L^\infty} \le 1$.
Let $\sS\subset S(\bff,\eps)$ be measurable.
Either
\begin{equation} \label{mesoscaleconclusion1} |\sS| \le C\eps^c, \end{equation}
or there exist a measurable set $S\subset \sS$,
and for each $j$, functions $f_j^*$ that are affine on intervals $J$ of length $\eps^\gamma$,
such that
\begin{equation} \label{mesoscaleconclusion2a} |S|\ge c|\sS|^C  \end{equation}
and
\begin{equation} \label{affineapproximation}
|f_j^*(\varphi_j(x))-f_j(\varphi_j(x))| \le C\eps^{1-\rho} \ \forall\,x\in S
\ \forall\,j\in\{1,2,3\}.
\end{equation}
Moreover, the derivatives $(f_j^*)'$ satisfy
\begin{equation} \label{f*primebound} |(f_j^*)'| \le C \eps^{-\rho}.  \end{equation}
\end{proposition}

The relation \eqref{affineapproximation} 
asserts an approximation to $f_j$ on intervals of lengths $\delta$.
If the parameters are chosen so that $\rho<1-\gamma$ then
the upper bound $O(\eps^{1-\rho})$ is of the form $O(\delta^\kappa)$ with
$\kappa$ strictly greater than $1$. Thus 
\eqref{affineapproximation} provides a quantitative differentiable approximation to $f_j$ at scale $\delta$, 
on a set $S$ of significantly large Lebesgue measure.

\medskip
Throughout the proof of Proposition~\ref{prop:mesoscale}, we maintain the relation $\delta = \eps^\gamma$. 
The exponent $\gamma\in(\tfrac12,1)$ remains fixed throughout the discussion.

Let $\rho>0$ be given.
We begin by applying Proposition~\ref{prop:microscale} to the datum $(\bff,\sS,\eps)$, 
with a parameter $\sigma$ that depends on $\rho$ and is to be chosen below.
Proposition~\ref{prop:microscale} has two alternative conclusions.
If the first conclusion holds then $|\sS| \le C\eps^c$, as desired. 
Therefore we may assume that the second conclusion 
holds. This provides a measurable set $\sS'\subset \sS$ satisfying
$|\sS'|\gtrsim |\sS|^C$, and a function 
$\bff^*=(f_1^*,f_2^*,f_3^*)$ such that each $f_j^*$ is constant on intervals of length $\eps$,
and satisfies $|f_j^*\circ\varphi_j(x)-f_j\circ\varphi_j(x)| = O(\eps^{1-\sigma})$
for each index $j$ and every point $x\in \sS'$.
To conclude the proof of Proposition~\ref{prop:mesoscale},
it suffices to show that if $\sigma$ is chosen sufficiently small as a function of $\rho$ then
there exist $c,C\in(0,\infty)$ such that for any measurable set $\tilde \sS\subset S(\bff^*,\eps^{1-\sigma})$
there exist  $\tilde \sS'\subset \tilde \sS$ and 
$\bff^\dagger=(f_1^\dagger,f_2^\dagger,f_3^\dagger)$ 
with each $f_j^\dagger$ affine on intervals of length $\delta$,
and with derivatives satisfying $|(f_j^\dagger)'| = O(\eps^{-\rho})$, such that
\begin{equation} |\tilde \sS'|\ge c|\tilde \sS|^C \end{equation}
and for each index $j$, 
\begin{equation}
|f_j^\dagger\circ\varphi_j(x)-f_j^*\circ\varphi_j(x)| \le C\eps^{1-\rho}
\ \forall\,x\in \tilde \sS'. 
\end{equation}

Henceforth, write $f_j$ in place of $f_j^*$, and $\sS$ in place of $\tilde \sS$, to simplify notation.
Thus the domain $\varphi_j(B)$ of each $f_j$ is partitioned into intervals, each of length $\eps$,
with $f_j$ constant on each of those intervals.

We claim that it suffices to treat the case in which
the functions $f_j$ are nearly constant at the mesoscale,
as well as at the microscale, that is,
\begin{equation}  \label{mesokey}
f_j(y)-f_j(y') = O(\delta^{1-\sigma}) 
\end{equation}
whenever $y,y'$
lie in a common interval $J$, among those intervals of length $\delta$ into which we have
partitioned $\varphi_j(B)$.
Indeed, since $\sS\subset S(\bff,\eps^{1-\sigma})\subset S(\bff,\delta)$
provided that $\sigma$ is chosen to be sufficiently small,
Proposition~\ref{prop:microscale}
can be applied once more, but now with $\eps$ replaced by $C\delta$.
The conclusion is that either $|\sS| = O(\delta^c) = O(\eps^{\gamma c})$,
or there exist $\barsS\subset \sS$ satisfying $|\barsS|\gtrsim |\sS|^C$
and functions $f_j^{**}$ 
that are constant on intervals of length $\delta$ and satisfy
\begin{equation}  \label{starstarproperty}
f_j^{**}\circ\varphi_j(x)-f_j\circ\varphi_j(x) = O(\delta^{1-\sigma}) \ \forall\,x\in \barsS
\end{equation}
for each $j\in\{1,2,3\}$.
In the former case, in which $|\sS| = O(\eps^{\gamma c})$,
the first alternative conclusion \eqref{mesoscaleconclusion1} 
of Proposition~\ref{prop:mesoscale} has been established, and the proof 
of that proposition is complete.

Consider the latter case.
The local constancy of $f_j^{**}$,
together with the inequality \eqref{starstarproperty}, imply that the functions $f_j$
satisfy \eqref{mesokey} with $\sS$ replaced by its subset $\barsS$.
We do not replace $f_j$ by $f_j^{**}$; the functions $f_j$
remain constant on intervals of length $\eps$ but are not necessarily
constant on any larger intervals.
Note also that near constancy at the mesoscale, as expressed in \eqref{mesokey},
involves a bound $O(\delta^{1-\sigma})$, not $O(\eps^{1-\sigma})$.
Our present goal is to show that either $|\barsS| = O(\eps^c)$,
or there exist $S\subset\barsS$ satisfying $|S|\gtrsim |\barsS|^C$
and functions $f_j^*$ satisfying
\eqref {affineapproximation} and \eqref{f*primebound}.


By dividing through by the nowhere vanishing analytic coefficient $a_3$,  
we may assume without loss of generality that $a_3\equiv 1$. 
Continue to denote by $W_3$ some nowhere vanishing real analytic vector field
in $\tilde B$ that annihilates $\varphi_3$.
Form 
\begin{equation}
\tilde S_1 = \{(x,s)\in B\times[-\delta,\delta]: (x,e^{sW_3}x)\in \barsS\times \barsS\},
\end{equation}
which satisfies $|S_1|\gtrsim |\barsS|^2\delta$.
Either the set of all points with $s>0$ contributes at least half of
the Lebesgue measure of $\tilde S_1$, or the set of all points with $s<0$ does so.
We discuss only the first case; the analysis will apply equally well to the second.
Replace $\tilde S_1$ by $S_1 = \{(x,s)\in\tilde S_1: s>0\}$.

For each $(x,s)\in S_1$,
\begin{equation}
\sum_{j=1}^2 \Big[a_j\cdot (f_j\circ\varphi_j)(e^{sW_3}x)
-a_j\cdot(f_j\circ\varphi_j)(x)\Big] = O(\eps);
\end{equation}
the terms involving $f_3$ cancel since $\varphi_3(e^{sW_3}x) = \varphi_3(x)$
and $a_3(e^{sW_3}x)=1=a_3(x)$.
Writing
\begin{equation}
a_j(e^{sW_3}x) -a_j(x) = sW_3a_j(x)+O(s^2)
\end{equation}
and dividing by $s$ gives
\begin{multline}
\sum_{j=1}^2 \Big[
(W_3a_j)(x)(f_j\circ\varphi_j)(x) + 
a_j(e^{sW_3}x)) \big[(f_j\circ\varphi_j)(e^{sW_3}x)-(f_j\circ\varphi_j)(x) \big]s^{-1}
\Big] 
\\ = O(\eps|s|^{-1})
\end{multline}
since $|s|^2\le \eps^{2\gamma}\le\eps$ because $\gamma > \tfrac12$.
Writing $a_j(e^{sW_3}x) = a_j(x) + O(s) = a_j(x)+O(\delta)$
and invoking the upper bound \eqref{mesokey}
to obtain $(f_j\circ\varphi_j)(e^{sW_3}x) - (f_j\circ\varphi_j)(x) = O(\delta^{1-\sigma})$
gives
\begin{multline}
\sum_{j=1}^2 \Big[
(W_3a_j)(x)(f_j\circ\varphi_j)(x) + 
a_j(x) \big[(f_j\circ\varphi_j)(e^{sW_3}x)-(f_j\circ\varphi_j)(x) \big]s^{-1}
\Big] 
\\ = O(\eps|s|^{-1}) +O(\delta^{2-\sigma}|s|^{-1}),
\end{multline}
which is again $O(\eps|s|^{-1})$ provided that $\sigma$
is chosen to be sufficiently small, since $\delta^2 = \eps^{2\gamma}$ and $\gamma>\tfrac12$.

The quantity $(f_j\circ\varphi_j)(e^{sW_3}x)$ can be simplified for all $(x,s)$
in a relatively large set.  Indeed, 
\begin{equation}
\varphi_j(e^{sW_3}x)
= \varphi_j(x) + sW_3\varphi_j(x) + O(s^2).
\end{equation}
The domain of $f_j$ has been partitioned into intervals of length $\eps$,
with $f_j$ constant on each of these intervals.
Therefore 
\begin{equation} \label{fjmesosimplified}
f_j(\varphi_j(e^{sW_3}x)) = f_j\big(\varphi_j(x) + sW_3\varphi_j(x)\big) 
\end{equation}
unless $\varphi_j(e^{sW_3}x)$ lies within distance $O(\delta^2) = O(\eps^{2\gamma})$
of one of  the endpoints of one of these intervals.
The set of all $(x,s)$ for which this happens,
has measure $O(\delta\cdot\eps^{-1}\cdot \eps^{2\gamma})
= O(\eps^{2\gamma-1}\delta)$.

The exponent $2\gamma-1$ is strictly positive by design.
Let $c_0,\varrho>0$ be small constants that satisfy
$2c_0<\varrho<\rho/2$,  $1-\gamma-\varrho>0$, and $2c_0 < 2\gamma-1$.
If $|\barsS|\le\eps^{c_0}$ then the proof of Proposition~\ref{prop:mesoscale} is complete.
If not, then recall that $|S_1|\gtrsim |\barsS|^2\delta
\ge \eps^{2c_0}\delta$. For small $\eps>0$, 
$\eps^{2\gamma-1}\delta\ll \eps^{2c_0}\delta$ is consequently negligible relative to $|S_1|$.
Therefore for all $x$ in a set $S_2\subset S_1$ of Lebesgue measure
$\gtrsim \delta\eps^{2c_0}$, \eqref{fjmesosimplified}
holds for both indices $j=1,2$.

Define $S_3$ to be the set of all $(x,s)\in S_2$ satisfying
$|s|\ge \eps^\varrho\delta$. Then 
$|S_2\setminus S_3| = O(\eps^\varrho\delta)$,
which is negligible relative to $|S_2|\gtrsim \eps^{2c_0}\delta$. 
Thus $|S_3|\gtrsim|S_2|\gtrsim \eps^{2c_0}\delta$,
and \eqref{fjmesosimplified} holds for every $(x,s)\in S_3$.
Therefore
\begin{multline} 
\sum_{j=1}^2 \Big[
(W_3a_j)(x)(f_j\circ\varphi_j)(x) + 
	a_j(x) \big[f_j(\varphi_j(x) + sW_3\varphi_j(x))-(f_j\circ\varphi_j)(x) \big]s^{-1} \Big] 
\\ = O(\eps|s|^{-1})\ \forall\,(x,s)\in S_3.
\end{multline}

Defining
\begin{equation}
\tilde F_j(y,s) = s^{-1}(f_j(y+s)-f_j(y)),
\end{equation}
this relation can be rewritten as
\begin{multline} 
\sum_{j=1}^2 \Big[
(W_3a_j)(x)(f_j\circ\varphi_j)(x) + 
a_j(x)\cdot W_3\varphi_j(x) \cdot \tilde F_j(\varphi_j(x),sW_3\varphi_j(x)) \Big] 
\\ = O(\eps^{1-\gamma-\varrho}) \ \forall\,(x,s)\in S_3.
\end{multline}
Substituting $F_j(x,t) = \tilde F_j(x,\delta t)$ and writing
$\psi_j(x,t) = (\varphi_j(x), tW_3(\varphi_j)(x))$, as in the microscale analysis, 
this becomes
\begin{multline} \label{differenced3}
 \sum_{j=1}^2 \Big[
(W_3a_j)(x)(f_j\circ\varphi_j)(x) + 
a_j(x)\,W_3\varphi_j(x) \cdot (F_j\circ\psi_j)(x,t) \Big] 
\\ = O(\eps^{1-\gamma-\varrho}) \ \forall\,(x,t)\in S_4
\end{multline}
where $S_4=\{(x,t)\in B\times(\eps^\varrho,1]: (x,\delta t)\in S_3\}$
satisfies 
\begin{equation}\label{|S4|bound}
|S_4|= \delta^{-1}|S_3| \gtrsim \eps^{2c_0} \end{equation}
and also
\begin{equation} \label{S4secondbound}
|S_4| \gtrsim |\barsS|^2.
\end{equation}

In comparison to the framework introduced in \S\ref{section:framework}
and exploited in \S\ref{section:microscale},
a simplification is that the variant mappings $\psi_j^\delta$ have been sidestepped
so that only their limits $\psi_j = \psi_j^0$ need be dealt with.
On the other hand, new terms $W_3a_j\cdot (f_j\circ\varphi_j)$ have arisen
at the mesoscale, and these are not negligible.
There are now effectively four unknown functions $f_1,F_1,f_2,F_2$,
rather than the three with which the analysis began.

The terms $W_3a_j(x)\,(f_j\circ\varphi_j)(x)$
are independent of $t$. To exploit this we introduce
\begin{equation} S^\sharp = \{(x,t,t')\in B\times(0,1]\times(0,1]:
(x,t)\in S_4 \text{ and } (x,t')\in S_4\}
\end{equation}
and subtract to obtain
\begin{equation}
 \sum_{j=1}^2 a_j(x) W_3\varphi_j(x)
\big[ (F_j\circ\psi_j)(x,t')-(F_j\circ\psi_j)(x,t) \big]  = O(\eps^{1-\gamma-\varrho})
\ \forall\, (x,t,t')\in S^\sharp.
\end{equation}
Thus we arrive at a sublevel problem in the $4$-dimensional $(x,t,t')$ space,
associated to the mappings $(x,t,t')\mapsto \Psi_j(x,t,t') = (\psi_j(x,t'),\psi_j(x,t))$, 
and with each quantity $(F_j\circ\psi_j)(x,t')-(F_j\circ\psi_j)(x,t)$
regarded as a function of $\Psi_j(x,t,t')$.
The terms in \eqref{differenced3}
involving $f_j\circ\varphi_j$ have been eliminated.

It is natural to seek to invoke a higher-dimensional analogue of
Proposition~\ref{prop:resurrected}.  However, in contrast to the microscale analysis,
this particular higher-dimensional sublevel set problem is degenerate, in the sense that 
vector fields $U_j$ in $\reals^4$ that annihilate $\Psi_j$ do not satisfy the bracket condition.
Instead, the submanifolds defined by constancy of $t'/t$ 
form a foliation of $\reals^4$ by three-dimensional leaves invariant under these vector fields.

To see this degeneracy,
for each $0\ne r\in\reals$, define $\scripth_r$ to be the set of all $(x,t,t')$ satisfying $t'=rt$.
Define 
\begin{multline} 
\Psi_j(x,t,t') = (\psi_j(x,t),\,\psi_j(x,t'))
\\
= \big((\varphi_j(x),tW_3\varphi_j(x)),\,(\varphi_j(x),t'W_3\varphi_j(x)) \big)\in\reals^2\times\reals^2.
\end{multline}
Since the first coordinate $\varphi_j(x)$
of $(\varphi_j(x),tW_3\varphi_j(x))$
is identically equal to the first coordinate of
$(\varphi_j(x),t'W_3\varphi_j(x))$,
each $\Psi_j$ takes
values in a three-dimensional subspace of $\reals^2\times\reals^2$.
Thus $\Psi_j$ may be regarded as mappings from $\reals^4$ to $\reals^3$.

For any $r\ne r'$ and any $j\in\{1,2\}$,
$\Psi_j(\scripth_r)\cap \Psi_j(\scripth_{r'}) = \emptyset$.
Indeed, for any $(x,t,t')\in\scripth_r$,
the ratio of the second coordinate 
$t' W_3\varphi_j(x)$
of $\psi_j(x,t')$ to the second coordinate $t W_3\varphi_3(x)$ of $\psi_j(x,t)$ 
equals $t'/t=r$. Thus $t'/t$ is determined by $\Psi_j(x,t,t')$.

For each $r\in[\eps^\varrho,\eps^{-\varrho}]$ define
\begin{equation} \label{Frdefn}
F_{j,r}^\sharp(y,t) = F_j(y,rt)-F_j(y,t) \end{equation}
and
\begin{equation}
S^\sharp_r = \{(x,t)\in\reals^3: (x,t,rt)\in S^\sharp\}.
\end{equation}
With the notations
\begin{equation}
\tilde a_j(x) = a_j(x) W_3\varphi_j(x) \ \text{ and } \ 
\eta = \eps^{1-\gamma-\varrho},  
\end{equation}
one has
\begin{equation} \label{Frrelation}
\sum_{j=1}^2 \tilde a_j(x) \,(F_{j,r}^\sharp\circ\psi_j)(x,t)
= O(\eta) \ \forall\,(x,t)\in S^\sharp_r.
\end{equation}

The inequality \eqref{Frrelation} relates $F_{1,r}^\sharp$ to $F_{2,r'}^\sharp$
only when $r'=r$; this expresses the degeneracy of this problem.
Thus we have a one-parameter family of sublevel sets, parametrized by $r$,
in a three-dimensional domain, of the type addressed by Proposition~\ref{prop:resurrected}.
Each of these sublevel sets will be analyzed in the same way as in the microscale analysis,
with the parameter $r$ subsequently taken into account.
While the functions $F^\sharp_{j,r}$ that appear in \eqref{Frrelation}
depend on $r$, the coefficients $\tilde a_j$, the mappings $\psi_j$,
and the inequality itself do not involve $r$.
Thus any conclusions about $F_{j,r}^\sharp$ gleaned from \eqref{Frrelation} hold uniformly in $r$.

The linear form $\sum_{j=1}^2 (\tilde a_j\circ\pi) (g_j\circ \psi_j)$
in \eqref{Frrelation} is not identical to the form
$\sum_{j=1}^2 (a_j\circ\pi) (g_j\circ \psi_j)$
that was encountered in \S\ref{section:microscale},
because the coefficients $\tilde a_j$ are different from $a_j$.
However, since $W_3\varphi_j|W_3\varphi_j|^\tau = \pm |W_3\varphi_j|^{\tilde\tau}$
with $\tilde\tau = \tau+1$,
the hypothesis \eqref{auxhyp3} for the form with coefficients $a_j$
is equivalent to the corresponding hypothesis for coefficients $\tilde a_j$.
Therefore Proposition~\ref{prop:resurrected} can be applied to \eqref{Frrelation},
to conclude that for each $r$, either $|S_r^\sharp| = O(\eta^c)$ or there exists
a measurable subset $S'_r\subset S_r^\sharp$
satisfying $|S'_r|\gtrsim |S_r^\sharp|^C$ such that
\[ F_{j,r}^\sharp(\psi_j(x,t))   = O(\eta^{1-\sigma})
\ \forall\,(x,t) \in S'_r
\ \forall\,j\in\{1,2\}. \]

Recall that $\eta = \eps^{1-\gamma-\varrho}$.
Let $c_1>0$ be small.
If 
\begin{equation} \label{sharpSbound}
\big| \big\{r: 
|S_r^\sharp| \gtrsim (\eps^{1-\gamma-\varrho})^c\big\} \big|
\le\eps^{c_1}
\end{equation}
then
we conclude that 
\begin{equation} |S_4|= O(\eps^{c_1}+\eta^c) 
= O(\eps^{c_1} + \eps^{(1-\gamma-\varrho)c}).\end{equation} 
By \eqref{S4secondbound}, this implies
an upper bound $|\barsS| = O(\eps^{c'})$ for some $c'>0$,
completing the proof of Proposition~\ref{prop:mesoscale}.

Suppose instead that \eqref{sharpSbound} does not hold.
Assuming that $\eps\le\eps_0$ and that $\eps_0$ is sufficiently small,
the set of all $(y,t,r)$ satisfying 
\begin{equation} 
F_j(y,rt)-F_j(y,t) = O(\eta^{1-\sigma})
\ \text{ for each $j\in\{1,2\}$} \end{equation} 
consequently has Lebesgue measure $\ge \eps^{c_2}$
with $c_2 = (1-\gamma-\varrho)\cdot c\cdot c_1/2$.
Therefore
\begin{equation} 
F_j(y,t')-F_j(y,t) = O(\eta^{1-\sigma})
\ \text{ for each $j\in\{1,2\}$}, \end{equation} 
for all $(y,t,t')$ in a set of measure $\gtrsim \eps^{c}$
for a certain exponent $c>0$.
By Fubini's theorem, there exists $\bart$ such that
\begin{equation} 
	F_j(y,t) = F_j(y,\bart) + O(\eta^{1-\sigma})
	\ \text{ for each $j\in\{1,2\}$}
\end{equation} 
for all $(y,t)$ in a set of measure $\gtrsim \eps^c$.
That is, setting $F_j^*(y) = F_j(y,\bart)$ and invoking the definition 
$F_j(y,t) = (\delta t)^{-1}(f_j(y+\delta t)-f_j(y))$,
\begin{equation}
	f_j(y+\delta t)-f_j(y) = \delta tF_j^*(y) + O(\delta \eta^{1-\sigma})
\end{equation}
for all $(y,t)$ in that same set.
Now 
\begin{equation} \delta \eta^{1-\sigma} = \eps^\gamma \eps^{(1-\gamma-\varrho)(1-\sigma)} 
= \eps^{1-\varrho-(1-\gamma)\sigma} = \eps^{1-\varrho'}  \end{equation}
where $\varrho' = \varrho + (1-\gamma)\sigma$ can be made to be as small
as desired, by choosing $\varrho,\sigma$ to be sufficiently small.
Thus substituting $s = \delta t$,
\begin{equation} \label{Ljdefined*}
	f_j(y+s) = f_j(y) + s F_j^*(y) 
	+ O(\eps^{1-\varrho'})
\end{equation}
for all $(y,s)$ that satisfy $|s|\lesssim \delta$
and lie in a certain set of Lebesgue measure $\gtrsim \delta\cdot \eps^c$.

As in the proof of Lemma~\ref{lemma:microstep2C},
this means that for each $j$, upon partitioning the domain of each $f_j$
into intervals $I_{j,n}$ of lengths $\delta$,
there exist affine functions $L_{j,n}$
so that if $f_j^*$ is defined to be $L_{j,n}$ on each $I_{j,n}$
then for all $x$ belonging to a subset $S^*\subset \barsS$
satisfying $|S^*|\gtrsim |\barsS|^C$,
\begin{equation}  \label{upon_partitioning}|
|f_j\circ\varphi_j(x) -f_j^*\circ\varphi_j(x)|
\lesssim \eps^{1-\varrho'}\ \ \forall\,j\in\{1,2\}\ \forall\,x\in S^*.
\end{equation} 

\medskip
To conclude the proof of Proposition~\ref{prop:mesoscale},
it remains to show that if 
$\varrho'$ is 
sufficiently small then the construction ensures that $(f_j^*)' = O(\eps^{-\rho})$.
This is a matter of unraveling the above sequence of steps.
For $y+s$ in each $\delta$-interval,
$f_j^*(y+s)$ satisfies
\begin{equation*}
	f_j(y+s) = f_j(y) + s F_j(y,\bart) 
\end{equation*}
with $y = \varphi_j(x)$ with $(x,\bart)\in S_4$. 
Thus the derivative in question equals $F_j(\varphi_j(x),\bart)$ for some $(x,t)\in S_4$.
Now $(x,\bart)\in S_4$ if and only if $(x,\delta\bart)\in S_3$,
whence $(x,\delta\bart)\in S_2$ and $|\bart|\ge\eps^\varrho$.
By the definition of $S_2$, this implies that $(x,e^{sW_3}x)\in \barsS\times\barsS$.
This, in turn, implies that \eqref{mesokey} holds, that is,
$f_j\circ\varphi_j(x)-f_j\circ\varphi_j(x') = O(\delta \eta^{1-\sigma})$ 
and consequently 
\begin{equation} \begin{aligned}
F_j(\varphi_j(x),\bart) &= (\delta\bart)^{-1}
\big((f_j\circ\varphi_j)(e^{\delta\bart W_3}x)-(f_j\circ\varphi_j)(x) \big)
\\ &= O\big((\delta \eps^\varrho)^{-1} \delta \eta^{1-\sigma}\big) 
\\ & = O\big(\eps^{1-\varrho-\varrho'}).
\end{aligned} \end{equation}
By choosing $\sigma$ and $\varrho$, hence also $\varrho'$, to be sufficiently small, we ensure that
this bound is $O(\eps^{-\rho})$. That concludes the proof of Proposition~\ref{prop:mesoscale}.
\qed


\section{Macroscale Analysis} \label{section:macroscale}

The overarching strategy of the proof of Theorem~\ref{thm:main}
is to upgrade the smallness of $\sum_j a_j \cdot (f_j\circ\varphi_j)$
to smallness of its gradient.  The gradient has two components, so 
a single scalar inequality is thereby transformed into two scalar inequalities,
allowing the elimination of one of the three unknown functions $f_j$
and thus reducing matters to a potentially simpler problem. In this section
we use the differentiability established through 
the mesoscale analysis to implement this strategy.

In the following proposition and its proof, $f'$ denotes the derivative, in the pointwise sense,
of a piecewise differentiable function. 

\begin{proposition}[Macroscale] \label{prop:macroscale}
Let $(\ba,\bPhi,\bff)$ satisfy all hypotheses of Theorem~\ref{thm:main}.
Suppose that the curvature of the web in $B$ defined by $\Phi$ does not vanish identically. 
For any $\varrho>0$ there exist
$\rho,c,C\in(0,\infty)$ with the following property.

Let $\delta>0$. 
Suppose that for each index $j$,
the domain of $f_j$ is partitioned into intervals $I$ of common lengths $\delta$,
and that the restriction of $f_j$ to each such interval $I$
is an affine function whose derivative satisfies $|f'_j|\le\delta^{-\rho}$.
Then the sublevel set
$S(\bff,\delta^{1+\varrho})$
satisfies
$|S(\bff,\delta^{1+\varrho})| \le C\delta^c$.
\end{proposition}

Only the case $\tau=0$ of the auxiliary hypothesis \eqref{auxhyp3} will be used
in the proof of Proposition~\ref{prop:macroscale}, but the main hypothesis and the lower bound
on $|f_3|$ will come into play for the first time.

To simplify notation below we analyze the subset $S(\bff,\delta^{1+2\varrho})$
rather than $S(\bff,\delta^{1+2\varrho})$.
Set $\eta = \delta^{1+2\varrho}$.
We will begin the proof of Proposition~\ref{prop:macroscale} 
by establishing the following approximation result.

\begin{lemma} \label{lemma:meromorphicity}
Let $(\ba,\bPhi)\in C^\omega(\tilde B)$.
Suppose that $a_j$ vanish nowhere, and that the gradients of the mappings
$\varphi_j$ are everywhere pairwise transverse.
Suppose that the curvature of the web defined by $\Phi$ does not vanish
identically, and that $(\ba,\bPhi)$ satisfies  the weak auxiliary hypothesis.

Let $k\in\{1,2,3\}$.
There exist compact sets $T_k$ and $\Gamma_k$, 
a function $(y,\bt,\theta)\mapsto  \alpha_{\bt,\theta}(y)$
that is $C^\omega$ in a neighborhood of $\varphi_k(B)\times T_k\times\Gamma_k$,
and a function $(y,\bt)\mapsto \beta_\bt(y)$ 
that is $C^\omega$ in a neighborhood of $\varphi_k(B)\times T_k$
with the following property.

For any $\varrho>0$ there exists $\rho_0(\varrho)>0$
such that for any $\rho\in(0,\rho_0(\varrho))$
there exist $c,C\in(0,\infty)$ with the following property.
Let $\delta>0$. Let $\bff$ be Lebesgue measurable
and satisfy $\norm{f_j}_{L^\infty} = O(1)$ for each index $j$.
Suppose that each
$f_j$ is affine on intervals of lengths $\delta$
and satisfies $|f'_j| = O(\delta^{-\rho})$.
Let $\sS\subset S(\bff,\delta^{1+2\varrho})$ be measurable. Then
either $|\sS| = O(\delta^c)$ or
there exist a measurable subset $\sS'\subset\sS$
satisfying $|\sS'|\gtrsim |\sS|^C$,
elements $\bt\in T_k$ and $\theta\in\Gamma_k$,
and a scalar $\barr\in[c\delta^{\rho},1]$
such that for every $y\in \varphi_k(\sS')$,
\begin{align}
&\barr \beta_\bt(y)\,f_k(y) = \alpha_{\bt,\theta}(y) + O(\delta^\varrho)
\\& |\beta_\bt(y)| \ge \delta^\rho.
\end{align}
\end{lemma}


\begin{proof}[Proof of Lemma~\ref{lemma:meromorphicity}]
It suffices to prove this for $k=3$.  The first step will be to establish
smallness of $|\nabla\big(\sum_j a_j (f_j\circ\varphi_j)\big)|$ 
on a significantly large subset of $\sS$.
Consider any ordered triple $(I_j: j\in\{1,2,3\})$
of intervals $I_j$ of length $\delta$ for which 
$\big(\bigcap_{j=1}^3 \varphi_j^{-1}(I_j)\big) \cap \sS \ne\emptyset$.
Choose any point
$\barx\in \big(\bigcap_{j=1}^3 \varphi_j^{-1}(I_j)\big) \cap \sS$.
For any $x\in \bigcap_j \varphi_j^{-1}(I_j)$,
the quantity $\sum_{j=1}^3 a_j(x)\,(f_j\circ\varphi_j)(x)$ can be expressed via
Taylor expansion as
\begin{multline}
\sum_{j=1}^3 a_j(x)\,(f_j\circ\varphi_j)(x) 
= \sum_{j=1}^3 a_j(\barx)\,(f_j\circ\varphi_j)(\barx)
\\
+ (x-\barx)\cdot \nabla \big(\sum_j a_j\,(f_j\circ\varphi_j)\big)(\barx)
+ O(\delta^{-\rho}|x-\barx|^2)
\end{multline}
since $f'_j = O(\delta^{-\rho})$, $f_j = O(1)$, 
the restriction of $f_j$ to $I_j$ has constant derivative, and $a_j,\varphi_j = O(1)$ in $C^2$ norm.
Therefore for any $x\in 
\big(\bigcap_{j=1}^3 \varphi_j^{-1}(I_j)\big) \cap \sS$,
\begin{equation}
\big|(x-\barx)\cdot \nabla \big(\sum_j a_j\,(f_j\circ\varphi_j)(\barx)\big)\big|
= O(\eta + \delta^{2-\rho}).
\end{equation}
Choose $\rho$ sufficiently small that $2-\rho\ge 1+2\varrho$.
Then $\eta + \delta^{2-\rho} = O(\delta^{1+2\varrho} + \delta^{2-\rho}) = O(\eta)$.

For any vector $v\in\reals^2$,
$\{x\in B(\barx,C\delta): |(x-\barx)\cdot v| = O(\eta)\}$
has Lebesgue measure $O(\min(\delta^2,\,\eta|v|^{-1}\delta))$. 
The set $\bigcap_{j=1}^3 \varphi_j^{-1}(I_j)$ is contained in $B(\barx,C\delta)$.
Therefore by choosing
$v = \nabla(\sum_j a_j (f_j\circ\varphi_j)(\barx))$, it follows that either
\begin{equation}
\big|\nabla \big(\sum_j a_j\,(f_j\circ\varphi_j)(\barx)\big)\big|
\le  \delta^{\varrho}  
\end{equation}
or
\begin{equation}
\big|\bigcap_j \varphi_j^{-1}(I_j) \cap \sS\big| 
= O(\eta |v|^{-1} \delta)
= O(\delta^{1+2\varrho}\delta^{-\varrho}\delta)
=O( \delta^{2+\varrho}).  
\end{equation}
If the former inequality holds then
\begin{equation}
\big|\nabla \big(\sum_j a_j\,(f_j\circ\varphi_j)(x)\big)\big|
= O(\delta + \delta^{1-\rho} + \delta^\varrho)
= O( \delta^{\varrho}  )
\end{equation}
as well, since $|x-\barx| = O(\delta)$,
$\norm{a_j}_{C^2} = O(1)$,
$f_j = O(1)$, $f'_j$ is constant on $I_j$,
and we may assume $\rho,\varrho$ to be sufficiently small that $1-\rho\ge\varrho$.

There are $O(\delta^{-2})$ ordered triples $(I_1,I_2,I_3)$
of intervals with $\bigcap_j \varphi_j^{-1}(I_j)\ne\emptyset$.
Summing over all of these, we conclude from the preceding dichotomy that
\begin{equation}
|\sS| \lesssim \delta^{\varrho}
+ \big|\big\{x\in S(\bff,\eta): 
|\nabla \big(\sum_j a_j\,(f_j\circ\varphi_j)(x)\big)|
	= O( \delta^{\varrho})\big\}\big|.  
\end{equation}
It remains to establish an upper bound for the measure of the set
\begin{equation}
S = \big\{x\in S(\bff,\eta): 
|\nabla \big(\sum_{j=1}^3  a_j\,(f_j\circ\varphi_j)(x)\big)| = O(\delta^{\varrho})\big\}.  
\end{equation}

We may assume that $a_3\equiv 1$. Then
\begin{equation}
S \subset \big\{x\in S(\bff,\eta): 
|\sum_{j=1}^2 W_3\big(a_j\,(f_j\circ\varphi_j)\big)(x)| = O(\delta^\varrho) \big\}.
\end{equation}
By sacrificing one scalar inequality, we have eliminated the contribution of $f_3$ from the sum.



Although the number of indices $j$
in play has been reduced from three to two, there are now effectively two arbitrary functions for each index. 
Indeed, write
\begin{equation}
W_3\big(a_j\,(f_j\circ\varphi_j)\big)
= a_j\cdot W_3\varphi_j\cdot (g_j\circ\varphi_j)
+ W_3a_j\cdot (f_j\circ\varphi_j)
\end{equation}
with 
$g_j = f'_j$.
The relationship between $g_j$ and $f_j$ provides limited
information in this context, for the sublevel set, and hence its
images under the mappings $\varphi_j$, could consist of many small
connected components, so that information about $f'_j$ cannot
be integrated to obtain useful macroscale information about $f_j$ itself.\footnote{Nonetheless,
it is clear that in future work, especially for sublevel set inequalities
with more than three summands, the tight relationship between $g_j$ and $f_j$ at one
particular scale should be exploited, in order to eliminate the need for an auxiliary hypothesis
at each recursive step.}
Thus for arbitrary measurable functions satisfying $f_j= O(1)$
and $g_j = O(\delta^{-\rho})$,
we seek to analyze the Lebesgue measure of 
\begin{equation} \label{four_unknowns}
	S_2 = \Big\{x\in B:
\big| \sum_{j=1}^2
\Big[ a_j(x)\cdot W_3\varphi_j(x)\cdot (g_j\circ\varphi_j)(x) 
+ W_3a_j(x)\cdot (f_j\circ\varphi_j)(x) \Big] \big|
= O(\delta^{\varrho}) \Big\}
\end{equation}
without positing any relationship between $f_j$ and $g_j$.
This increase in the number of unknown functions, from three to four, 
is the final difficulty to be overcome. 
A related issue arose in the mesoscale analysis, but we need to proceed differently here.

The reduction in the number of mappings $\varphi_j$
in the definition of the sublevel set, from three to two, is a crucial simplification.
To begin to exploit it, change variables so that $\varphi_1,\varphi_2$
take the form $\varphi_j(x_1,x_2)\equiv x_j$.
Let $S_3$ be the set of all $(\bt,y) = ((t,t'),y) \in\reals^2\times\reals^1$
such that $((t,y),(t',y))\in S_2\times S_2$.
Its three-dimensional Lebesgue measure  satisfies $|S_3|\gtrsim |S_2|^2$,
by the Cauchy-Schwarz inequality.

Let $\barc>0$ be a small constant and
define $S_3^*$ to be the set of all $\bt$ such that
\[ |\{y: (\bt,y)\in S_2\}| \ge \barc|S_2|.\]
If $\barc$ is chosen to be sufficiently small then necessarily $|S_3^*|\gtrsim |S_3|$.

We now follow a path taken in \S20 of \cite{triosc}.
For each $(\bt,y)\in S_3$, \eqref{four_unknowns} can be regarded as a matrix equation
\begin{equation}\label{matrixeqn} 
B(\bt,y)
\begin{pmatrix} g_2(y) \\ \\ f_2(y) \end{pmatrix}
= A(\bt,y) + O(\delta^{\varrho}) \end{equation}
for an unknown quantity $\begin{pmatrix} g_2(y) \\ f_2(y)\end{pmatrix}$,
with right-hand side $A(\bt,y)$ and coefficient matix $B(\bt,y)$ defined by
\begin{equation} \label{amatrixdefn}
A(\bt,y) = - \begin{pmatrix}
a_1(t,y))W_3\varphi_1(t,y)g_1(t) + W_3a_1(t,y)f_1(t)
\\ \\ a_1(t',y))W_3\varphi_1(t',y)g_1(t') + W_3a_1(t',y)f_1(t')
\end{pmatrix}  \end{equation}
and
	\begin{equation}
		B(\bt,y) = 
		\begin{pmatrix}
			a_2(t,y)W_3\varphi_2(t,y) & W_3a_2(t,y)
			\\ \\
			a_2(t',y)W_3\varphi_2(t',y) & W_3a_2(t',y)
		\end{pmatrix}.
	\end{equation}
We introduce a parameter $\theta = (\theta_1,\dots,\theta_4)\in\reals^4$ 
that satisfies $|\theta| = O(\delta^{-\rho})$ and define
\begin{equation} \label{amatrixtheta}
A_*(\bt,\theta,y) = - \begin{pmatrix}
a_1(t,y))W_3\varphi_1(t,y)\theta_1 + W_3a_1(t,y)\theta_2
\\ \\ a_1(t',y))W_3\varphi_1(t',y)\theta_3 + W_3a_1(t',y)\theta_4
\end{pmatrix}.  \end{equation}
The relation \eqref{matrixeqn} is thus rewritten as an instance of the more general relation 
\begin{equation}\label{matrixeqntheta} 
B(\bt,y) \begin{pmatrix} g_2(y) \\ \\ f_2(y) \end{pmatrix}
= A_*(\bt,\theta,y) + O(\delta^{\varrho}). \end{equation}

Define 
\begin{equation} \beta_\bt(y) = \det(B(\bt,y)),\end{equation}
which is a real analytic function of $(\bt,y)$.
According to Lemma~\ref{lemma:det(B)}, the $C^\omega$ function $(\bt,y)\mapsto \beta_\bt(y)$ 
does not vanish identically on any nonempty open set.
This is a consequence of the weak auxiliary hypothesis; it does not rely
on the full strength of the main hypothesis.

Since $(\bt,y)\mapsto \beta_\bt(y)$ is a real analytic function that does not vanish identically, 
	\begin{equation}  \label{betasublevelset}
		|\{(\bt,y): |\beta_\bt(y)| < \delta^{\rho}\}|
		\le C'\delta^{A\rho}
	\end{equation}
for some $A<\infty$ and $C'<\infty$.
If $|S_3|\le 2C'\delta^{A\rho}$ then $|S_2| = O(|S_3|^{1/2}) = O(\delta^c)$
with $c = A\rho/2$. Therefore $|\sS| = O(\delta^c)$,
which is one of the alternative conclusions of Lemma~\ref{lemma:meromorphicity}.
Therefore we may assume that 
$|S_3|\ge 2C'\delta^{A\rho}$. 

Multiplying by the cofactor matrix associated to $B$ 
transforms the equation \eqref{matrixeqn} to
\begin{equation} \label{transformed}
\beta_\bt(y)
\begin{pmatrix} g_2(y) \\ f_2(y) \end{pmatrix}
= \scripta(\bt,\theta(\bt),y) + O(\delta^\varrho)
\end{equation}
with $\scripta(\bt,\theta,y)$ a $C^\omega$ $\reals^2$-valued function of $(\bt,\theta,y)$
and with $\theta(\bt) = (f_1(t),g_1(t),f_1(t'),g_1(t'))$. 
$\scripta$ satisfies $\scripta(\bt,\theta,y) = O(\delta^{-\rho})$ uniformly in $(\bt,y)$,
since $f_1 = O(1)$ and $g_1 = O(\delta^{-\rho})$.
It depends linearly on $\theta$. Therefore we may restrict $\theta$ to a fixed bounded
subset of $\reals^4$ and write $\scripta(\bt,R\theta,y)$ where $0\le R$ is $O(\delta^{-\rho})$.

Restricting attention to the component of the vector equation \eqref{transformed} involving $f_2$ gives
\begin{equation} \beta_\bt(y) f_2(y) = R \alpha_{\bt,\theta}(y) + O(\delta^\varrho) 
\ \forall\,(\bt,y)\in S_3, \end{equation}
where 
\begin{equation}\left\{	\begin{aligned} R &= 1+|(f_1(t),g_1(t),f_1(t'),g_1(t'))| \in[1,O(\delta^{-\rho})),
\\ \theta &= R^{-1}(f_1(t),g_1(t),f_1(t'),g_1(t')), \end{aligned} \right. \end{equation}
and $\alpha_{\bt,\theta}(y)$ is a real analytic function of $(\bt,\theta,y)$.

There exists $\bt$ such that
$S_3^+ = \{y\in\reals: (\bt,y)\in S_3)\}$ satisfies $|S_3^+|\gtrsim |S_3|\gtrsim \delta^{A\rho}$,
where $|\cdot|$ denotes one-dimensional Lebesgue measure
on the left-hand side of this inequality, and three-dimensional measure on the right.
There exists $\theta$, which depends on $\bt$ but not on $y$, such that
\begin{equation} \label{f2betarelation}
\beta_\bt(y)\,f_2(y) = R\alpha_{\bt,\theta}(y) + O(\delta^\varrho) \ \forall\,  y\in S_3^+.
\end{equation}

It will be more convenient in the analysis below to rewrite \eqref{f2betarelation} in the equivalent form
\begin{equation} \label{f2betarelation2}
\barr \beta_\bt(y)\,f_2(y) = \alpha_{\bt,\theta}(y) + O(\delta^\varrho)
\ \forall\, y\in S_3^+
\end{equation}
with $c\delta^\rho\le \barr\le 1$. 

This is a form of the second alternative
conclusion of Lemma~\ref{lemma:meromorphicity},
but the lemma asserts the stronger conclusion that there exists
a subset $S\subset \sS$ of measure $\gtrsim |\sS|^c$ such that \eqref{f2betarelation2}
holds for every $y\in \varphi_2(S)$.
The stronger conclusion is obtained from a simple modification of this proof. 
Indeed, let $c_0>0$ be small and  define $y\in\varphi_2(B)$ to be rich
if the set of all $x\in \sS$ satisfying $\varphi_2(x)=y$ has measure $\ge c_0 |\sS|$.
If $c_0$ is chosen sufficiently small then the set $\sS'\subset \sS$ of all $x\in 
\sS$ such that $\varphi_2(x)$ is rich has measure $\gtrsim |\sS|$.
Apply the above reasoning to $\sS'$, rather than to $\sS$.
Define $S\subset \sS'$ to be the set of all $x\in \sS'$
such that $\varphi_2(x)\in S_3^+$.
Then $|S|\gtrsim |S_3^+|\cdot |\sS|\gtrsim |\sS|^c$ for some $c>0$.
Thus the second conclusion of 
Lemma~\ref{lemma:meromorphicity} has been established,
completing the proof of the lemma.
\end{proof}

\medskip
We next use Lemma~\ref{lemma:meromorphicity} to complete the proof of Proposition~\ref{prop:macroscale}.
By applying Lemma~\ref{lemma:meromorphicity} to the three indices $k=1,2,3$ in succession,
we conclude that either $|S(\bff,\delta^{1+2\varrho})|\le \delta^{c'_0}$,
or there exist a measurable subset $\sS\subset S(\bff,\delta^{1+2\varrho})$
satisfying $|\sS|\gtrsim |S(\bff,\delta^{1+2\varrho})|^C$ and
three three-tuples $(\bt_k,\theta_k, \barr_k)$ of parameters such that for each index $j\in\{1,2,3\}$,
\begin{align}
&\barr_j \beta_{\bt_j}^j(y) f_j(y) = \alpha_{\bt_j,\theta_j}^j(y) + O(\delta^\varrho)
\label{alphaandbeta}
\ \forall\, y\in \varphi_j(\sS)
\\ &|\beta_{\bt_j}^j(y)| \ge\delta^\rho\ \forall\,y\in \varphi_j(\sS)
\label{betalowerbound}
\\& c\delta^\rho\le\barr_j\le 1.
\label{barrrange}
\end{align}


\medskip
We may assume that
$|S(\bff,\delta^{1+2\varrho})| >  \delta^{c'_0}$
for a small constant $c'_0$, for otherwise
the conclusion asserted by Proposition~\ref{prop:macroscale} holds.
Consequently $|\sS|$ satisfies a lower bound of the same type, with an exponent proportional to $c'_0$.
Choose $(\bt_k,\theta_k,\barr_k)$ satisfying the conclusions
\eqref{alphaandbeta}, \eqref{betalowerbound}, \eqref{barrrange}.
Write $\bar\bt = (\bar\bt_1,\bar\bt_2,\bar\bt_3)$,
$\bar\theta= (\bar\theta_1,\bar\theta_2,\bar\theta_3)$,
and $\barr = (\barr_1,\barr_2,\barr_3)$. Then
\begin{multline} \label{gettingthere1}
\prod_{k=1}^3 \barr_k \beta_{\bar\bt_k}^k(\varphi_k(x))
\sum_{j=1}^3 a_j(x)\,(f_j\circ\varphi_j)(x) 
\\ =
\sum_{j=1}^3 a_j(x)\,
\prod_{k\ne j} \barr_k \beta_{\bar\bt_k}^k(\varphi_k(x))
\,\alpha_{\bt_j,\bar\theta_j}^j(\varphi_j(x)) 
\ + \ O(\delta^\varrho)
\end{multline}
for all $x\in \sS$.

Introduce the function 
\begin{equation} \label{Fintroduced}
F(x,\bt,\theta,r) =  F_{\bt,\theta,r}(x))  
= \sum_{j=1}^3 a_j(x)
\prod_{k\ne j} r_k \beta_{\bt_k}^k(\varphi_k(x))
\,\alpha_{\bt_j,\theta_j}^j(\varphi_j(x)), 
\end{equation}
where $\bt = (\bt_1,\bt_2,\bt_3)$, $r = (r_1,r_2,r_3)$,
and $\theta= (\theta_1,\theta_2,\theta_3)$.
The parameter $(x,\bt,\theta,r)$ varies freely over a certain compact domain,
in a neighborhood of which $F$ is an analytic function of $(x,\bt,\theta,r)$.
For the particular parameter $(\bar\bt,\bar\theta,\barr)$
that appears in \eqref{gettingthere1},
\begin{equation} |F(x,\bar\bt,\bar\theta,\barr)| = O(\delta^\varrho)
\ \forall\,x\in\sS. \end{equation}

We claim that $F$ does not vanish identically on any nonempty open set in the
space of parameters $(x,\bt,\theta,r)$.
It is through this claim that the main hypothesis of Theorem~\ref{thm:main} first comes into play.
If $F\equiv 0$ on an open set $U$ of the parameter space, 
then since none of the factors $\beta_{\bt_k}^k$ vanish identically,
$ \sum_{j=1}^3 a_j(x) \,\alpha_{\bt_j,\theta_j}^j(\varphi_j(x)) \equiv 0$ 
would vanish identically in $U$.
By the main hypothesis, $\alpha_{\bt_j,\theta_j}^j\circ\varphi_j$ would vanish identically in
a certain nonempty open subset of $B$,
for each $j\in\{1,2,3\}$, for all $\big((\bt_k,\theta_k): k\in \{1,2,3\}\big)$.

This is a contradiction.
Indeed, \eqref{amatrixtheta} expresses
$\alpha_{\bt_2,\theta_2}^2(y)$ as the second component of the vector
$B(\bt_2,y)^{-1} A_*(\bt_2,\theta_2,y)$, where $\theta_2 = (\theta_{2,1},\dots,\theta_{2,4})\in\reals^4$.
As $\theta_2$ varies over $\reals^4$,
the vector $A_*(\bt_2,\theta_2,y)$ varies over all of $\reals^2$
for generic $(\bt_2,\theta_2)$ since $a_2(t,y)$ and $W_3\varphi_2(t,y)$ vanish nowhere.
Moreover, for generic $(\bt_2,y)$ 
$B(\bt_2,y)$ is invertible 
and hence does not map $\reals^2$ into the subspace of $\reals^2$
defined by the vanishing of the second component. 
Thus it is not the case that 
$\alpha_{\bt_2,\theta_2}^2\circ \varphi_2$ vanishes identically in a certain nonempty open set
for all $\big((\bt_k,\theta_k): k\in \{1,2,3\}\big)$.

\medskip
Choose and fix $x_0\in B$.
Define
\begin{equation}
\scriptg(\bt,\theta,r) =  \sum_{|\alpha|\le N} |\partial^\alpha_x F(x_0,\bt,\theta,r)|^2,
\end{equation}
with $N$ chosen so that there exists $C$ such that
\begin{equation} \label{Fandscriptgcomparable}
C^{-1} \sum_{0\le |\alpha|\le N} |\partial^\alpha_x F(x,\bt,\theta,r)|^2
\le \scriptg(\bt,\theta,r) 
\le C \sum_{0\le |\alpha|\le N} |\partial^\alpha_x F(x,\bt,\theta,r)|^2
\end{equation}
uniformly for all $(x,\bt,\theta,r)$. 
Lemma~\ref{lemma:steinstreetish} guarantees that such an $N$ exists.


Let 
\begin{equation}
\Sigma = \{(\bt,\theta,r): \scriptg(\bt,\theta,r)=0\}.
\end{equation}
In this definition, $r$ is allowed to vary freely over $[0,1]^3$,
while $\bt,\theta$ vary over closed balls in Euclidean spaces. 

$\Sigma$ is a $C^\omega$ variety of positive codimension.
By \L{}ojasiewicz's theorem,
there exist $c,\kappa>0$ such that
\begin{equation} \label{eq:invokedLoj}
\scriptg(\bt,\theta,r) \ge c\distance((\bt,\theta,r),\Sigma)^\kappa
\ \ \forall\,(\bt,\theta,r).
\end{equation}

\begin{lemma} \label{lemma:invokeLoj}
If $0<\rho\le \varrho/4$ then
there exists a constant $c>0$ such that
for all data satisfying the hypotheses of Proposition~\ref{prop:macroscale},
the parameter $(\bar\bt,\bar\theta,\barr)$ that appears in \eqref{gettingthere1} satisfies
\begin{equation}
	\distance((\bar\bt,\bar\theta,\barr),\Sigma)\ge c\delta^{2\rho}.
\end{equation}
\end{lemma}

\begin{proof}
We claim that for any $(\bt,\theta,r)\in\Sigma$, 
at least one of the following three vanishing properties holds.
	\begin{enumerate}
	\item $\alpha^j_{\bt_j,\theta_j}\circ\varphi_j$ vanishes identically  on $B$ for every $j\in\{1,2,3\}$.
	\item  There exists $k\in\{1,2,3\}$ such that
	$\beta_{\bt_k}\circ\varphi_k$ vanishes identically on $B$. 
\item  There exists $k\in\{1,2,3\}$ such that $r_k=0$.
	\end{enumerate}
Indeed, for $(\bt,\theta,r)\in\Sigma$ the function 
$\sum_{|\alpha|\le N} |\partial^\alpha_x F(x,\bt,\theta,r)|^2$
vanishes identically for $x\in B$. In particular, $F(x,\bt,\theta,r)=0$ for every $x\in B$.
If none of the three factors $\beta_{\bt_k}^k\circ\varphi_k$
vanish identically on $B$, and if no $r_k$ vanishes, then 
there exists a nonempty open subset $V\subset B$ on which none of these three vanish.
Multiply by $\prod_{k=1}^3 (\beta_{\bt_j}^j\circ\varphi_j)^{-1}$ in \eqref{Fintroduced}
to conclude that 
\begin{equation*}
	\sum_{j=1}^3 a_j(x)\cdot
	[(\beta_{\bt_j}^j)^{-1} \alpha^j_{\bt_j,\theta_j}]\circ\varphi_j(x)
	=0\ \ \forall\,x\in V.
\end{equation*}
The main hypothesis of Theorem~\ref{thm:main} guarantees that for each $j\in\{1,2,3\}$,
$(\beta_{\bt_j}^j)^{-1} \alpha^j_{\bt_j,\theta_j}$ vanishes identically on $\varphi_j(V)$.
Therefore $\alpha^j_{\bt_j,\theta_j}$ vanishes identically in $\varphi_j(V)$,
hence also in $\varphi_j(B)$ by analytic continuation. This establishes the claim.

Now consider the parameter $(\bar\bt,\bar\theta,\barr)$ in \eqref{gettingthere1}.
According to \eqref{betalowerbound},
$|\beta_{\bar\bt_j}^j(\varphi_j(y))|\ge \delta^\rho$ for every $y\in\sS$, for each index $j$.
Therefore 
$(\bar\bt,\bar\theta,\barr)$ lies at distance $\gtrsim\delta^\rho$ from the
subset of $\Sigma$ on which any $\beta_{\bt_j}^j\circ\varphi_j$ vanishes identically.

Likewise, the parameter $\barr$ in \eqref{gettingthere1} belongs to $[c\delta^\rho,1]^3$
and therefore $(\bar\bt,\theta,\barr)$ 
lies at distance at least $c\delta^\rho$ from the subset of $\Sigma$ on which $r=0$.

For any $y\in\varphi_3(\sS)$, the relation
$\barr_3 \beta^3_{\bar\bt_3}(y) f_3(y) = \alpha_{\bar\bt_3,\bar\theta_3}^3(y) + O(\delta^\varrho)$,
together with the hypothesis that $|f_3(y)|\ge 1$ 
	--- used here for the first time ---
and the lower bounds established in the preceding paragraphs, imply that
	\begin{equation}
		|\alpha_{\bar\bt_3,\bar\theta_3}^3(y)|
		\gtrsim \delta^\rho\cdot \delta^\rho\cdot 1 - O(\delta^\varrho).
	\end{equation}
Since $\rho$ was chosen to satisfy $4\rho \le \varrho$,
$|\alpha_{\bar\bt_3,\bar\theta_3}^3(y)|\gtrsim \delta^{2\rho}$
for every $y\in \varphi_3(\sS)$ provided that $\delta = \eps^\gamma$ is sufficiently small;
this holds if $\eps_0$ is chosen to be sufficiently small. 
Therefore the parameter $(\bar\bt,\bar\theta,\barr)$ in \eqref{gettingthere1}
lies at distance $\gtrsim \delta^{2\rho}$ from the subset of $\Sigma$ on which
$\alpha^3_{\bt_3,\theta_3}$ vanishes identically.
\end{proof}

By Lemma~\ref{lemma:invokeLoj} and \eqref{eq:invokedLoj},
$\scriptg(\bart,\bar\theta,\barr)\gtrsim \delta^{2\kappa \rho}$. 
Therefore by \eqref{Fandscriptgcomparable},
\begin{equation}
\sum_{|\alpha|\le N} |\partial^\alpha_x F(x,\bar\bt,\bar\theta,\barr)|^2
\gtrsim\delta^{C\rho} \ \forall\,x\in \sS.
\end{equation}
Moreover, the $C^{N+1}$ norm of $F$, as a function of $x$,
is bounded above by a finite constant, uniformly in $(\bt,\theta,r)$.
By Lemma~\ref{lemma:quantsublevel}, 
the measure of the sublevel set on which $|F| = O(\delta^{\varrho})$
is $O(\delta^{c\varrho-C\rho})$ for certain constants $c,C\in(0,\infty)$.
Choosing $\rho$ to be sufficiently small, as a function of $\varrho$,
ensures that $c\varrho-C\rho>0$,
completing the proof of Proposition~\ref{prop:macroscale}.
\qed

\section{Conclusion of proof for the nondegenerate case}\label{section:conclusion}
We are now in a position to complete the proof of Theorem~\ref{thm:main} in the main case.

\begin{proof}[Proof of Theorem~\ref{thm:main} in the nondegenerate case]
Let $(\ba,\bPhi)$ satisfy the hypotheses of the nondegenerate case of the theorem.
As was explained in the introduction, we may  also
assume that the curvature of the web in $B$ defined by $\Phi$ does not vanish identically.
Choose small positive quantities $\varrho,\rho,\eps_0$,
with $\rho$ depending on the choice of $\varrho$ and $\eps_0$ depending on both.
Let $\eps\in(0,\eps_0]$.
Choose and fix any $\gamma\in(\tfrac12,1)$.

Let $\bff$ be measurable and satisfy $\norm{\bff}_{L^\infty} = O(1)$
and $|f_3(y)|\ge 1$ for every $y\in\varphi_3(B)$.
By Proposition~\ref{prop:mesoscale},
either $|S(\bff,\eps)|\le \eps^c$,
in which case the proof is complete,
or there exist $\sS\subset S(\bff,\eps)$ 
satisfying $|\sS|\gtrsim |S(\bff,\eps)|^C$
and functions $g_j$
that are affine on intervals of length $\eps^\gamma$ 
and satisfy
$|g_j(\varphi_j(x))-f_j(\varphi_j(x))| \le \eps^{1-\rho}$
for every $x\in \sS$ for each index $j$,
and also $|g_3(\varphi_3(x))| \gtrsim 1$ for every $x\in \sS$.
Moreover, according to \eqref{f*primebound}
these functions can be constructed so that their derivatives satisfy
$|g'_j(y)| \le C\eps^{-\rho}$ for every $y$.
The parameters $c,C$ also depend on $\rho$ but not on $\eps,\bff$.
In the case in which the proof is not already complete,
it suffices to show that $|S(\bg,\eps^{1-\rho})| \le C\eps^c$.

Choose $\varrho$ to satisfy $(1+\varrho)\gamma<1$.
The piecewise affine function
$\bg$ satisfies the hypotheses of Proposition~\ref{prop:macroscale}, whose conclusion is that
if $\rho$ is chosen to be sufficiently small as a function of $\varrho$ in the preceding step
then $|S(\bg,(\eps^\gamma)^{1+\varrho})| \le C\eps^{c\gamma}$.
Moreover, if $\rho,\varrho$ are both sufficiently small  then
$(1+\varrho)\gamma<(1-\rho)$.  Since $\eps\le\eps_0$,
$\eps^{(1+\varrho)\gamma}>\eps^{1-\rho}$
and therefore $S(\bg,\eps^{1-\rho}) \subset S(\bg,(\eps^\gamma)^{1+\varrho})$.
Thus $|S(\bg,\eps^{1-\rho})|\le C\eps^c$, and the proof of Theorem~\ref{thm:main} is complete
in the nondegenerate case.
\end{proof}

\section{The degenerate case} \label{section:degenerate}

We next turn to the degenerate case of Theorem~\ref{thm:main},
showing how the degenerate case is a consequence of the nondegenerate case
via a rather general reduction procedure.
Define $\Sigma\subset\tilde B$ to be the degenerate locus.
$\Sigma$ is the union of the real analytic varieties defined by 
(i) the vanishing of some $a_j(x)$,
and (ii) the linear dependence of gradients $\nabla\varphi_i(x),\nabla\varphi_j(x)$
for some pair of indices $i\ne j$.

To avoid steps in which we divided by some $a_k$
or implicitly divided by some component of some gradient $\nabla\varphi_j$,
we modify various definitions in the analysis, as follows. 
Define the vector fields 
\begin{equation}
	W_k = 
	\frac{\partial\varphi_k}{\partial x_2}\cdot\frac{\partial}{\partial x_1}
	- \frac{\partial\varphi_k}{\partial x_1}\cdot\frac{\partial}{\partial x_2}
\end{equation}
so that $W_k(\varphi_k)\equiv 0$, 
$W_k\in C^\omega$,
and $W_k$ does not vanish identically.
The mappings $\psi_j^\eps$ are now defined in terms of this modified $W_3$.
One can likewise construct real analytic vector fields
$V_j^\eps$, for $j=1,2$, that satisfy $V_j(\psi_j^\eps)\equiv 0$
but do not vanish identically.

The definitions of the quantities $b^\eps_{n}$ 
should likewise be modified to eliminate denominators.
For odd $n$, define $b^\eps_n$ 
by changing its definition \eqref{bndefn} to 
\begin{multline}
b_n^\eps(\bt,z) = - 
{a_1(\Theta^\eps_{n,z}(t_1,\dots,t_n))} \cdot {a_2(\Theta_{n-1,z}^\eps (t_1,\dots,t_{n-1}))} 
\\ \cdot 	{a_1(\Theta_{n-2,z}^\eps(t_1,\dots,t_{n-2}))} 
\cdots {a_1(\Theta_{1,z}^\eps(t_1))},
\end{multline}
the product of the numerators only in \eqref{bndefn},
incorporating also the minus sign; and
define $\tilde b^\eps_n$ to be the product of the denominators in \eqref{bndefn}
For even $n$ make the corresponding definitions, 
based on the numerators and denominators in \eqref{bdefinition2}.

The relation \eqref{Thirdly} was 
$(g_j\circ\psi_j^\eps)(\Theta^\eps_{n,z}(\bt)) = g_1(z) b_n^\eps(\bt,z) + O(\eps)$;
with the revised definition of $b_n^\eps$ it becomes 
\begin{equation} \label{Thirdlymodified}
\tilde b_n^\eps(\bt,z)(g_j\circ\psi_j^\eps)(\Theta^\eps_{n,z}(\bt)) = g_1(z) b_n^\eps(\bt,z) + O(\eps),
\end{equation}
with $n=3$ for $j=2$ and $n=4$ for $j=1$.
Likewise, modify the definition \eqref{scriptfdefn} of $\scriptf_z(\bt,s,s')$
by multiplying by a real analytic function of $(z,\bt,s,s')$ that does not vanish 
vanish identically on any open set, and that suffices to cancel all denominators
implicitly present.  With this modification, there exist $C,\tau\in(0,\infty)$ satisfying
\[ \big|\big\{(z,\bt,s,s'): |\scriptf_z(\bt,s,s')|<\delta\big\}\big|
\le C\delta^\tau\ \forall\,\delta>0.\]

Let $c_1,c_2>0$ be small parameters.
The set of all $x\in B$ lying within
distance $\eps^{c_1}$ of the positive codimension real analytic variety $\Sigma$ 
has Lebesgue measure $O(\eps^{c_3})$
for some $c_3=c_3(c_1)>0$.
Thus $B$ can be expressed as the union of a set of measure $O(\eps^{c_3})$
and of $O(\eps^{-2c_2})$ balls $B_\nu$, with each ball having radius $\eps^{c_2}$
and lying at distance $\gtrsim \eps^{c_1}$ from $\Sigma$.
By \L{}ojasiewicz's theorem,
each coefficient $a_j$ satisfies $|a_j|\gtrsim \eps^{c_4}$
at every point in each $B_\nu$, where $c_4 = Cc_1$.
Lower bounds of this type also hold for $|\nabla\varphi_j|$
and $\det(\nabla\varphi_i,\nabla\varphi_j)$ on each $B_\nu$ 
for any pair of distinct indices $i,j$.

Choosing $c_2$ to be $C_5c_1$ for a sufficiently large constant $C_5$
and using these bounds, the proof of Proposition~\ref{prop:microscale}
can be repeated on each ball $B_\nu$
with a loss of a factor that is $O(\eps^{-Cc_1-Cc_2})$ for some constant $C<\infty$.
A typical point is that an upper bound $O(1)$ for the $C^2$ norm of each $\varphi_k$
together with a lower bound $c\eps^{c_5}$ for $\det(\nabla\varphi_i,\nabla\varphi_j )$
for distinct indices $i,j)$
allow the conclusion that the restriction of each mapping $x\mapsto (\varphi(x),\varphi_j(x))$
to $B_\nu$ is a diffeomorphism, for all sufficiently small $\eps>0$, if $c_2 < 2c_5$. 
This legitimizes passing from a lower bound for the Lebesgue measure of a set $S\subset B$
to a lower bound of the form $c|S|^2$ 
for the Lebesgue measure of $\tilde S = \{(x,x')\in S\times S: \varphi_j(x)=\varphi_j(x')\}$,
as was done repeatedly in the analysis of the nondegenerate case.

Summing over all balls $B_\nu$ and taking the exceptional
set of points near $\Sigma$ into account,
the conclusion is that if $c_1,c_2$ are chosen sufficiently small relative
to the exponent $\tau$ in the upper bound for $|\scriptf_z(\bt,s,s')|$, then either
\begin{equation} \label{gettingfussy} |S(\bff,\eps)| \le C\eps^c \eps^{-Cc_1-Cc_2} + C\eps^{c_3}, \end{equation}
or there exist a measurable set $S$ contained in $S(\bff,\eps)\cap B_\nu$
for some index $\nu$ and functions $f_j^*$
that satisfy the conclusions \eqref{microscaleconclusion2A} and \eqref{microscaleconclusion2B}
of Proposition~\ref{prop:microscale}.
Choose $c_1,c_2$ to also be sufficiently small that the net exponent
$c-Cc_1-Cc_2$ in \eqref{gettingfussy} is positive.

The same reasoning applies to the mesoscale and macroscale analyses,
completing the proof of Theorem~\ref{thm:main} in the general case.

\part{Complements and further results} 

\section{Reformulation of the main hypothesis} \label{section:smoothcase}

We establish an equivalent reformulation of the main hypothesis. 
This provides a straightforward extension of that hypothesis to $C^\infty$ data $(\ba,\bPhi)$.
To any $\bff=(f_1,f_2,f_3)$ associate the function
\begin{equation} G_\bff(x) = \sum_{j=1}^3 a_j(x)(f_j\circ\varphi_j)(x).\end{equation}

\begin{proposition} \label{prop:equivalent_hypothesis}
Let $\ba,\bPhi$ be real analytic in a neighborhood $\tilde B$
of a closed ball $B\subset\reals^2$.
Assume that none of the mappings $\varphi_j$ are constant on $B$,
and that there exists no pair of distinct indices $i\ne j\in\{1,2,3\}$ for which
$\nabla\varphi_i$ and $\nabla\varphi_j$ are everywhere linearly dependent.
Then the following are equivalent.
\begin{enumerate}
\item
$(\ba,\bPhi)$ satisfies the main hypothesis.
\item
For each point $\barx\in B$ at which $(\ba,\bPhi)$ is nondegenerate,
there exists  a positive integer $M$ such that for any 
$\bff\in C^M$ defined in a neighborhood of $\Phi(\barx)$, 
if $G_\bff$ vanishes to order $M$ at $\barx$
then every $f_j$ vanishes to order $M$ at $\varphi_j(\barx)$.
\item
There exist $M\in\naturals$, $C<\infty$, and $ \tau>0$
and a real analytic variety $\Sigma\subset \tilde B$
of positive codimension
such that for every $\barx\in B$
and every tuple of functions
$\bff\in C^M$ defined in a neighborhood of $\Phi(\barx)$, 
\begin{equation}
\dist(\barx,\Sigma)^\tau
\sum_{j=1}^3 \sum_{0\le k\le M} \big|(\frac{d^k}{dy^k} f_j)(\varphi_j(\barx))\big|
\le C
\sum_{0\le |\alpha| \le M}
\Big|\frac{\partial^\alpha}{\partial x^\alpha}G_\bff(\barx) \Big|.
\end{equation}
\end{enumerate}
\end{proposition}

The proof developed below has subsequently been applied in Proposition~7.1 of
\cite{christzhou} to prove a corresponding result for inequalities involving two indices $j$,
rather than three. 

Condition (3) of the Proposition 
directly implies (2), which directly implies (1).
Conversely, (1) implies (3)
by the next lemma and \L{}ojasiewicz's theorem.
We say that a function vanishes to order $N$ at a point
if the function, and all of its partial derivatives of orders
less than or equal to $N$ vanish at that point.

\begin{lemma} \label{lemma:finite_type}
Let $(\ba,\bPhi)$ satisfy the main hypothesis and be nondegenerate.
There exists  
a positive integer $M$ such that for any $\barx\in B$ and any 
$\bff\in C^M$ in a neighborhood of $\Phi(\barx)$, 
if $G_\bff$ vanishes to order $M$ at $\barx$
then every $f_j$ vanishes to order $M$ at $\varphi_j(\barx)$.
\end{lemma}


Conversely, if $(\ba,\bPhi)$ satisfies the conclusion of Lemma~\ref{lemma:finite_type}, 
then the main hypothesis of Theorem~\ref{thm:main} is satisfied.

\begin{proof}[Proof of Lemma~\ref{lemma:finite_type}]
	Let $U\subset B$ be open and nonempty.
Let $\bff = (f_1,f_2,f_3)$, with each $f_j\in C^2(\varphi_j(U))$. 
Suppose that $G_\bff$ vanishes identically on some open set. 
With $b_j = a_j/a_3$,
\[ 0 =  W_3(a_3^{-1}G_\bff)
= \sum_{j=1}^2 \big[W_3(b_j)(f_j\circ\varphi_j) + b_j W_3(\varphi_j)(f'_j\circ\varphi_j)
\big]. \]
By dividing by $b_2 W_3(\varphi_2)$ and then applying the vector field $W_2$,
we obtain a linear relation, with $C^\omega$ coefficients,
between $f_j\circ\varphi_j$ and $f'_j\circ\varphi_j$
for $j=1,2$ and $f''_1\circ\varphi_1$.
The coefficient of $f''_1\circ\varphi_1$ vanishes nowhere, since $W_2(\varphi_1)$ vanishes nowhere.
Thus we may express $f''_1\circ\varphi_1$
as a $C^\omega$--linear combination of those $f_j\circ\varphi_j$ and $f'_j\circ\varphi_j$
with $j\in\{1,2\}$.

Since the roles of the indices can be freely permuted,
there is an expression of this type for $(f''_j\circ\varphi_j: j\in\{1,2,3\})$
as a $C^\omega$ linear combination
of $((f_j\circ\varphi_j,f'_j\circ\varphi_j): j\in\{1,2,3\})$.
If each $f_j\in C^\infty$, this relation can be differentiated arbitrarily many times.
We conclude that if $G_\bff$ vanishes identically on some open set $U$, 
then for each $x\in U$ and for every $n\ge 2$, the $n$-th derivative 
$f_j^{(n)}(\varphi_j(x))$ is determined 
as a linear combination of $(f_k(\varphi_k(x)),f'_k(\varphi_k(x))): k\in\{1,2,3\})$. 
The coefficients in these linear combinations are real analytic functions of $x$.

Likewise, if $G_\bff$ vanishes to order $N$ at some point $x$,
and if each $f_j\in C^N$, then the $N$-jet of each $f_j$ at $\varphi_j(x)$
is uniquely determined as a linear function of the $1$-jets of 
the three functions $f_k$ at $\varphi_k(x)$.
If $\bff$ is a formal power series at $\Phi(x)$,
by which we mean that each $f_j$ is a formal power series at $\varphi_j(x-\barx)$,
then as a formal power series in $x-\barx$ , $G_\bff$ vanishes identically
if and only if the full formal power series of each $f_j$ at $\varphi_j(\barx)$
equals a certain linear function of the $1$-jets of the three functions $f_k$ at 
$\varphi_k(\barx)$.

Composition, multiplication, and addition of formal power series are well-defined.
We claim that if $\bff$ is a formal power series at $\Phi(\barx)$, with the property that  $G_\bff=0$ as a formal
power series in $x-\barx$, then $\bff$ has positive radius of convergence. 
That is, any formal power series solution of $G_\bff=0$ is real analytic.
This claim will be shown below, and assumed here.

Consider any $\barx\in B$. Define $V_N$ to be the vector space of all $N$-jets
$\bff$ at $\Phi(\barx)$ such that the associated $N$-jet $G$ vanishes.
These are nested in the natural sense. That is, 
the image of $V_N$ under the natural map from $N$-jets to $N-1$-jets
is contained in $V_{N-1}$.
This map is injective for every $N\ge 2$, since as we have shown,
partial derivatives of order $N$ are determined as linear functions of
lower-order partial derivatives.
Thus $\dim(V_N) \le \dim(V_{N-1})$ for every $N\ge 2$. 
Therefore either there exists $M$ such that $V_M = \{0\}$,
or there exist $M$  and $D$ such that $V_N$ has dimension $D$ for every $N\ge M$.

Let $\barx\in B$, and $N\ge 2$. 
If $V_N(\barx)=\{0\}$ then there exists a neighborhood $V$ of $\barx$ such that
$V_N(x)=\{0\}$ for every $x\in V$.
Indeed, $V_k(x)=\{0\}$ if and only if a certain linear mapping is injective,
and injectivity is stable under small perturbations.


On the other hand, if $\dim(V_N(\barx))=\dim(V_M(\barx))>0$ for every $N\ge M$
then for each $N\ge M$, every $N$-jet $\bg$ satisfying $G_\bg(\barx)=0$
in the sense of $N$-jets, extends in a unique way to an $N+1$--jet with the 
corresponding property. Beginning with a nonzero element of $V_M(\barx)$
and iterating this process, produces a nonzero formal power series solution
of $G_{\bg^*} = 0$.
By the claim announced above, the formal power series $\bg^*$ defines a $C^\omega$ solution $\bff$
of $G_\bff\equiv 0$ in some neighborhood of $\barx$. Since the $M$-jet of $\bff$ at $\barx$ is nonzero,
$\bff$ does not vanish identically. Therefore $\bff$ violates the main hypothesis
of Theorem~\ref{thm:main}.
Thus only the case in which $V_N(\barx)=\{0\}$ for all sufficiently large $N$, can arise. 

We have seen that triviality of $V_M(x)$ is an open condition in $x$.
Therefore by compactness, there exists a single $M$ satisfying this conclusion for every point
$x$ in the closed ball $B$.
\end{proof}

It remains to prove the claim. Choose any real analytic curve $(-1,1)\owns t\mapsto \gamma(t)\in B$
with $\gamma(0)=\barx$ and $(\varphi_j\circ\gamma)'(t)\ne 0$ for every $(j,t)$. 
Consider any function
\[ u(t) = ((g_k\circ\varphi_k\circ\gamma(t),g'_k\circ\varphi_k\circ\gamma(t)): k\in\{1,2,3\})\]
where each $g_k$ is an unknown function defined in a neighborhood of $\varphi_k(\barx)$.
By the above discussion, 
if $\bff$ is smooth and $G_\bff\equiv 0$ in a neighborhood of $\barx$
then the $\reals^2$--valued function $u$ defined by 
\[ u(t) = ((f_k\circ\varphi_k\circ\gamma(t),f'_k\circ\varphi_k\circ\gamma(t)): k\in\{1,2,3\})\]
satisfies a certain linear ordinary differential equation
\begin{equation} \label{uODE} \frac{d}{dt}u(t) = H(u(t))\end{equation}
in a neighborhood of $t=0$.
The coefficients of this equation are real analytic functions of $t$,
and are determined by $(\ba,\bPhi)$ and the choice of $\gamma$.


If $\bff$ is a formal power series at $\Phi(\barx)$, and if
$G_\bff=0$ in the sense of formal power series, then
the formal power series $v$ defined by
\[ v(t) = ((f_k\circ\varphi_k\circ\gamma,f'_k\circ\varphi_k\circ\gamma): k\in\{1,2,3\})\]
is a solution of the same differential equation \eqref{uODE}, in the sense of formal power series.

Since \eqref{uODE} is a linear ordinary differential equation with analytic coefficients 
and nonvanishing leading-order coefficient,
there exists a real analytic solution $u$ of \eqref{uODE} in a neighborhood of $t=0$,
with the initial condition $u(0) = 
((f_k\circ\varphi_k\circ\gamma(0),f'_k\circ\varphi_k\circ\gamma(0)): k\in\{1,2,3\})$.
This solution is unique. Its Taylor series at $t=0$ must coincide with
the formal power series $v$. 
Therefore the formal power series $t\mapsto f_k\circ\varphi_k\circ\gamma$
is convergent in a neighborhood of $t=0$, for each $k\in\{1,2,3\}$.
Since $\varphi_k\circ \gamma$ is a locally invertible mapping from $\reals^1$ to $\reals^1$ 
with real analytic inverse, the formal power series $f_k$ are likewise convergent. \qed

\medskip
We next formulate a variant, which bears the same relationship to a hypothesis used in \cite{triosc}
as does Lemma~\ref{lemma:finite_type} to the main hypothesis of Theorem~\ref{thm:main}.
In Lemma~\ref{lemma:finite_type2}, the notation is
$ G_\bff(x) = \sum_{j=1}^3 (f_j\circ\varphi_j)(x)$;
the coefficients $a_j$ are identically equal to $1$.

\begin{lemma} \label{lemma:finite_type2}
Let $\varphi_j$ be real analytic mappings from $\tilde B$ to  $\reals^1$.
Suppose that for any $i\ne j$, $\nabla\varphi_i$ and $\nabla\varphi_j$
are linearly independent at each point of $B$.
Assume that for any connected open set $U\subset\tilde B$
and any $\bff\in C^\omega(\Phi(U))$ satisfying $G_\bff\equiv 0$ in $U$,
the functions $f_j\circ\varphi_j$ are all constant in $U$.
There exist a positive integer $M$ and a positive real number $C$
such that for any $\barx\in B$, for any 
$\bff\in C^M$ in a neighborhood of $\Phi(\barx)$ satisfying $f_j(\varphi_j(\barx))=0$ for each index $j$,
\begin{equation}
	\sum_{j=1}^3 \sum_{0<k\le M} |f_j^{(k)}(\varphi_j(\barx))|
	\le C\sum_{0<|\alpha|\le M} \big|\frac{\partial^\alpha G_\bff}{\partial x^\alpha}(\barx) \big|.
\end{equation}
\end{lemma}

Here $f_j^{(k)}$ is the $k$-th order derivative of $f_j$.
The sum on the right-hand side extends over the indicated multi-indices $\alpha$,
with $|\alpha| = |(\alpha_1,\alpha_2)| = \alpha_1+\alpha_2$.
The conclusion for a fixed $M,\barx$ is equivalent to the assertion that if $G_\bff-G_\bff(\barx)$
vanishes to order $M$ at $\barx$ then each $f_j-f_j(\varphi_j(\barx))$ vanishes to order $M$
at $\varphi_j(\barx)$. With $M$ held fixed, validity of this assertion at $\barx_0$
implies its validity at all $\barx$ in some neighborhood of $\barx_0$,
and implies the inequality with $C$ independent of $\barx$ in that neighborhood.

The proof of Lemma~\ref{lemma:finite_type2}
is almost identical to that of Lemma~\ref{lemma:finite_type},
and is consequently omitted. \qed



\section{Supplementary remarks} \label{section:furtherremarks}

We formulate a more general version of Proposition~\ref{prop:resurrected},
which is proved in the same way as that proposition.
It will not be used in this paper, so details of its proof are omitted.

\begin{proposition} \label{prop:alt2sublevel}
Let $B\subset\reals^3$ be a closed ball of positive, finite radius,
and let $\tilde B$ be an open neighborhood of $B$.
Let $b_j:\tilde B\to\reals$ be nowhere vanishing $C^\omega$ functions.
Let $\psi_1,\psi_2:\tilde B\to\reals^1$ be $C^\omega$
submersions whose gradients are everywhere linearly independent.
Assume that there exist nowhere vanishing  $C^\omega$ vector fields $V_j$
in $\tilde B$ that satisfy $V_j(\psi_j)\equiv 0$, and 
satisfy the bracket condition at each point of $B$.
Let sublevel sets $S(\bg,\eps)$ be defined as in \eqref{notation:resurrected}.

There exist $\bh = (h_1,h_2)\in C^\omega$ 
satisfying $\sum_{j=1}^2 b_j(x) (h_j\circ\psi_j)(x)\equiv 0$
in $B\times (0,\infty)$,
and $C<\infty$, with the following property.
For any $\eps>0$ and any measurable functions $g_j:\psi_j(B)\to\reals$
there exist $\alpha\in\reals$
and a measurable subset $S\subset S(\bg,\eps)$
satisfying \[ |S|\ge C^{-1} |S(\bg,\eps)|^C\]
such that
\begin{equation} \label{digressionconclusion}
|g_j(y) - \alpha h_j(y)| \le C \eps
\ \ \forall\, y\in \psi_j(S)\ \forall\, j\in\{1,2\}.
\end{equation}
\end{proposition}

In particular, the vector space of all solutions of
$\sum_{j=1}^2 b_j\cdot (g_j\circ\psi_j)\equiv 0$
in any connected nonempty open set is a vector space
of dimension $\le 1$. If the dimension equals $0$
then \eqref{digressionconclusion} states that $|g_j(y)| = O(\eps)$
for every $y\in\psi_j(S)$.


\medskip
The next result may be useful in the treatment, in future work, of those rare data $(\ba,\bPhi)$ for which
the main hypothesis holds, but the auxiliary hypothesis does not.

\begin{observation} \label{observe:continuation}
Let $\ba,\bPhi$ be $C^\omega$. Assume that each $a_j$ vanishes nowhere,
and that $\nabla\varphi_1,\nabla\varphi_2$ are everywhere linearly independent.
Let $B$ be connected and simply connected.
Let $U\subset B\times(0,\infty)$ be a nonempty connected open set.
Let $g_j \in C^\omega(\psi_j(U))$
satisfy $\sum_{j=1}^2 (a_j\circ\pi)\cdot (g_j\circ\psi_j) \equiv 0$ in $U$. 
Then each function $g_j$ extends to a real analytic function in $\varphi_j(B)\times(0,\infty)$.
\end{observation}

By analytic continuation, then, the extensions satisfy
$\sum_{j=1}^2 (a_j\circ\pi)\cdot (g_j\circ\psi_j) \equiv 0$ in $B$. 
For equations of this type,
existence of local solutions is equivalent to existence of global solutions.

\begin{proof}
Let $U'\subset\reals^2$
be a nonempty connected open subset such that $U'\times I\subset U$
for some nonempty open interval $I\subset\reals$.
We have shown that there exists $\tau\in\reals$
such that for $j=1,2$, $g_j(y,t)= h_j(y)t^\tau$
on $U'\times I$ for some $h_j\in C^\omega(\varphi_j(U'))$.
These functions $h_j$ satisfy $\sum_{j=1}^2  b_j(x)(h_j\circ\varphi_j)(x)\equiv 0$
in $U'$, with $b_j(x) = a_j(x)(W_3\varphi_j(x))^\tau$ being nowhere vanishing real 
analytic functions.

It suffices to show that each $h_j$ extends to
a $C^\omega$ function $\tilde h_j:\varphi_j(B)\to\reals$; by the 
connectivity of $B$, it then follows that $\sum_{j=1}^2 b_j\cdot(\tilde h_j\circ\varphi_j)$
vanishes identically on $B$.

For $j\in\{1,2\}$ let $W_j$ be nowhere vanishing $C^\omega$ vector fields
satisfying $W_j(\varphi_j)\equiv 0$.
Let $(i,j)$ be any permutation of $(1,2)$.
Consider any curve $\Gamma$
parametrized by $s\mapsto e^{sW_j}\barx$ with $\barx\in U'$.
Suppose that $e^{sW_j}\barx\in B$ for every $s$ belonging to 
some open interval $J\subset\reals$ that contains $0$.
Then 
\[(h_i\circ\varphi_i)(x) \equiv -a_j(x)a_i(x)^{-1} (h_j\circ\varphi_j)(x)\]
for all $x\in\Gamma$ in some neighborhood of $\barx$.
That is,
\[(h_i\circ\varphi_i)(e^{sW_j}\barx) \equiv 
-a_j(e^{sW_j}\barx)a_i(e^{sW_j}\barx)^{-1} (h_j\circ\varphi_j)(e^{sW_j}\barx)\]
for all $s$ in some neighborhood of $0$.
The right-hand side is a well-defined real analytic function
of $s$ on the full interval $J$, since
$\varphi_j(e^{sW_j}\barx)$ is constant and consequently
$(h_j\circ\varphi_j)(e^{sW_j}\barx)$ is likewise constant.
Thus the function $s\mapsto (h_i\circ\varphi_i)(e^{sW_j}\barx)$
extends to a real analytic function on the full interval $J$.
The mapping $\varphi_i$ restricts to an invertible real analytic
mapping from $J$ to the interval $\{e^{sW_j}\barx: s\in J\}$.
Therefore $h_i$ extends to a real analytic function
on that same interval.

For any two points in $B$ there exists a continuous path in $B$
from one point to the other, consisting of a concatenation of
integral curves of $W_1$ and $W_2$. Since each $h_k\circ\varphi_i$
is constant along integral curves of $W_k$,
it follows that $h_i$ extends to a real analytic function on 
the full interval $\varphi_j(B)$.
\end{proof}

\medskip
The equation $\sum_{j=1}^3 a_j\cdot (f_j\circ\varphi_j)=0$ shares this extension property.

\begin{observation} \label{observation:continuation}
Let $\ba,\bPhi$ be $C^\omega$. Assume that each $a_j$ vanishes nowhere,
and that $\nabla\varphi_i,\nabla\varphi_j$ are everywhere linearly independent
for all $i\ne j\in\{1,2,3\}$.
Let $B\subset\reals^2$ be connected and simply connected,
and let $U\subset B$ be nonempty, connected, and open.
Let $f_j:\varphi_j(U)\to\reals$ be $C^\omega$ and
satisfy $\sum_{j=1}^3 a_j\cdot (f_j\circ\varphi_j) \equiv 0$ in $U$. 
Then each $f_j$ extends to a real analytic function in $\varphi_j(B)$.
\end{observation}

\begin{proof}
Assume without loss of generality that $a_3\equiv 1$.
Apply $W_3$ to obtain
\[ \sum_{j=1}^2 \big[
W_3a_j\cdot (f_j\circ\varphi_j)
+ a_j\cdot W_3\varphi_j\cdot (f'_j\circ\varphi_j)
\big]=0 \ \text{ in $U$.}\]
On any integral curve of $W_2$ in any open subset of $B$ in which
both functions $f_j\circ\varphi_j$ are analytic,
the two terms with $j=2$ are constant, and this equation
becomes an inhomogeneous linear ordinary differential equation for the restriction of $f_1$
to an open interval. 
The coefficients of this differential equation are analytic,
and the leading coefficient, which is $a_1\cdot W_3\varphi_1$ composed with
an invertible analytic mapping, vanishes nowhere.
Therefore $f_1$ continues analytically to the image under $\varphi_1$
of any integral curve of $W_2$ that passes through $U$.

The proof of analytic extendability for $j=1,2$ is completed as in the final paragraph
of the preceding argument.  For $j=3$, it suffices to permute the indices.
\end{proof}

\begin{observation}
Under the hypotheses of Observation~\ref{observation:continuation},
the vector space of all real analytic solutions $\bff$ 
of the equation $\sum_{j=1}^3 a_j\cdot(f_j\circ\varphi_j)\equiv 0$ in $B$
has finite dimension. 
\end{observation}

\begin{proof}
The vector space in question has dimension less than or equal to $3$.
Indeed, consider any point $\barx\in B$,
and consider the integral curve $\Gamma$ of $W_2$ that passes through $\barx$.
The restriction of $f_1\circ\varphi_1$ to $\Gamma$ --- equivalently,
the restriction of $f_1$ to the subinterval $\varphi_1(\Gamma)$ of its domain ---
is uniquely determined by the three scalar quantities 
$f_1(\varphi_1(\barx))$, $f_2(\varphi_2(\barx))$,
and $f'_2(\varphi_2(\barx))$, as well as by the fixed data $\ba,\bPhi$,
through the ordinary  differential equation that it satisfies.
The restriction of $f_2\circ\varphi_2$ to $\Gamma$ is constant,
so both functions $f_j\circ\varphi_j|_\Gamma$ are determined
by the initial three quantities.

Consider the integral curve $\Gamma'$ of $W_1$ that passes through any point $\barx'\in\Gamma$.
The restriction of $f_2\circ\varphi_2$ to $\Gamma'$
is uniquely determined  by
the three scalar quantities $f_2(\varphi_2(\barx')) = f_2(\barx)$, 
$f_1(\varphi_1(\barx'))$,
and $f'_1(\varphi_1(\barx'))$,
and thus by 
$f_1(\varphi_1(\barx))$, $f_2(\varphi_2(\barx))$,
and $f'_2(\varphi_2(\barx))$. 
Thus both functions $f_j\circ\varphi_j$ are uniquely determined
at $e^{s_1W_1}e^{s_2W_2}\barx$
for all $(s_1,s_2)$ in an open neighborhood of $(0,0)$.
This reasoning can be iterated.
Thus for $j=1,2$, $f_j$ is uniquely determined in a neighborhood of $\varphi_j(\barx)$, 
and hence by analytic continuation in the connected domain $\varphi_j(B)$.

The same conclusion is reached for $f_3$ by using the relation
$\sum_{j=1}^3 a_j\cdot (f_j\circ\varphi_j) =0$ to
solve for $f_3(\varphi_3(\barx))$,
then repeating the above analysis with the indices permuted.
\end{proof}


\section{Approximation by regular functions}

The proof of Theorem~\ref{thm:main} has implicitly established the following
intermediate result, in which the main hypothesis of Theorem~\ref{thm:main}
is not assumed; nor is any lower bound assumed for $|\bff|$.
We write $\alpha_{j,\omega}(y) = \alpha_j(y,\omega)$
and $\beta_{j,\omega}(y) = \beta_j(y,\omega)$
to emphasize that these are regarded as functions of $y$
that depend on a parameter $\omega$.
 
\begin{theorem} \label{thm:regularity}
Let $B\subset\reals^2$ be a closed ball, and let $\tilde B$ be an open neighborhood of $B$.
Let $\ba,\bPhi\in C^\omega(\tilde B)$.
Assume that none of the mappings $\varphi_j$ are constant on $B$,
and that for each $i\ne j\in\{1,2,3\}$,
$\nabla\varphi_i(x)$ and $\nabla\varphi_j(x)$ are not everywhere linearly dependent on $B$.
Assume that the curvature of the web defined by $\Phi$ does not vanish identically.
Let $(\ba,\bPhi)$ be nondegenerate and satisfy the auxiliary hypothesis. 

There exist $N<\infty$, a compact set $\Omega\subset\reals^N$,
open neighborhoods $U$ of $B$ and $\tilde\Omega$ of $\Omega$,
and $C^\omega$ functions $\alpha_j,\beta_j:\varphi_j(U)\times\tilde\Omega\to\reals$,
such that for each $\rho>0$ there exist $c,C,\tau\in(0,\infty)$
with the following property. 

Let $\eps>0$ be arbitrary.
Let $\bff$ be an ordered triple of Lebesgue measurable functions
satisfying $\norm{f_j}_{L^\infty}\le 1$.
Either
\begin{equation} |S(\bff,\eps)|\le C'\eps^{\tau} \end{equation}
or there exist a measurable set $\sS\subset S(\bff,\eps)$ satisfying
\begin{equation} |\sS|\ge c|S(\bff,\eps)|^C \end{equation}
and a parameter $\omega\in \Omega$ such that for each $j\in\{1,2,3\}$,
\begin{equation}
\left\{ \begin{aligned}
& |f_j(y)-\frac{\alpha_{j,\omega}(y)}{\beta_{j,\omega}(y)}| 
\le C\eps^{1-\rho} \ \ \forall\,y\in\varphi_j(\sS)
\\
&|\beta_{j,\omega}(y)| \ge \eps^\rho\ \ \forall\,y\in\varphi_j(\sS).
\end{aligned} \right.
\end{equation}
\end{theorem}

Results of this type may play a role in the removal of auxiliary hypotheses
such as those in Theorems~\ref{thm:main} and \ref{thm:linearmappings}.

\section{A special case with arbitrarily many summands}  \label{section:linearmappings}

In this section we consider sublevel sets $S(\bff,\eps)$ 
for sums $\sum_{j=1}^n a_j(x)\,(f_j\circ\varphi_j)(x)$
with an arbitrary number $n$ of terms, in the special case in which all mappings $\varphi_j$
are linear. Such a situation is specified by
a datum \[\scriptd = (n,\bPhi,\ba) =  (n,\{\varphi_j: 1\le j\le n\},\{a_j: 1\le j\le n\}).\]

Below, we define a family of derived data $\scriptd^*$
associated to a datum $\scriptd$. This family is finite, but has cardinality
very roughly of size $2^{2^{2^{\cdots}}}$ with an exponential tower of height $n-2$.
To $\scriptd$ is associated the linear functional
$\sum_{j=1}^n a_j(x)\,(f_j\circ\varphi_j)(x)=0$;
to each derived datum $\scriptd^*$ is associated a corresponding functional.

\subsection{Formulation of the inequality}
The following result applies to sublevel set inequalities with arbitrarily many summands.

\begin{theorem}  \label{thm:linearmappings}
Let $B\subset\reals^2$ be a closed ball of positive, finite radius.
Let $\tilde B\subset\reals^2$ be an open neighborhood of $B$.
Let $n\ge 1$. 
For each $j\in\{1,2,\dots,n\}$ let $\varphi_j:\reals^2\to\reals^1$
be a surjective linear mapping, and
let $a_j:\tilde B\to\complex$ be real analytic. 
Let $\scriptd  = (n,\{\varphi_j\},\{a_{j}\})$.

Suppose that for any two distinct indices $i\ne j\in J$,
$\nabla\varphi_i,\nabla\varphi_j\in \reals^2$ are linearly independent.
Suppose that none of the coefficients $a_{j}$ vanish identically in $B$.
Suppose that for any nonempty open set $U\subset\tilde B$,
and for $\scriptd$ as well as for any datum $\scriptd^*$ associated to $\scriptd$, 
any real analytic solution $\bff$ of 
the linear equation associated to $\scriptd^*$
vanishes identically in $\bPhi(U)$.

Then there exist $C<\infty$ and $\tau>0$ such that for any Lebesgue measurable $\bff$
and any $\eps>0$,
\begin{equation}
|\{x\in S(\bff,\eps): \sum_j  |f_{j}\circ\varphi_j(x)|\ge 1\}| \le C\eps^\tau.
\end{equation}
\end{theorem}

The proof of this theorem combines a simplification 
of the proof of Theorem~\ref{thm:main} with a recursion.
The complexity of this recursion increases quite rapidly as 
the number of summands increases.
We conjecture that the conclusion holds without
any auxiliary hypotheses involving associated data.
Since each datum $\scriptd^*$ gives rise to its own auxiliary
hypothesis, and since the number of data associated to a given
datum $\scriptd$ grows so rapidly as the order of $\scriptd$ increases,
it would be desirable to eliminate these auxiliary hypotheses.

By a ratio of real analytic functions we will mean such a ratio, with denominator
not vanishing identically in any nonempty open set.
Such a ratio is well-defined on the complement of an analytic variety
of positive codimension. 

\subsection{Associated data and associated linear equations}
In order to complete the statement of Theorem~\ref{thm:linearmappings},
we next introduce the associated data and their associated linear equations. 
Let $n\ge 1$. 
Let $J$ be an index set of cardinality $n$.
For each $j\in J$
let $\varphi_j:\reals^2\to\reals^1$ be linear and surjective, and 
let $a_j:\tilde B\to\complex$ be a ratio of real analytic functions
that does not vanish identically in any nonempty open set.

Let $W_j$ be nonzero constant-coefficient real vector fields that satisfy $W_j(\varphi_j)\equiv 0$.
We are interested in sublevel sets
\[ S(\bff,\eps)
= \big\{x\in B: \big|\sum_{j\in J} a_j(x)(f_j\circ\varphi_j)(x)\big| <\eps\big\}. \]
Our main hypothesis
is that there are no nonzero $C^\omega$ solutions, in any nonempty open set,
of the equation
\begin{equation} \label{nequation} \sum_{j\in J} a_j(x)(f_j\circ\varphi_j)(x)=0. \end{equation}
However, the analysis developed below requires certain auxiliary hypotheses,
which are formulated in terms of certain associated equations that generalize \eqref{nequation}.
These associated equations take the general form
\begin{equation} \label{nequation2} \sum_{j=1}^m \sum_{k_j=1}^{N_j} 
a_{j,k_j}(x)(f_{j,k_j}\circ\varphi_j)(x)=0. \end{equation}
They are constructed recursively from $n$, $\{\varphi_j\}$, $\{a_j\}$ as follows.

Begin with a more general datum 
\[\scriptd = (n,\{\varphi_j: 1\le j\le n\}, \{N_j: 1\le j\le n\}, \{a_{j,k_j}: 1\le k_j\le N_j\})\]
for any $n,N_j\in\naturals$.
Construct an associated datum, as follows.

If $N_n=1$ set $m=n-1$ and 
$\varphi_j^* = \varphi_j$ for each $j\le m$. Set
$N_j^* = 2N_j$ for each $j\le m$.
For each $j \le m$ define
\begin{alignat*}{2} 
	a_{j,k_j}^* &= a_{j,k_j}/a_{n,k_{N_n}} &&\forall\,1\le k_j\le N_j
	\\ a_{j,k_j}^* &= W_n (a_{j,k_j}/a_{n,k_{N_n}})  \qquad &&\forall\, N_j<k_j\le 2N_j. 
\end{alignat*}

If $N_n>1$ set $m=n$. Set 
$\varphi_j^* = \varphi_j$ for each $j\le m$. 
Set $N_j^* = 2N_j$ for each $j\le n-1$, and $N_m^* = N_n-1$.
For each $j\in \{1,2,\dots,n-1\}$ 
and each $k_j\in\{1,2,\dots,N_j\}$,
define $a_{j,k_j}^*$ as in the preceding paragraph.
Set \[a_{n,k_n}^* = a_{n,k_n}/a_{n,k_{N_n}} \ \forall\, k_n<N_n. \]
The datum $\scriptd^* = (m,\{\varphi_j^*\}, \{N_j^*\}, \{a_{j,k_j}^*\})$
is said to be associated to $\scriptd = (n,\{\varphi_j\},\{N_j\},\{a_{j,k_j}\})$.

Applying this association rule recursively, and making association a transitive relation,
generates a finite collection of data associated to $\scriptd$.
We apply this association rule with $(1,2,\dots,n)$ replaced by an arbitrary
permutation of itself, and with the indices $k_j$ also arbitrarily
permuted for each index $j$, so that any associated datum $\scriptd^*$
for any such permutation is also associated to $\scriptd$.

\subsection{Formulation of the theorem for general data $\scriptd^*$}
Theorem~\ref{thm:linearmappings} is the special case 
in which $N_j=1$ for every index $j$ of the following result. 
We write
\[\bff = \big(f_{j,k_j}: j\in J \text{ and } 1\le k_j\le N_j\big)\]
and
\begin{equation} \label{doublesum}
S(\bff,\eps) = \big\{x\in B: \big|\sum_{j=1}^n \sum_{k_j=1}^{N_j}
a_{j,k_j}(x)\,(f_{j,k_j}\circ\varphi_j)(x)\big|<\eps\big\}.
\end{equation}

\begin{theorem} \label{thm:linearrecursive}
Let $B\subset\reals^2$ be a closed ball of positive, finite radius.
Let $\tilde B \subset\reals^2$ be an open neighborhood of $B$.
Let $n\ge 1$. 
For each $j\in\{1,2,\dots,n\}$ let $N_j\in\naturals$.
For each $j\in\{1,2,\dots,n\}$ let $\varphi_j:\reals^2\to\reals^1$
be a surjective linear mapping.
For each $j\in\{1,2,\dots,n\}$ and each $k\in\{1,2,\dots,N_j\}$ let $a_{j,k}:
\tilde B\to\complex$ be a ratio of real analytic functions.
Let $J = \{1,2,\dots,n\}$  and $\scriptd  = (J,\{\varphi_j\},\{N_j\},\{a_{j,k_j}\})$.

Suppose that for any two distinct indices $i\ne j\in J$,
$\nabla\varphi_i,\nabla\varphi_j\in \reals^2$ are linearly independent.
Suppose that none of the coefficients $a_{j,k_j}$ vanish identically in $B$.
Suppose that for any nonempty open set $U\subset\tilde B$,
and for $\scriptd$ as well as for any datum 
$\scriptd^* = \big(n^*,\{\varphi_j^*\},\{N_j^*\},\{a_{j,k_j}^*\}\big)$ 
associated to $\scriptd$, 
any real analytic solution $\bff$ of \[\sum_{j=1}^{n^*} \sum_{k_j=1}^{N_j^*}
a_{j,k}^*(x)\,(f_{j,k}\circ\varphi^*_j)(x)=0 \ \text{ in $U$} \]
satisfies $f_{j,k}\equiv 0$ in $\varphi_j(U)$ for all $(j,k)$.

Then there exist $C<\infty$ and $\tau>0$ such that for any Lebesgue measurable $\bff$
and any $\eps>0$,
\begin{equation}
	|\{x\in S(\bff,\eps): \sum_{j=1}^n \sum_{k_j=1}^{N_j} |f_{j,k}\circ\varphi_j(x)|\ge 1\}| \le C\eps^\tau.
\end{equation}
\end{theorem}

It can be shown that in the sense of jets of sufficiently high order at a point,
generic data $\scriptd$ do satisfy all of the hypotheses of the theorem.

\subsection{The case $n=1$} \label{section:n=1}


In this subsection we prove Theorem~\ref{thm:linearrecursive} for $n=1$.
With only one mapping $\varphi = \varphi_1$,
the linearizability hypothesis is tautologous wherever $\nabla\varphi\ne 0$.
The recursion on $N_1$ is now straightforward, and no auxiliary hypotheses are needed.
The statement is as follows.


\begin{proposition} \label{prop:n=1}
Let $B\subset\reals^2$ be a closed ball of finite, positive radius. 
Let $\tilde B$ be an open neighborhood of $B$.
Let $N\in\naturals$.
For each $k\in\{1,2,\cdots,N\}$ let $a_k:\tilde B\to\reals$ be real analytic.
Let $\varphi:\tilde B\to\reals$ be real analytic,  and be nonconstant on $B$.
Suppose that for any nonempty open set $U\subset\tilde B$
and any $\bff = (f_1,\dots,f_N)\in C^\omega(\varphi(U))$,
if $\sum_{k=1}^N a_k(x) (f_k\circ\varphi)(x)=0$ in $U$ then $\bff$
vanishes identically in $\varphi(U)$.
Then there exist $\tau>0$ and $C<\infty$ such that for any $\eps>0$
and any Lebesgue measurable $\bff$,
\begin{equation}
\big|\big\{
x\in B: |\sum_{k=1}^N a_k(x)(f_k\circ\varphi)(x)|<\eps
\ \text{ and } |\bff\circ\varphi(x)|\ge 1 \big\}\big| \le C\eps^\tau.
\end{equation}
\end{proposition}

\begin{proof}
Consider the case in which $\varphi(x,y)\equiv y$.
Write $\bx = (x_1,\dots,x_N)\in\reals^N$. Define
\begin{equation}
A(\bx,y) = \begin{pmatrix} a_k(x_j,y) \end{pmatrix}_{j,k=1}^N\ 
\end{equation}
for those $(\bx,y)$ with each $(x_j,y)\in\tilde B$.
Define \begin{equation} \alpha(\bx,y) = \det(A(\bx,y)).  \end{equation}
We claim that $\alpha$ does not vanish identically in $B$.
The conclusion of Proposition~\ref{prop:n=1}
follows from this claim, by reasoning developed above.  Indeed, if
\begin{equation}
S = \big\{ x\in B: |\sum_{k=1}^N a_k(x)(f_k\circ\varphi)(x)|<\eps
\ \text{ and } |\bff\circ\varphi(x)|\ge 1 \big\}
\end{equation}
then \[ S^* = \big\{ (\bx,y): (x_k,y)\in S\ \forall\,k\in\{1,2,\dots,N\} \big\}\] satisfies
\begin{equation} |S^*| \ge c|S|^C.  \end{equation}
If $(\bx,y)\in S^*$ then 
\[A(\bx,y) \begin{pmatrix} f_1(y) \\ \vdots \\ f_N(y)\end{pmatrix} = O(\eps).\]
Since $|\bff(y)|\ge 1$, it follows that $\alpha(\bx,y) = \det(A((\bx,y)))$ satisfies
\begin{equation} |\alpha(\bx,y)| = O(\eps).  \end{equation}
But if $\alpha$ does not vanish identically, then since it is a real analytic function, $\alpha$ satisfies
\begin{equation}
\big|\big\{ (\bx,y): |\alpha(\bx,y)| \le\delta \big\}\big| = O(\delta^\tau)
\end{equation}
for some exponent $\tau>0$.

To prove the claim, suppose to the contrary that $\det(A(\bx,y))$ were to vanish identically.
Then in some nonempty open set, possibly after permuting the indices $k$,
there would exist real analytic functions $c_k(\bx,y)$ that did not vanish identically
and that satisfied
\[ a_N(x_j,y) = \sum_{k=1}^{N-1} c_k(\bx,y) a_k(x_j,y)\]
for all $j\in\{1,2,\dots,N\}$	 for every $(\bx,y)\in U$.
Differentiating with respect to $x_i$ for any $i\ne j$ gives
\[  \sum_{k=1}^{N-1} \frac{\partial c_k(\bx,y)}{\partial x_i} a_k(x_j,y)\equiv 0\]
in $U$.

We proceed by induction on $N$, the case $N=1$ being the quintessence of triviality.
For the inductive step,
there are two cases. If there exists a pair of indices $i\ne j$
for which there exists $k$ such that $\frac{\partial c_k(\bx,y)}{\partial x_i}$ 
does not vanish identically in $U$,
then the conclusion follows by induction on $N$.  On the other hand,	
if $\frac{\partial c_k(\bx,y)}{\partial x_i}$ vanishes identically for any $i$
then each $c_k(\bx,y)$ is a function $c_k(y)$ of $y$ alone, and we have
\[ a_N(x,y) = \sum_{k=1}^{N-1} c_k(y) a_k(x,y).\]
The main hypothesis of the Proposition is contradicted upon
defining $f_N\equiv -1$ and $f_k(y) = c_k(y)$ for $k<N$.

This completes the discussion of the special case in which $\varphi_1(x,y) = \varphi(x,y)\equiv y$.
The nondegenerate case, in which $\nabla\varphi$ vanishes nowhere, is reducible to this special case
by a change of variables. 
The general case can be reduced to the nondegenerate case by the same reasoning that was used
in \S\ref{section:degenerate}
to reduce the general case of Theorem~\ref{thm:main} to the nondegenerate case of that theorem.
\end{proof}


\subsection{Sketch of the proof of Theorem~\ref{thm:linearrecursive} for $n>1$}

The proof is a variant of the proof of Theorem~\ref{thm:main},
with some significant differences. 
We will sketch it, indicating the points at which these differences arise.

The proof is by a double induction, with the outer induction on $n$
and the inner induction on $N_n$ with $n$ and $(N_j: j<n)$ held fixed.
Let $W_j$ be a nonzero real analytic vector field, with constant coefficients, 
satisfying $W_j(\varphi_j)\equiv 0$.
Let $N_j^* = N_j$ for $j<n$ and $=N_n-1$ for $j=n$.

By the permutation invariance of the hypotheses, we may assume that
\begin{equation}\label{f11lowerbound} |f_{1,1}(y)|\ge 1 
\ \text{ for each $y\in\varphi_1(B)$}\end{equation}
and that $f_{j,k}=O(1)$ for all $j,k$.

Let $\eps,\bff$ be given.
Let $\sS$ be the set of all $x\in S(\bff,\eps)$
such that $\sum_j \sum_{k_j} |f_{j,k}\circ\varphi_j(x)|\ge 1$.
We may assume without generality that 
$\sum_j \sum_{k_j} |f_{j,k}\circ\varphi_j(x)|\ge 1$ for every $x\in B$.
We may also assume without loss of generality that $a_{n,N_n}\equiv 1$.

For the microscale step,  let a small quantity $\sigma>0$ be given.
Consider the set $S_1$ of all $(x,s)$ with $x\in B$ and $|s|\le\eps$
such that $x\in \sS$ and $e^{sW_n}x\in \sS$.
This set satisfies $|S_1|\gtrsim \eps |\sS|^C$.
For any $(x,s)\in S_1$,
\begin{equation} \label{Linear_system1}
\sum_{j=1}^n \sum_{k_j=1}^{N_j^*} a_{j,k_j}(x)F_{j,k_j}(\varphi_j(x),sW_n(\varphi_j)) = O(\eps)
\end{equation}
where 
\begin{equation}
F_{j,k}(y,t) = f_{j,k}(y+t)-f_{j,k}(y).
\end{equation}
This is the first point at which the analysis for linear mappings $\varphi_j$
diverges from the analysis for the case of webs with nonvanishing curvature. 
Since each $W_n\varphi_j$ is a constant independent of $x$, \eqref{Linear_system1} is a family 
of inequalities for the functions
$x\mapsto F_{j,k_j}(\varphi_j(x),sW_n(\varphi_j)$;
we are dealing with a one-parameter family of sublevel set inequalities for $x\in B$
parametrized by $s$,
rather than with a single irreducible sublevel set inequality for $(x,s)\in B\times \reals^1$.
Moreover, the datum
$\big(n,\{\varphi_j\},\{N_j\},\{a_{j,k_j}\}\big)$
defining each inequality is independent of $s$,
and is the same datum with which we began, 
except that $N_n$ has been replaced by $N_n-1$,
effectively equivalent to requiring that $f_{n,N_n}$ vanishes identically.
Thus associated data do not arise in this step.

For each $s\in[-\eps,\eps]$ we may apply the induction hypothesis to the
tuple of functions $\tilde f_{j,k_j}(y) = F_{j,k_j}(y,sW_n(\varphi_j))$
indexed by $(j,k_j)$ satisfying $1\le j\le n$,
$1\le k_j\le N_j$ for $n<N$, and $1\le k_n\le N_n-1$.
We conclude that the set of all $x\in S_1$ satisfing 
\[\sum_{j=1}^n \sum_{k_j = 1}^{N_j^*} |F_{j,k_j}(\varphi_j(x),sW_n\varphi_j)|\gtrsim \eps^{1-\sigma}\]
has Lebesgue measure $\lesssim \eps^{c\sigma}$,
uniformly for all parameters $s$.

Therefore either $|S_1|\lesssim \eps^{1+c\sigma}$, in which case the proof is complete,
or there exists $S_2\subset S_1$ whose Lebesgue measure satisfies $\eps^{-1}|S_2|\gtrsim (\eps^{-1}|S_1|)^C$	
such that for every $j$ and every $k_j\le N_j^*$,
\[ |f_{j,k_j}(\varphi_j(x)+sW_n\varphi_j)-f_{j,k_j}(\varphi_j(x))| = O(\eps^{1-\sigma})
\ \ \forall\,(x,s)\in S_2.\]
By permuting the indices and repeating this argument, the same conclusion follows for $f_{n,k_N}$.
We assume henceforth that the second possibility arises.

As in the proof of Proposition~\ref{prop:microscale}, it follows that there
exist functions $f_{j,k_j}^*$ that are constant on intervals of length $\eps$,
in the same sense as in the discussion in \S\ref{section:microscale},
such that \[|f_{j,k_j}\circ\varphi_j(x)-f_{j,k_j}^*\circ\varphi_j(x)| = O(\eps^{1-\sigma})\]
for every $(j,k_j)$, for every $x\in S_3$, where $S_3\subset \sS$
satisfies $|S_3|\gtrsim |\sS|^C$.
We may, and do, replace $f_{j,k_j}$ by $f_{j,k_j}^*$ henceforth.

We next adapt the mesoscale analysis. Let a small quantity $\rho>0$ be given.
Fix $\gamma\in(\tfrac12,1)$ and set $\delta = \eps^\gamma$.
The above reasoning can be repeated at scale $\delta$. Therefore either the proof is
again complete, or there exists
$S_4\subset S_3$ satisfying $|S_4|\gtrsim |\sS|^C$ such that
\begin{equation} \label{deltascaleLip}
|f_{j,k_j}\circ\varphi_j(x) - f_{j,k_j}\circ\varphi_j(x')| = O(\delta^{1-\sigma}) 
\end{equation}
whenever $x,x'\in S_4$ satisfy $|x-x'| = O(\delta)$.

Let $\varrho>0$ be another small parameter, to be specified below.
Consider the set $S_5$ of all ordered pairs $(x',x)\in S_4\times S_4$ 
of the form $x' = e^{sW_n}x$ with $\delta^{1+\varrho} \le |s|\le \delta$. 
Define 
\[ F_{j,k_j}(y,t) = t^{-1}\big(f_{j,k_j}(y+t)-f_{j,k_j}(y)\big)\]
for $t\ne 0$
and
\[ \psi_j(x,t) = (\varphi_j(x),tW_n\varphi_j).\]
Provided that $|S_4|\gg \delta^\varrho$, as we may assume, $S_5$ satisfies
$|S_5| \gtrsim \delta |S_4|^2$. For $(x',x)\in S_5$,
\[ \sum_{j=1}^n \sum_{k_j=1}^{N_j^*}
\Big[W_na_{j,k_j}(x) f_{j,k_j}(x) + 
a_{j,k_j}(x) W_n\varphi_j (F_{j,k_j}\circ\psi_j)(x,s) \Big]
= O(\eps^{1-\gamma-\varrho}) \]
for every $(x',x) = (e^{sW_n}x,x)\in S_5$,
provided that $\sigma,\varrho$ are sufficiently small.
Note that each quantity $W_n\varphi_j$ is a constant.
We have used the bound \eqref{deltascaleLip}
to ensure that 
\[|a_{j,k_j}(x + sW_n\varphi_j)-a_{j,k_j}(x)| \cdot |F_{j,k_j}\circ\psi_j(x,s)|
= O(\delta \cdot \delta^{1-\sigma}(\delta\delta^{\varrho})^{-1})
= O(\delta^{1-\sigma-\varrho}),
\]
which is negligible relative to $\eps^{1-\gamma-\varrho} = \eps^{1-\varrho}\delta^{-1}$
for sufficiently small $\eps,\varrho,\sigma$ since
$\delta^2 = \eps^{2\gamma}$ is negligible relative to $\eps^1$ because $\gamma > \tfrac12$.

Let $S_6$ be the set of all $(x,s,s')$ such that
$(x,s)\in S_5$ and $(x,s')\in S_5$. By the Cauchy-Schwarz inequality,
$\delta^{-2} |S_6|\gtrsim (\delta^{-1}|S_5|)^2$.
For any $(x,s,s')\in S_6$,
\[ \sum_{j=1}^n \sum_{k_j=1}^{N_j^*}
a_{j,k_j}(x) 
\Big[
(W_n\varphi_j \cdot F_{j,k_j}\circ\psi_j)(x,s') 
- (W_n\varphi_j \cdot F_{j,k_j}\circ\psi_j)(x,s) 
\Big]
= O(\eps^{1-\gamma-\varrho}). \]
Because of the special form $\psi_j(x,s) = (\varphi_j(x), sW_n\varphi_j)$
with $W_n\varphi_j$ constant, 
for each fixed ordered pair $(s',s)$, this is a sublevel set inequality in the variable $x\in B\subset\reals^2$,
of the same form as that with which we began, except that $N_n$ has been replaced by $N_n^* = N_n-1$.

Invoking the induction hypothesis for each $(s,s')$,
we conclude, as in the microscale step above, 
that either $\delta^{-2}|S_6| = O(\eps^{c})$ for some $c>0$,
or there exists $S_7\subset S_6$ satisfying $\delta^{-2} |S_7|\gtrsim (\delta^{-2} |S_6|)^C$
such that 
\begin{equation} \label{longexponent}
\big| (F_{j,k_j}\circ\psi_j)(x,s') - (F_{j,k_j}\circ\psi_j)(x,s) \big| 
= O(\eps^{(1-\gamma-\varrho)(1-\sigma)})\ \forall\,j\le n\ \forall\,k_j\le N_j^*
\end{equation}
for every $(x,s',s)\in S_7$.
The same reasoning can be repeated with the roles of the indices permuted
to yield the same conclusion for the index $(n,N_n)$ as well,
on a set $S_8\subset S_7$ whose Lebesgue measure satisfies
$\delta^{-2}|S_8| \gtrsim (\delta^{-2}|S_7|)^C$.

Choose $\sigma,\varrho$ so that $(1-\gamma-\varrho)(1-\sigma) >\gamma$.
It follows from \eqref{longexponent}, as in the proof of \eqref{upon_partitioning}
in \S\ref{section:mesoscale},
that there exist $S_9\subset\sS$ of measure $|S_9|\gtrsim |\sS|^C$
and functions $f_{j,k_j}^*$
that are affine on intervals of lengths $\eps^\gamma$
and satisfy
\begin{equation*} 
|f_{j,k_j}^*(\varphi_j(x))-f_{j,k_j}(\varphi_j(x))| \le C\eps^{1-\rho} \ \forall\,x\in S
\ \forall\,j\in\{1,2,3\}.
\end{equation*}
Moreover, the derivatives $(f_{j,k_j}^*)'$ satisfy
\begin{equation*} 
|(f_{j,k_j}^*)'| \le C \eps^{-\rho}.  \end{equation*}
These constants $C$ depend on $\rho$.
We may replace each $f_{j,k_j}$ by $f_{j,k_j}^*$ henceforth.

We next carry out the macroscale step.
It is here that associated data arise.
Let $\varrho>0$ be another sufficiently small parameter, to be chosen below.
By the same reasoning as in \S\ref{section:macroscale},
it suffices to show that the set $S$ defined to be
\begin{equation*}
	S = \big\{x\in B:
	\Big|\sum_{j=1}^n \sum_{k_j=1}^{N_j}
	W_na_{j,k_j}(x)\,(f_{j,k_j}\circ\varphi_j)(x)
	+ a_{j,k_j}(x)\,W_n\varphi_j(x)\, (g_{j,k_j}\circ\varphi_j)(x)
	\Big|<\eps^\varrho
	\big\}
\end{equation*}
satisfies $|S| = O(\eps^c)$ for some $c = c(\varrho)>0$.
Here $g_{j,k_j}$ arises as the derivative of $f_{j,k_j}$,
but as in \S\ref{section:macroscale}, we regard it as an independent function,
not necessarily related to $f_{j,k_j}$ in any way.
This sum is interpreted under the convention that if $N_n=1$ then
the summation extends only over $j\in\{1,2,\dots,n-1\}$.

Rewrite this sum by setting $N_j^* = 2N_j$ for $j<n$,
and for $j<n$, $f_{j,k_j}=g_{j,k_j-N_j}$ for $N_j<k_j\le N_j^*$.
For $j=n$, $W_n\varphi_j=0$ so the contribution of $j=n$ is
$\sum_{k_n \le N_n-1} a_{n,k_n}(x)\,(f_{n,k_n}\circ\varphi_n)(x)$.
Thus the sum becomes
\[ \sum_{j=1}^n \sum_{k_j=1}^{N_j^*} a_{j,k_j}(x)\,(f_{j,k_j}\circ\varphi_j)(x)\]
if $N_n>1$, and becomes
\[ \sum_{j=1}^{n-1} \sum_{k_j=1}^{N_j^*} a_{j,k_j}(x)\,(f_{j,k_j}\circ\varphi_j)(x)\]
if $N_n=1$.
This is a sum for the datum 
\[\scriptd^* = (n,\{\varphi_j: j\le n\},\{N_j^*\},\{a_{j,k_j}: j\le n\text{ and } k_j\le N_j^*\})\]
if $N_n>1$,
and for the datum
\[\scriptd^* = (n-1,\{\varphi_j: j\le n-1\},\{N_j^*\},\{a_{j,k_j}: j\le n-1\text{ and }k_j\le N_j^*\})\]
if $N_n=1$.
In each case, the datum $\scriptd^*$ is associated to $\scriptd$.

Thus we have arrived at a sublevel set inequality, not for the given datum 
but rather for an associated datum $\scriptd^*$
under one step of the recursion defined above.
By the inductive hypothesis, a sublevel set inequality holds for this
associated datum. Since $|f_{1,1}(y)|\ge 1$ for every $y\in\varphi_1(B)$,
the inductive hypothesis ensures that $|S|$ satisfies such a upper bound,
completing the proof of Theorem~\ref{thm:linearrecursive}.
\qed

\medskip
It is the author's hope that the analysis developed in this paper provides an 
outline for a treatment of sums $\sum_{j=1}^n a_j\cdot(f_j\circ\varphi_j)$ 
with general mappings $\varphi_j$ and with an arbitrary number of summands, 
by a recursive argument,
albeit under increasingly complicated auxiliary hypotheses as $n$ increases. 
Moreover, it would be desirable to remove the auxiliary hypotheses,
which for linear mappings are encoded by associated data,
and which would take a still more complicated form for general mappings.
As of this writing, work in this direction is underway.



\section{Another form of degeneracy} \label{section:epilogue}

In \S\ref{section:intro} we stated that an excluded degenerate case of Theorem~\ref{thm:main}, 
in which two or more mappings $\varphi_j$ have gradients that are everywhere linearly dependent,
can be treated by a modification of the main argument.
No auxiliary hypothesis is needed for those cases.
In this section, we justify that statement.

There are two subcases, of which the first, 
the subcase in which the gradients of all three mappings $\varphi_j$
are everywhere linearly dependent, was treated above in \S\ref{section:n=1}.
The remaining subcase, in which $\nabla\varphi_2,\nabla\varphi_3$
are everywhere linearly dependent but $\nabla\varphi_1$ is linearly
independent of these at generic points, is a special case of 
the case $n=2$ of Theorem~\ref{thm:linearrecursive} if it is possible to change
variables so that all three mappings become linear.

In the general case of this remaining subcase,
a change of notation we have a sum of the type appearing in
\eqref{doublesum} with $n=2$, $N_1=1$, and $N_2=2$.
Assuming for simplicity that $\nabla\varphi_j$ vanish nowhere,
and that $\{\nabla\varphi_1,\nabla\varphi_2\}$ is everywhere linearly independent,
we may change variables so that $\varphi_j(x_1,x_2)\equiv x_j$.
By fixing $N_1$ values of $x_2$ and allowing $x_1$ to vary freely,
we obtain a linear system of $N_1$ approximate equations
for the $N_1$ quantities $f_{1,k_1}(x_1)$. 
This can be exploited by proceeding as in the discussion
beginning with the matrix equation \eqref{matrixeqn}
in the macroscale analysis of \S\ref{section:macroscale}.
See also \S\ref{section:n=1}.

The more general case, in which no $\nabla\varphi_j$ vanishes identically
but each may vanish on some variety of positive codimension
and $\nabla\varphi_1,\nabla\varphi_2$ are linearly independent at generic points
but not necessarily at every point, can be reduced to the 
case of the preceding paragraph by the arguments in \S\ref{section:degenerate}.

\end{document}